\documentclass[10pt,reqno]{amsart}
\usepackage{amssymb,mathrsfs,color}
\usepackage{pinlabel}

\usepackage{amsmath, amsfonts, amsthm, verbatim}
\usepackage{epstopdf, mathtools}
\mathtoolsset{showonlyrefs}
\usepackage{color}

\usepackage{empheq}
\usepackage[shortlabels]{enumitem}
\usepackage[T1]{fontenc}

\def\bR {\mathbb{R}}
\def\bS {\mathbb{S}}

\def\cK {\mathcal{K}}

\def\cZ {\mathcal{Z}}

\newcommand{\tx}[1]{\mathrm{#1}}

\newcommand{\supp}{\operatorname{supp}}

\newcommand{\eee}{\mathrm e}

\newcommand{\vd}{\mathrm{d}}
\newcommand{\udr}{\,r\vd r}

\newcommand{\uln}[1]{{\underline{ #1 }}}

\definecolor{deepgreen}{cmyk}{1,0,1,0.5}

\newcommand{\A}{\mathcal{A}}

\newcommand{\E}{\mathcal{E}}

\newcommand{\HH}{\mathcal{H}}

\newcommand{\LL}{\mathcal{L}}
\newcommand{\NN}{\mathcal{N}}
\newcommand{\M}{\mathcal{M}}

\newcommand{\QQ}{\mathcal{Q}}

\newcommand{\ZZ}{\mathcal{Z}}

\newcommand{\C}{\mathbb{C}}

\newcommand{\N}{\mathbb{N}}
\newcommand{\R}{\mathbb{R}}
\newcommand{\Sp}{\mathbb{S}}
\newcommand{\Z}{\mathbb{Z}}

\newcommand{\g}{\mathbf{g}}

\newcommand{\m}{\mathbf{m}}

\newcommand{\al}{\alpha}
\newcommand{\be}{\beta}
\newcommand{\ga}{\gamma}
\newcommand{\de}{\delta}
\newcommand{\e}{\varepsilon}
\newcommand{\fy}{\varphi}
\newcommand{\om}{\omega}
\newcommand{\la}{\lambda}
\newcommand{\te}{\theta}

\newcommand{\s}{\sigma}

\newcommand{\ka}{\kappa}

\newcommand{\si}{\varsigma}

\newcommand{\De}{\Delta}
\newcommand{\Om}{\Omega}

\newcommand{\La}{\Lambda}

\newcommand{\p}{\partial}
\newcommand{\na}{\nabla}

\makeatletter

\newcommand{\Rmnum}[1]{\expandafter\@slowromancap\romannumeral #1@}
\makeatother

\newcommand{\I}{\infty}

\newcommand{\ti}{\widetilde}

\newcommand{\U}{\underline}

\newcommand{\ang}[1]{\left\langle{#1}\right\rangle}
\newcommand{\abs}[1]{\left\lvert{#1}\right\rvert}

\newcommand{\ant}[1]{\begin{align*}\begin{split} #1 \end{split}\end{align*}}
\newcommand{\EQ}[1]{\begin{equation}\begin{split} #1 \end{split}\end{equation}}
\newcommand{\pmat}[1]{\begin{pmatrix} #1 \end{pmatrix}}

\newcommand{\Del}[1]{}

\numberwithin{equation}{section}

\newtheorem{thm}{Theorem}[section]
\newtheorem{cor}[thm]{Corollary}
\newtheorem{lem}[thm]{Lemma}
\newtheorem{prop}[thm]{Proposition}

\newtheorem{claim}[thm]{Claim}
\theoremstyle{remark}
\newtheorem{rem}[thm]{Remark}
\newtheorem{defn}[thm]{Definition}

\newcommand{\mwhere}{{\ \ \text{where} \ \ }}
\newcommand{\mand}{{\ \ \text{and} \ \  }}
\newcommand{\mor}{{\ \ \text{or} \ \ }}
\newcommand{\mif}{{\ \ \text{if} \ \ }}

\newcommand{\mas}{{\ \ \text{as} \ \ }}

\newcommand{\dd}[1]{\frac{\ud}{\ud{#1}}}

\newcommand{\rdr}{ \, r \, \mathrm{d}r}
\newcommand{\dr}{\, \mathrm{d}r}
\newcommand{\rest}{\!\!\restriction}

\newcommand{\ula}{\underline{\lambda}}
\newcommand{\umu}{\underline{\mu}}

\usepackage{tensor}

\usepackage{graphicx}
\usepackage{xcolor} 
\usepackage{tensor}
\usepackage{slashed}
\usepackage{cite}
\definecolor{green}{rgb}{0,0.8,0} 

\newcommand{\ud}{\mathrm{d}}

\newcommand{\eps}{\epsilon}

\newcommand{\bfd}{{\bf d}}

\newcommand{\calK}{\mathcal K}

\begin{document}

\title[Two-bubble dynamics for wave maps]{Two-bubble dynamics for threshold solutions \\ to the  wave maps equation}
\author{Jacek Jendrej}
\author{Andrew Lawrie}

\begin{abstract}
We consider the energy-critical wave maps equation $\bR^{1+2} \to \bS^2$
in the equivariant case, with equivariance degree $k \geq 2$.
It is known that initial data of energy $< 8\pi k$ and topological degree zero
leads to global solutions that scatter in both time directions.
We consider the threshold case of energy $8 \pi k $.
We prove that the solution is defined for all time
and either scatters in both time directions,
or converges to a superposition of two harmonic maps in one time direction
and scatters in the other time direction. In the latter case, we describe
the asymptotic behavior of the scales of the two harmonic maps.

The proof combines the classical concentration-compactness techniques of Kenig-Merle
with a modulation analysis of interactions of two harmonic maps in the absence of excess radiation.

%
\end{abstract}

\maketitle

\section{Introduction}

This paper concerns energy critical wave maps
$
\Psi: (\R^{1+2}_{t, x}, \m) \to (\mathcal{M}, \g),
$
where  $\m$ is the Minkowski metric and $\M$ is a Riemannian manifold with a metric $\g$. Wave maps arise in the physics literature as examples of nonlinear $\sigma$-models. A particularly  interesting case is when the target manifold admits nontrivial finite energy stationary wave maps, or harmonic maps, as these give simple examples of topological (albeit unstable) solitons.  Mathematically, wave maps simultaneously generalize the classical harmonic maps equation to Lorenztian domains as well as the free wave equation to manifold-valued maps.

Viewing $(\M, \g)$ as an isometrically embedded sub-manifold of Euclidean space $(\R^N, \ang{ \cdot, \cdot}_{\R^N})$,  a~\emph{wave map} is defined as a formal critical point of the Lagrangian action
 \EQ{
\LL(\Psi)  =  \frac{1}{2} \int_{\R^{1+2}}  \m^{\al \be} \ang{ \p_\al \Psi , \,  \p_\be \Psi }_{\R^N} \,  \ud x\, \ud t.
}
The Euler-Lagrange equations are given by 
\EQ{
\Box \Psi \perp T_{\Psi} \M,
}
which can be rewritten as 
\EQ{ \label{eq:wm}
\Box \Psi = \mathcal{S}(\Psi)(\p \Psi, \p \Psi),
}
where $\mathcal{S}$ denotes the second fundamental form of the embedding $(\M, \g) \hookrightarrow (\R^N, \ang{ \cdot, \cdot}).$
The conserved energy is given by 
\EQ{ \label{eq:en1} 
\E(\Psi, \p_t \Psi)(t)  =  \frac{1}{2} \int_{\R^2}  \abs{\p_t \Psi(t)}^2  +  \abs{\na \Psi(t)}^2 \, \ud x  = \textrm{constant}.
}
Smooth finite energy initial data  for~\eqref{eq:wm} consist of a pair  $\vec \Psi(0) = (\Psi_0, \Psi_1)$, where 
\EQ{ \label{eq:data} 
\Psi_0(x) \in \M   \subset \R^N , \quad \Psi_1(x)  \in  T_{\Psi_0(x)}\M, \quad  \forall  \, \, x \in \R^2.
} 
We assume here that  we can find a fixed vector $\Psi_\infty \in \M$ so that 
\EQ{\label{eq:infpoint}
\Psi_0(x) \to \Psi_\infty \mas \abs{x} \to \infty. 
}
Wave maps on $\R^{1+2}_{t, x}$ are called~\emph{energy critical} because the conserved energy and the equation are invariant under the same scaling: If $\vec \Psi(t)$ solves~\eqref{eq:wm} then so does 
\EQ{ \label{eq:scale}
\vec \Psi_{\la}(t, x) := (\Psi_\la(t, x), \p_t \Psi_\la(t, x)) := \Big(\Psi(t/\la,  x/\la), \, \frac{1}{\la} \p_t \Psi(  t/ \la , x/ \la)\Big)
}
and it also holds that $\E(\vec \Psi_\la) = \E(\vec\Psi)$.

The geometry of the target manifold, and in particular the existence of non-constant finite energy harmonic maps $\Psi: \R^2 \to \M$,   plays a crucial role in determining the possible dynamics of solutions to the wave maps equation. Here we'll focus on a special case when the target manifold is the $2$-sphere, $\M = \Sp^2 \subset \R^3$ with the round metric $\g$. One advantage is that here the harmonic maps are explicit: by a classical theorem of Eells and Wood~\cite{EW} they are either holomorphic or anti-holomorphic with respect to the complex structure on $\Sp^2$, and  by~\eqref{eq:infpoint} with the removable singularity theorem~\cite{SU} they can thus be identified with the rational functions $\rho: \C_{\infty} \to \C_{\infty}$.  It follows that each harmonic map $\R^2 \to \Sp^2$ has a topological degree given by the degree of the corresponding rational map. 

In fact, the condition~\eqref{eq:infpoint} allows us to assign a topological degree to each smooth finite energy data. Given data  $(\Psi_0, \Psi_1)$, we can  identify $\Psi_0$ with a map $\ti \Psi_0: \Sp^2 \to \Sp^2$ by assigning the point at $\infty$ to the vector $\Psi_\infty:= \lim_{\abs{x} \to \infty} \Psi(x)$. 
Abusing notation slightly by writing $\ti \Psi_0 = \Psi_0$, the degree of the map $\Psi_0$ is defined by
\EQ{
\deg(\Psi_0) :=   \frac{1}{ \textrm{Area}(\Sp^2)} \int_{\Sp^2}  \Psi^*_0( \om) = \frac{1}{ 4\pi} \int_{\R^2}   \Psi_0^*( \om)\in \Z 
}
where $\om$ is the area element  of $\Sp^2 \subset \R^3$. The degree $\deg(\Psi_0)$ is preserved by the smooth wave map flow on its maximal interval of existence $I_{\max}$,  that is, 
\EQ{
\deg( \Psi_0) = \deg(\Psi(t)) \quad \forall t \in I_{\max}. 
}
Importantly, any harmonic map  $\QQ_k$ of degree $k$ minimizes the energy amongst all degree $k$ wave maps, and in fact 
\EQ{
\E(\QQ_k) = 4 \pi \abs{ \deg(\QQ_k) }= 4 \pi \abs{k}. 
}


\subsection{$k$-equivariant wave maps} 
To simplify the analysis we'll take advantage of a symmetry reduction and study a restricted class of maps  $\Psi$ satisfying the equivariance relation $\Psi \circ \rho^k = \rho^k \circ \Psi$ for all rotations $\rho \in SO(2)$.
We consider a subclass of such maps known as $k$-equivariant, or $k$-corotational, which correspond to equivariant maps that in local coordinates take the form
\ant{
\Psi(t, r, \theta) = (\psi(t, r), k \te ) \hookrightarrow ( \sin \psi \cos k\theta , \sin \psi \sin  k \theta, \cos \psi) \in \Sp^2 \subset \R^3,
}
where $\psi$ is the colatitude measured from the north pole of the sphere and the metric on $\Sp^2$ is given by $ds^2 = d \psi^2+ \sin^2 \psi\,  d \om^2$.  
The Euler-Lagrange equations~\eqref{eq:wm} reduce to an equation for $\psi$ and we are led to the Cauchy problem:
\EQ{ \label{eq:wmk}
\Bigg\{\begin{aligned}
&\psi_{tt} -  \psi_{rr}  - \frac{1}{r} \psi_r + k^2  \frac{\sin 2 \psi}{2r^2} = 0, \\ 
&(\psi(0), \partial_t \psi(0)) = (\psi_0, \psi_1).
\end{aligned}
}
We'll often use the notation $\vec \psi(t)$ to denote the pair 
\EQ{
\vec \psi(t, r):= ( \psi(t, r), \psi_t(t, r))
}
and we remark that the scaling~\eqref{eq:scale} can be expressed as follows:   If $\vec \psi(t, r)$ is a solution to~\eqref{eq:wmk} then so is 
 \EQ{
 \vec \psi_\la(t, r)  = \Big(\psi( t/ \la, r/ \la), \frac{1}{\la} \psi_t(t/ \la, r/ \la)\Big)
 }
for each fixed $\la>0$. 
%

 The conserved energy from~\eqref{eq:en1} takes the form 
\EQ{ \label{eq:en} 
 \E( \vec \psi(t) ) = 2 \pi \frac{1}{2}\int_0^\I  \left( (\p_t \psi (t, r))^2  +  ( \p_r \psi(t, r))^2 + k^2\frac{\sin^2\psi(t, r)}{r^2} \, \right)\,\rdr  .
 }
 From the above it's clear that any $k$-equivariant data  $\vec \psi(0,r)$  of finite energy  must satisfy $\lim_{r \to 0}\psi(0, r) = m\pi$ and  $\lim_{r \to \infty}\psi(0,\infty)=n\pi$  for some $m,n \in \Z$. Since the smooth wave map flow depends continuously on the initial data these integers are fixed over any time interval  $t\in I$ on which the solution is defined. This splits the energy space into disjoint classes according to this topological condition and it is natural to  consider the Cauchy  problem \eqref{eq:wmk} within a fixed class 
\EQ{\label{eq:Hnm}
 \HH_{m \pi, n \pi} := \{ (\psi_0, \psi_1) \mid \E(\psi_0, \psi_1) < \infty \mand \lim_{r \to 0}\psi_0(r) =m\pi, \, \lim_{r \to \infty}\psi_0(r) = n\pi\}. 
}
We can restrict to $\HH_{0, n\pi}$ and we'll denote these by $\HH_{n \pi} := \HH_{0, n \pi}$. We also define $\HH= \bigcup_{n\in \Z} \HH_{n \pi}$ to be the full energy space.

The equivariant reduction introduces a good deal of rigidity into the problem, but still allows us access to the family of harmonic maps. Indeed the degree $k$ harmonic map $\QQ_k$ corresponding to $z\mapsto z^k$ as a holomorphic map $\C_\infty \to \C_{\infty}$, can be expressed uniquely (up to scaling) as the $\abs{k}$-equivariant map 
\EQ{
\QQ_k(r, \te) = (Q_k(r),  k \te) \hookrightarrow ( \sin Q_k \cos k\theta , \sin Q_k \sin  k \theta, \cos Q_k) \in \Sp^2 \subset \R^3,
}
where $Q_k$ is the explicit finite energy stationary solution to~\eqref{eq:wmk} given by 
 \EQ{
 Q_k(r) := 2 \arctan r^k.
 }
 Note that $Q_k(r)$ satisfies 
 \EQ{\label{eq:degQ}
 Q_k( 0) = 0, \quad  \lim_{r \to \infty} Q_k(r) =  \pi. 
 }
 We often write $\vec Q_k := (Q_k, 0)$.  We see that  $\E(\vec Q_k) = 4\pi k$,  which is minimal amongst all $k$-equivariant maps in the energy class $\HH_\pi$; see Section~\ref{s:hm} for a direct argument. 

 Here we consider $k$-equivariant maps $\vec \psi = (\psi_0, \psi_1)$ in the class $\HH_0$, i.e., that satisfy
 \EQ{ \label{eq:deg0} 
 \lim_{r \to 0} \psi_0( r) = 0 \mand   \lim_{r \to \infty} \psi_0(r) =  0,  
 }
 so that $\psi_0$ is the polar angle of a finite energy map $\Psi_0$ into $\Sp^2$ with $\deg(\Psi_0) = 0$.  
 Before stating our mains results let us first motivate this restriction with a brief summary of recent developments.

\subsection{Threshold Theorems and Bubbling} 
The energy critical wave maps equation~\eqref{eq:wm} has been extensively studied over the past several decades; \cite{CTZcpam, CTZduke, STZ92, STZ94, KlaMac93, KlaMac94, KlaMac95, KlaMac97, KlaSel97, KlaSel02, Tat98, Tao1, Tao2, Tat01, Kri04}. In recent years the focus has centered on understanding the nonlinear dynamics of solutions with large energy.  At the end of the last decade, the following remarkable~\emph{sub-threshold conjecture} was established~\cite{ST1, ST2, KS, Tao37}: Every wave map with energy less than that of the first nontrivial harmonic map is globally regular on $\R^{1+2}$ and scatters to a constant map.  The role of the least energy harmonic map in the statement of the sub-threshold conjecture is based on fundamental work of Struwe~\cite{Struwe}, who showed that the smooth equivariant wave map flow can only develop a singularity by concentrating energy at the tip of a light cone by bubbling off at least one non-trivial finite energy harmonic map.  In breakthrough works, Krieger, Schlag, Tataru~\cite{KST}, Rodnianski, Sterbenz~\cite{RS}, and Rapha\"el, Rodnianski~\cite{RR}, constructed examples solutions of such blow-up by bubbling, with the latter two works yielding a stable blow-up regime; see also the recent stability analysis of Krieger~\cite{Krieger17} for type-II blow ups solutions to the energy critical NLW, which suggests that the solutions from~\cite{KST} should also exhibit stability properties. 

The starting point for the present work  is the following natural question: Can one give a satisfactory description of the possible dynamics for arbitrary initial data? In dispersive models such as~\eqref{eq:wm} this is typically referred to as the \emph{soliton resolution conjecture}, which states roughly that any smooth solution asymptotically decouples into weakly interacting (possibly concentrating) solitons plus free radiation. The wave maps equation~\eqref{eq:wm} with $\NN = \Sp^2$  is an intriguing model in which to study this question: all stationary solutions (the harmonic maps) are known explicitly, the conserved topological degree of the solution introduces additional rigidity, and the equivariant reduction~\eqref{eq:wmk} greatly simplifies certain aspects of the analysis without destroying the essential mechanisms of truly nonlinear behavior, e.g., solitons, blow-up. There has been exciting recent progress in this direction for the general equation \eqref{eq:wm}, see~\cite{Gri, DJKM16}. Here we focus on the equivariant model~\eqref{eq:wmk} where more is known. 

Our analysis is motivated by several results proved in the last few years; we will in fact use some of them explicitly. First, note that, by continuity, $\lim_{r\to 0}\psi(t, r)$ and $\lim_{r\to\infty}\psi(t, r)$ are independent of $t$. Hence scattering to a constant map is only possible if $\lim_{r\to 0}\psi_0(r) = \lim_{r\to\infty}\psi_0(r)$. We can assume without loss of generality
that both these limits equal $0$, i.e. the initial data $(\psi_0, \psi_1)$ is in~$\HH_0$. For such maps, the following refined threshold theorem was proved in~\cite{CKLS1}. 

\begin{thm}[$2\E(\vec Q)$ Threshold Theorem] \emph{\cite[Theorem $1.1$]{CKLS1}\label{t:2EQ}}  For any smooth initial data $\vec\psi(0) \in \HH_0$ with 
\EQ{
\E(\vec\psi(0)) < 2\E(\vec Q_k) = 8 \pi k, 
}
there exists a unique global evolution $\vec \psi \in C^0(\R; \HH_0)$. Moreover, $\vec\psi(t)$ scatters to zero in  both time directions, i.e.,  there exist solutions $\vec \fy_L^\pm$ to the linearized equation~\eqref{eq:2dlin} such that 
\EQ{ \label{scat}
 \vec{\psi}(t) = \vec \fy_L^\pm(t) + o_{\HH_0}(1) \mas t \to  \pm\infty. 
}
\end{thm}
The analogous result for the full model without symmetries was obtained by the second author and Oh in~\cite{LO1}, as a consequence of the bubbling analysis in~\cite{ST2}. 
The heuristic reasoning behind the threshold $2 \E(\vec Q)$ is as follows. The topological degree counts (with orientation) the number of times a map `wraps around' $\Sp^2$. If a harmonic map of degree $k$ bubbles off from a wave map $\vec \psi(t)$,  then, in order for $\vec \psi(t)$ to satisfy $\deg(\psi) = 0$, it must `unwrap' precisely $k$ times away from the bubble. The minimum energy cost for wrapping and unwrapping is $4 \pi k$, which is also the energy of $Q_k$. The total energy cost is at least $8 \pi k   = 2\E(\vec Q_k)$. 

Similar intuition motivated the works~\cite{CKLS1, CKLS2}, which established soliton resolution for $1$-equivariant maps with energies that only allow for one concentrating bubble, namely for data in $\HH_{\pi}$ with energy below $3\E(Q)$. These works showed that for any such solution there exists a regular map $\vec \fy \in \HH_0$ (free radiation if the solution is global) and a continuous dynamical scale  $\la(t)   \in [0, \infty)$ such that 
\EQ{ \label{eq:sr1} 
 \vec \psi(t) = \vec Q_{\la(t)}  +  \vec \fy(t) + o_{\HH_0}(1) \mas t \to T_+. 
} 
 Cote~\cite{Cote15} and later Jia, Kenig~\cite{JK} extended the theory to handle arbitrary energies, the latter work in all equivariance classes. 
It was shown that in  this case the decomposition~\eqref{eq:sr1} holds with possibly many concentrating harmonic maps, but only along at least one~\emph{sequence of times}  $t_n \to T_+$. 
The proofs of~\cite{CKLS1, CKLS2, Cote15, JK} rely heavily on concentration compactness techniques and were all inspired by the remarkable series of papers by Duyckaerts, Kenig, and Merle~\cite{DKM1, DKM2, DKM3, DKM4} on the focusing quintic nonlinear wave equation in $3$ space dimensions. We'll discuss these latter works more below; see Remark~\ref{r:dkm}. 

\begin{thm}[Sequential Decomposition]\emph{ \cite{Cote15, JK}}\label{t:cjk} Let $\vec \psi(t)\in \HH_{\ell \pi}$ be a smooth solution to~\eqref{eq:wmk} on $[0, T_+)$. Then there exists a sequence of times $t_n \to T_+$, an integer $J \in \N$, a regular map $\vec \fy \in \HH_0$, sequences of scales $\la_{n, j}$ and signs $\iota_j \in \{-1, 1\}$ for $j \in \{1,  \dots, J\}$,  so that 
\EQ{ \label{eq:seq} 
 \vec \psi(t_n)  = \sum_{j =1}^J \iota_j \vec Q_{\la_{n,j}}  + \vec \fy(t_n) + o_{\HH_0}(1) \mas n \to \infty
}
In the case of finite time blow-up at least one scale $\la_{n, 1} \to 0$ as $n \to \infty$ and $\vec \fy(t) \to \vec \fy(1)$ is a finite energy  map  with $\E(\vec \fy(1)) = \E(\vec \psi) - J\E(\vec Q)$. In the case of a global solution,  $\vec \fy(t)$ can be taken to be a solution to the linear wave equation~\eqref{eq:2dlin} and 
 signs $\iota_j$ are required to match up so that $$\lim_{r \to \infty} \vec \psi(0, r) = \ell \pi  =  \lim_{r \to \infty} \sum_{j =1}^J \iota_j \vec Q_{\la_{n,j}}(r).$$ 
\end{thm}  
\begin{rem}
A decomposition into bubbles for a sequence of times for the full non-equivariant model was obtained by Grinis \cite{Gri} up to an error that vanishes in a weaker Besov-type norm. 
Duyckaerts, Jia, Kenig and Merle \cite{DJKM16} proved that for energies slightly above
$\E(\vec Q)$ a one-bubble decomposition holds for continuous time.
The same authors obtained in \cite{DJKM1} a sequential decomposition into bubbles
in the case of the focusing energy critical power-type nonlinear wave equation (NLW).  
\end{rem}

Theorem~\ref{t:cjk} raises two natural questions: 
\begin{itemize} 
\item Are there any solutions to~\eqref{eq:wmk} with $J \ge 2$ in~\eqref{eq:seq}, i.e., are there any solutions that form more than one bubble? 
\item And, if so, does the decomposition~\eqref{eq:seq} hold continuously in time, i.e., does soliton resolution hold for~\eqref{eq:wmk}? 
\end{itemize}

In view of Theorems~\ref{t:2EQ} and~\ref{t:cjk} it is natural to ask both questions at the minimal possible energy level where multiple bubble dynamics can occur, namely for solutions $\vec\psi(t) \in \HH_0$ having \emph{threshold energy}, that is such that
 \EQ{\label{eq:energy}
 \E( \vec \psi) =  2 \E(\vec Q).
  }
In \cite{JJ-AJM} the first author obtained  an affirmative answer to the first question, proving the following result. 
\begin{thm}\emph{\cite[Theorem $2$]{JJ-AJM}}
  \label{thm:deux-bulles-wmap}
Let $k > 2$. There exists a solution $\vec\psi: (-\infty, T_0] \to \HH_0$ of \eqref{eq:wmk}
such that
  \begin{equation}
    \label{eq:mainthm-wmap}
    \lim_{t\to -\infty}\big\|\vec\psi(t) - \big({-}\vec Q + \vec Q_{\gamma_k |t|^{-\frac{2}{k-2}}}\big)\big\|_{\HH_0} = 0,
  \end{equation}
  where $\gamma_k > 0$ is an explicit constant depending on $k$. \qed
\end{thm}
\begin{rem}
Similar solutions could be obtained for $k = 2$ by the same method. 
\end{rem}
 \subsection{Main result}
In this paper, we address the problem of \emph{classification} of solutions
at threshold energy level, in the spirit of the works of Duyckaerts and Merle~\cite{DM, DM-NLS}.
The major difficulty in the analysis is that in our case the threshold solutions
contain two bubbles, which leads to significantly more complicated dynamics. 

 Let $\vec \psi(t) : (T_-, T_+) \to \HH_0$ be a solution to~\eqref{eq:wmk} with $\E(\vec \psi) = 2 \E(\vec Q)$. We will say that $\vec\psi(t)$ is a \emph{two-bubble in
 the forward time direction} if there exist $\iota \in \{1, -1\}$
 and continuous functions $\lambda(t), \mu(t) > 0$ such that
 \EQ{
\lim_{t \to T_+} \| (\psi(t) - \iota(Q_{\la(t)} - Q_{\mu(t)}), \psi_t(t))\|_{\HH_0} = 0, \quad \lambda(t) \ll \mu(t)\text{ as }t \to T_+.
}
The notion of a \emph{two-bubble in the backward time direction} is defined similarly. We prove the following result. 
 \begin{thm}[Main Theorem] \label{t:main} 
Fix any equivariance class $k \ge 2$.  Let $\vec\psi(t) :(T_-, T_+) \to \HH_0$ be a solution to~\eqref{eq:wmk} such that
\EQ{
\E(\vec \psi) = 2 \E(\vec Q) = 8\pi k.
}
Then $T_- = -\infty$, $T_+ = +\infty$ and one the following alternatives holds:
 \begin{itemize}[leftmargin=0.5cm]
\item $\vec \psi(t)$ scatters in both time directions,
\item $\vec\psi(t)$ scatters in one time direction and is a two-bubble in the other time direction with the scales of the bubbles $\lambda(t), \mu(t)$ satisfying
\EQ{
\mu(t) \to \mu_0 \in (0, +\infty),\qquad \lambda(t) \to 0.
}
\end{itemize} 
 \end{thm}  
 \begin{rem}
 As a by-product of the proof, we will determine the rate of decay of $\lambda(t)$ in the two-bubble case. Suppose a two-bubble solution forms as $t \to \infty$: If $k \ge 3$ there exists a constant $C_k > 0$ such that  $\frac{1}{C_k}\mu_0 t^{-\frac{2}{k-2}} \leq \la(t) \leq C_k\mu_0 t^{-\frac{2}{k-2}}$
 for $t$ large enough, see \eqref{eq:lambda-asym-k}.
 In the case $k=2$ there exists a constant $C > 0$ such that we have $\exp( -C  t) \le \la(t) \le \exp(- t/C)$
 for $t$ large enough, see \eqref{eq:lambda-asym-2}. 
 \end{rem}

 \begin{rem}
   In particular, the two-bubble solutions from Theorem~\ref{thm:deux-bulles-wmap}
   scatter in forward time, which provides an example of an orbit
   connecting different types of dynamical behavior for positive and negative times.
\end{rem}
\begin{rem}
   Non-existence of solutions which form a pure two-bubble in both time directions
   is reminiscent of the work of Martel and Merle \cite{MM11, MM11-2}.
   This seems to be a typical feature of models which are not completely integrable.

   One of the main points of our paper is an analysis of what we could call a \emph{collision
   of bubbles} in the simplest possible case of threshold energy. 
 \end{rem}
\begin{rem}
Recall that in \cite{DM} a \emph{complete classification} at the threshold energy
was obtained. It is tempting to believe that the solutions from Theorem~\ref{thm:deux-bulles-wmap} should play a similar role as the solution $W^-$ from \cite{DM},
in which case they would be unique non-dispersive solutions up to rescaling. This remains an open question.
\end{rem}

\begin{rem}
We conjecture that for $k = 1$ a similar result holds,
but in the two-bubble case $\lambda(t) \to 0$ in finite time.
The slower decay of $Q$ would be a source of additional
technical difficulties in Section 3, but the general scheme could be applied
without major changes.
\end{rem}
\begin{rem}
Our method establishes the exact analog of Theorem~\ref{t:main} in the case of the equivariant Yang-Mills equation, by making the usual analogy between equivariant Yang-Mills and $k=2$-equivariant wave maps, see for example~\cite[Appendix]{CKLS1} for the analog of the Threshold Theorem~\ref{t:2EQ} and~\cite{JJ-AJM} for the analog of Theorem~\ref{thm:deux-bulles-wmap}. There the harmonic map $Q$ is replaced by the first instanton, the notion of topological degree is replaced by the second Chern number, and the threshold energy is exactly twice the energy of the first instanton. 
\end{rem}

\begin{rem}\label{r:dkm} 
The full soliton resolution conjecture was established for the radial solutions of focusing energy critical NLW by Duyckaerts, Kenig, and Merle in the landmark work~\cite{DKM4}. This result is the only known case of a complete \emph{continuous-in-time} classification for a model
that is not completely integrable.   The proof relies on a particularly strong  form of the ``channels of energy" method introduced by the same authors. However, proving channel of energy estimates in other settings  is a delicate issue, and the strong form of these estimates used in~\cite{DKM4} is known to fail for the linear wave equation in even dimensions, see~\cite{CKS}.

Aside from~\cite{DKM4}, Theorem~\ref{t:main} is the only other classification result for a dispersive equation that holds for continuous times in the presence of more than one non-trivial elliptic profile. Upgrading sequential decompositions such as Theorem~\ref{t:cjk} or the one in~\cite{DJKM1} to hold for continuous times is regarded as a major open problem. 

%
 \end{rem}

 \subsection{Structure of the proof}
 Inspired by the work of Duyckaerts and Merle~\cite{DM},
 we merge the concentration-compactness techniques
 with a careful analysis of the \emph{modulation equations} governing the evolution
 of the scales $\lambda(t)$ and $\mu(t)$.  As mentioned above,
 the main difference with respect to \cite{DM} consists in the fact that
 our threshold solutions contain two bubbles, one of which is concentrating,
 whereas in \cite{DM} the modulation happens essentially around one stationary bubble. Thus our analysis requires substantially new technique. 
 Our proof can be summarized as follows.
 
 \noindent
 \textbf{Step 1.} If the solution does not scatter, then,
 by a special case of Theorem~\ref{t:cjk},
 it approaches a two-bubble configuration for a sequence of times.
 
 \noindent
 \textbf{Step 2.} We divide the time axis into regions where the solution is close
 to a two-bubble configuration, which we can call the \emph{bad intervals} $[a_m, b_m]$,
  and regions where it is not, which are the \emph{good intervals} $[b_{m}, a_{m+1}]$.
 
 \noindent
 \textbf{Step 3.} On a bad interval $[a_m, b_m]$, we decompose the solution as follows:
 \EQ{
   \vec\psi(t) = \big({-}Q_{\mu(t)} + Q_{\lambda(t)} + g(t), \partial_t \psi(t)\big).
 }
 In order to specify the values of $\lambda(t)$ and $\mu(t)$, we use suitable
 orthogonality conditions, see Lemma~\ref{l:modeq}.
 For technical reasons, we introduce a parameter $\zeta(t)$ such that $|\zeta(t) - \lambda(t)| \ll \lambda(t)$.
 We consider $c_m \in (a_m, b_m)$
 where the quantity $\zeta(t) / \mu(t)$ attains its global minimum on $[a_m, b_m]$
 (we make sure that the minimum is not attained at one of the endpoints).
 
 The orthogonality conditions yield modulation equations for the evolution
 of $\zeta(t)$ and $\mu(t)$.
 From these equations we can deduce crucial information about the behavior
 of $\mu(t)$ and $\zeta(t)$ for $t \geq c_m$ and $t \leq c_m$.
 We refer to the beginning of Section~\ref{s:mod} for a short description of the method.
 The main conclusion can be intuitively phrased as follows:
 $\mu(t)$ does not change much on a bad interval, whereas $\zeta(t)$ grows in a controlled
 way both for $t \geq c_m$ and $t \leq c_m$. The decisive point is that the bad interval
 $[a_m, b_m]$ can be long if $\zeta(c_m)/\mu(c_m)$ is small, but
 \begin{equation}
 		\int_{a_m}^{b_m} \Big(\frac{\zeta(t)}{\mu(t)}\Big)^\frac k2 \ud t \leq C_k, \qquad C_k\text{ depending only on }k. \label{eq:bad-error-contr}
 \end{equation}
 Note that the only information about the solution which is used in this process
 is the fact that $\E(\vec\psi(t)) = 2\E(\vec Q)$, $\vec\psi(c_m)$ is close to a two-bubble configuration
 and
 \EQ{
 \dd t\Big\vert_{t = c_m} \big(\zeta(t)/\mu(t)\big) = 0.
 }
 
 \noindent
 \textbf{Step 4.} Using concentration-compactness arguments and Theorem~\ref{t:2EQ}
 we obtain that the solution has the \emph{compactness property} on the union of the good intervals.
 Now the idea is to run a convexity argument based on a monotonicity formula between
 two times where $\vec\psi(t)$ is close to a two-bubble. It is as this stage that we reach a contradiction --  if the solution has exited a neighborhood of two-bubble configurations during the interim, the total cost in terms of time derivative is too great to allow it to return. This is a type of \emph{no-return} result and one can draw parallels here to the ignition and ejection lemmas from the work of Krieger, Nakanishi, Schlag~\cite{KNS13, KNS15} concerning near ground-state dynamics for the energy critical NLW.  
 
There can potentially be many good and bad intervals between the two times where $\vec \psi(t)$ is close to a two-bubble. 
 It is well-known that one needs to use a cut-off in the monotonicity formula,
 which introduces an error in the estimates. On the good intervals, this error
 is controlled thanks to the compactness property. On the bad intervals,
 the bound \eqref{eq:bad-error-contr} comes into play. More precisely,
 we obtain that the error on a bad interval is absorbed by positive terms
 obtained on intervals preceding and following the bad interval.

 \noindent
 \textbf{Step 5.}
 Once the convergence to a two-bubble for continuous time is proved,
 we deduce easily from the modulation equations that the solution is global and $\mu(t) \to \mu_0 \in (0, +\infty)$.
 Scattering on at least one side follows easily from the previous analysis.
 Namely, if the solution is non-scattering in both time directions,
 then the time axis is divided into two bad regions near $\pm \infty$ and one good interval
 in between. We reach a contradiction by a similar (but simpler) argument as in~Step~4.

 \subsection{Acknowledgements}
J. Jendrej was supported by the ERC grant 291214 BLOWDISOL and by the NSF grant DMS-1463746.
This work was completed during his postdoc at the University of Chicago.
A. Lawrie was supported by NSF grant DMS-1700127. We would like to thank Rapha\"el C\^ote for many helpful discussions. And lastly, we would like to thank the anonymous referees for their careful reading of an earlier version of the manuscript and for suggesting substantial improvements.

 \section{Preliminaries and technical lemmas}

 In this section we establish a few preliminary facts about solutions to~\eqref{eq:wmk} that will be required in our analysis. We first aggregate here some notation. 
 
 \subsection{Notation} 
Given a radial function $f: \R^d \to \R$ we'll abuse notation and simply write $f = f(r)$, where $r = \abs{x}$. We'll also drop the factor $2\pi$ in our notation for the $L^2$ pairing of radial functions on $\R^2$ 
\EQ{
\ang{f \mid g}  := \frac{1}{2\pi}\ang{ f \mid g}_{L^2(\R^2)} = \int_0^\infty f(r) g(r) \, r \ud r
}
Recall the definition of the space $\HH_0$:
\EQ{
\HH_0:= \{ (\psi_0, \psi_1) \mid \E(\psi_0, \psi_1)< \infty, \quad \lim_{r \to 0} \psi_0(r) = \lim_{r \to \infty} \psi_0(r) = 0 \}
}
We define a norm $H$ by 
 \EQ{
 \| \psi_0 \|_{H}^2 := \int_0^\infty  \left(( \p_r \psi_0(r))^2 +  k^2  \frac{(\psi_0(r))^2}{r^2} \right)  r \ud r 
  }
 and for pairs $\vec \psi = (\psi_0, \psi_1) \in \HH_0$ we write  
 \EQ{
 \| \vec \psi \|_{\HH_0} := \| (\psi_0, \psi_1)\|_{H \times L^2}.
 }
The change of variables $r \mapsto e^x$ gives us an identification between the radial functions $H(\R^2)$ and $H^1(\R)$, i.e., $\psi_0(r) \in H \Leftrightarrow \psi_0(e^x) \in H^1(\R)$. In particular this means that 
\EQ{
 \| \psi_0 \|_{L^{\infty}} \le C \| \psi_0 \|_{H}
}
Scaling invariance plays a key role in our analysis.
Given a radial function $\phi: \R^2  \to \R$ we denote the $\dot H^1$ and $L^2$  re-scalings as follows 
\EQ{ \label{eq:scaledef} 
\phi_\la(r) = \phi(r/ \la), \quad
\phi_{\ula}(r)  = \frac{1}{\la} \phi(r/ \la)
}
The corresponding infinitesimal generators  are given by 
\EQ{ \label{eq:LaLa0} 
&\La \phi := -\frac{\partial}{\partial \lambda}\bigg|_{\lambda = 1} \phi_\la = r \p_r \phi  \quad (\dot H^1_{\textrm{rad}}(\R^2) \,  \textrm{scaling}) \\
& \La_0 \phi := -\frac{\partial}{\partial \lambda}\bigg|_{\lambda = 1} \phi_{\ula} = (1 + r \p_r ) \phi  \quad (L^2_{\textrm{rad}}(\R^2) \,  \textrm{scaling})
}

\subsection{Review of the Cauchy theory %
} \label{s:2-4}
 For initial data $(\varphi_0, \varphi_1)$ in the class $\HH_0$ the formulation of the Cauchy problem~\eqref{eq:wmk} can be modified to take into account the strong repulsive potential term in the nonlinearity: 
\ant{
\frac{k^2\sin(2\phi)}{2r^2}  = \frac{ k^2}{r^2}\phi + \frac{ k^2}{2r^2} (\sin(2\phi) - 2\phi)  =  \frac{k^2}{r^2} \phi + \frac{O(\phi^3)}{r^2} 
} 
The presence of the potential $\frac{k^2}{r^2}$ indicates that the linear wave equation, 
\EQ{ \label{eq:2dlin} 
( \p_t^2 - \Delta_{\R^2} + \frac{k^2}{r^2} ) \psi = 0
,
}
of \eqref{eq:wmk} has more dispersion than the $2d$ wave equation. In fact, it has the same dispersion as the free wave equation in dimension $d = 2k+2$
as can be seen from the following change of variables:  given a radial function $\phi \in H$, define $v(r)$ by  $\phi(r) = r^k v(r)$. Then 
\EQ{ \label{eq:free} 
\frac{1}{r^k}(- \Delta_{\R^2} + \frac{k^2}{r^2}) \phi = -\De_{\R^{2k+2}} v, \quad \| \phi \|_{H}  = \|v \|_{\dot{H}^1(\R^{2k+2})}. 
}
Thus one way of studying solutions $\vec \psi(t) \in \HH_0$ of  Cauchy problem~\eqref{eq:wmk} is to define $\vec v(t)  = (r^{-k} \psi(t), r^{-k} \psi_t(t)) \in \dot H^1 \times L^2 (\R^{2k+2})$ and analyze the equivalent Cauchy problem for the radial nonlinear wave equation in $\R^{1+ (2k+2)}_{t, x}$ satisfied by $\vec v(t)$. Unfortunately, this route leads to unpleasant technicalities when $k >2$ (spatial dimension $=2k+2>6$) due to the high dimension and the particular structure of the nonlinearity.

There is a simpler approach that allows us to treat the scattering theory for the Cauchy problem~\eqref{eq:wmk} for all equivariance classes $k \ge 1$ in a unified fashion.
The idea is to make use of some, but not all, of the extra dispersion in $-\De_{\R^2} + k^2/r^2$.
Indeed, given a solution $\vec \psi(t)$ to~\eqref{eq:wmk} we define  $u$ by $ru = \psi$  and obtain the following Cauchy problem for $u$. 
\EQ{\label{eq:4d}
&u_{tt} - u_{rr} -\frac{3}{r} u_r  +\frac{k^2 -1}{r^2}  u  =  k^2\frac{2r u - \sin(2r u)}{2r^3} =: Z(ru) u^3 \\
&\vec u(0)= (u_0, u_1). 
}
where the function $Z$ defined above is a clearly smooth, bounded, even function. The linear part of~\eqref{eq:4d} is the radial wave equation in $\R^{1+4}$ with a \emph{repulsive} inverse square potential,  namely 
\EQ{\label{eq:4dlin}
&v_{tt} - v_{rr} -\frac{3}{r} v_r  + \frac{k^2-1}{r^2}v=0.
}
For each $k\ge 1$, define the norm $H_k$ for radially symmetric functions $v$ on $\R^4$ by 
\EQ{
\| v \|_{H_k(\R^4)}^2:= \int_0^\infty \left[(\p_r v)^2 + \frac{(k^2-1)}{r^2} v^2 \right] \, r^3 \, \ud r
}
Solutions to~\eqref{eq:4dlin} conserve the $H_k$ norms.  By Hardy's inequality we have 
\EQ{ \label{eq:hardy} 
\|v \|_{H_k(\R^4)} \simeq \| v \|_{\dot{H}^1(\R^4)}
} 
The mapping, 
 \EQ{
 H_k \times L^2 (\R^4)  \ni (u_0, u_1) \mapsto  ( \psi_0, \psi_1):= (ru_0, ru_1) \in H \times L^2 (\R^2)
 }
 satisfies 
 \EQ{ \label{eq:2-4}
\|(u_0, u_1) \|_{\dot{H}^1 \times L^2(\R^4)} \simeq  \| (u_0, u_1) \|_{H_k \times L^2 (\R^4)}  = \| (\psi_0, \psi_1) \|_{H \times L^2 (\R^2)}
 }
Thus we can conclude that the Cauchy problem for~\eqref{eq:4d} with initial data in $\dot{H}^1 \times L^2(\R^4)$ is equivalent to the Cauchy problem for~\eqref{eq:wmk} for initial data $(\psi_0, \psi_1) \in \HH_0$, allowing us to give a scattering criterion for solutions $\vec \psi(t) \in \HH_0$ to~\eqref{eq:wmk}. 

\begin{lem}\label{l:scattering} Let $\vec \psi(0) = (\psi_0, \psi_1) \in \HH_0$. Then there exists a unique solution $\vec \psi(t) \in \HH_0$ to~\eqref{eq:wmk} defined on a maximal interval of existence $I_{\max}(\vec \psi) :=(-T_-(\vec \psi), T_+(\vec \psi))$ with the following properties: Define 
\EQ{
\vec u(t, r) = (r^{-1}\psi(t, r), r^{-1} \psi_t(t, r)  )\in \dot{H}^1\times L^2(\R^4)
}
Then for any compact time interval $J \Subset I_{\max}$ we have 
\EQ{
\| u\|_{L^3_t(J; L^6_x(\R^4))} \le C(J)< \infty
} 
In addition, if 
\EQ{
\| u\|_{L^3_t([0, T_+(\vec \psi)); L^6_x(\R^4))} < \infty
}
then $T_+ = \infty$ and $\vec \psi(t)$ scatters $t \to \infty$, i.e., there exists a solution $\vec\phi_L(t) \in \HH_0$ to~\eqref{eq:2dlin} so that 
\EQ{
\| \vec \psi(t) - \vec \phi_L(t) \|_{\HH_0} \to 0 \mas t \to \infty.
}
Conversely, any solution $\vec \psi(t)$ that scatters as $t \to \infty$ satisfies $$\|  \psi/ r\|_{L^3_tL^6_x([0, \infty) \times \R^4)} < \infty.$$ 
\end{lem}

The proof of Lemma~\ref{l:scattering} is standard consequence of Strichartz estimates for~\eqref{eq:4dlin} and the equivalence of the Cauchy problems~\eqref{eq:wmk} and~\eqref{eq:4d}.  In this case, we need Strichartz estimates for the radial wave equation in $\R^{1+4}$ with a repulsive inverse square potential. For these we can cite the more general results of Planchon, Stalker, and Tahvildar-Zadeh~\cite{PST03b}; see also~\cite{BPST03, BPST04} which cover the non-radial case.  

\begin{lem}[Strichartz estimates]\emph{ \cite[Corollary 3.9]{PST03b}} \label{l:strich} Fix $k \ge 1$ and let $\vec v(t)$ be a radial solution to the linear equation 
\EQ{
v_{tt}  - v_{rr} - \frac{3}{r} v_r + \frac{k^2 -1}{r^2} v  = F(t, r), \quad \vec v(0) = (v_0, v_1) \in \dot{H}^1 \times L^2 (\R^4)
}
Then, for any time interval $I \subset \R$ we have 
\EQ{ \label{eq:strich} 
\| v \|_{L^{3}_t L^6_x(I \times \R^4)} + \sup_{t \in I}\| \vec v(t) \|_{\dot{H}^1 \times L^2(\R^4)} \lesssim   \| \vec v(0) \|_{\dot{H}^1 \times L^2(\R^4)} + \| F \|_{L^1_t, L^2_x(I \times \R^4)}
} 
where the implicit constant above is independent of $I$. 
\end{lem}  

 We'll also explicitly require the following nonlinear perturbation lemma from~\cite{KM08}; see also~\cite[Lemma 2.18]{CKLS1}. 
 \begin{lem}[Perturbation Lemma]\emph{\cite[Theorem $2.20$]{KM08}~\cite[Lemma 2.18]{CKLS1} \label{l:pert}} There are continuous functions $\e_0,  C_0: (0, \infty) \to (0, \infty)$ such that the following holds: Let $I\subset \R$ be an open interval, (possibly unbounded), $\psi, \fy \in C^0(I; H) \cap C^1(I; L^2) $ radial functions satisfying for some $A>0$ 
\begin{align*} 
&\|\vec \psi\|_{L^{\I}(I; H \times L^2(\R^2))}+ \|\vec\fy\|_{L^{\I}(I; H \times L^2(\R^2))}+ \|\fy/r\|_{L^3_t(I; L^6_x(\R^4))} \le A\\
&\|\textrm{eq}(\psi/r)\|_{L^1_t(I; L^2_x(\R^4))}+\|\textrm{eq}(\fy/r)\|_{L^1_t(I; L^2_x(\R^4))} + \|w_0\|_{L^3_t(I; L^6_x)} \le \e \le \e_0(A)
\end{align*} 
where $\textrm{eq}(\psi/r):= (\Box_{\R^4} + \frac{k^2-1}{r^2}) (\psi/r) +(\psi/r)^3Z(\psi)$ in the sense of distributions, and $\vec w_0(t):= S(t-t_0)(\vec \psi-\vec \fy)(t_0)$ with $t_0 \in I$ arbitrary, but fixed and $S$ denoting the linear wave evolution operator in $\R^{1+4}$ (i.e., the propagator for~\eqref{eq:2dlin}). Then,
\begin{align*} 
\|\vec \psi -\vec \fy - \vec w_0\|_{L^{\I}_t(I; H \times L^2(\R^2))} + \|\frac{1}{r}(\psi-\fy)\|_{L^3_t(I; L^6_x(\R^4))} \le C_0(A) \e
\end{align*} 
In particular, $\|\psi/r\|_{L^3_t(I; L^6_x(\R^4))} < \I$. 
\end{lem} 

%
%

 \subsection{Concentration Compactness} \label{s:cc} 
Another consequence of~\eqref{eq:2-4} and Lemma~\ref{l:strich} is that we can translate the concentration compactness theory of Bahouri and G\'erard to solutions to~\eqref{eq:2dlin} and~\eqref{eq:wmk}. 
We begin by stating the linear profile decompositions in the $4d$ setting for solutions to~\eqref{eq:4dlin}. 


\begin{lem}[Linear $4d$ profile decomposition] \emph{\cite[Main Theorem]{BG}} \label{l:bg} Let $k \ge1$ be fixed.  Consider a sequence  $\vec u_n = (u_{n, 0}, u_{n, 1}) \in {H}_k \times L^2( \R^4)$ which is bounded in the sense that  $ \|\vec u_n\|_{H_k \times L^2(\R^4)} \lesssim 1$. Then, up to passing to a subsequence,  there exists a sequence of solutions to~\eqref{eq:4dlin}, $\vec V_L^j  \in H_k \times L^2(\R^4)$,  sequences of times $\{t_{n,j}\}\subset \R$, and sequences of scales $\{\la_{n, j}\}\subset (0, \infty)$, and  $\vec w_n^k$ defined by 
\EQ{ \label{eq:4dlinprof}  
\vec u_n(r) = \sum_{j=1}^k (\frac{1}{\la_{n, j}}V_L^j\left( \frac{-t_{n, j}}{ \la_{n, j}}, \frac{r}{ \la_{n, j}}\right),  \frac{1}{(\la_{n, j})^2}\p_t V_L^j\left( \frac{-t_{n, j}}{ \la_{n, j}}, \frac{r}{ \la_{n, j}}\right)) + (w_{n, 0}^k, w_{n, 1}^k)(r)\\
}
so that the following statements hold: Let $w_{n, L}^k(t)$ denote the linear evolution of the data $\vec w_n^k$, i.e., solutions to~\eqref{eq:4dlin}. Then, for any $j \le k$, 
\EQ{ \label{eq:w-weak}
(\la_{n}^j w_{n, L}^k(  t_{n, j},  \la_{n, j}\cdot) , \la_{n, j}^2 w_n^k(  t_{n, j},  \la_{n, j}\cdot)) \rightharpoonup 0\, \,  \textrm{weakly in} \, \,  H_k \times L^2(\R^4). 
}
In addition, for any $j\neq k$ we have
\EQ{ \label{eq:oscales}
\frac{\la_{n, j}}{\la_{n, k}} + \frac{\la_{n, k}}{\la_{n, j}} + \frac{\abs{t_{n, j}-t_{n, k}}}{\la_{n, j}} + \frac{\abs{t_{n, j}-t_{n, k}}}{\la_{n, k}} \to \infty \quad \textrm{as} \quad n \to \infty.
}
Moreover, the errors $\vec w_n^k$ vanish asymptotically in the sense that 
\EQ{ \label{eq:w-in-strich}
\limsup_{n \to \infty} \left\| w_{n, L}^k\right\|_{L^{\infty}_tL^4_x \cap L^3_tL^6_x( \R \times \R^4)}  \to 0 \quad \textrm{as} \quad k \to \infty.
}
Finally, we have almost-orthogonality of the $H_k \times L^2$ norms of the decomposition: 
\EQ{ \label{eq:free-en-ort}
\|\vec u_n\|_{H_k \times L^2}^2 = \sum_{1 \le j \le k} \| \vec V_L^j( - t_{n, j}/ \la_{n, j}) \|_{H_k \times L^2}^2  +  \|\vec w_n^k\|_{H_k\times L^2}^2 + o_n(1) \mas n \to \infty
}

\end{lem}

\begin{rem}
The difference between Lemma~\ref{l:bg} and the main theorem in~\cite{BG} is that here we have phrased matters in terms of solutions to the $4d$ linear wave equation with a repulsive inverse square potential~\eqref{eq:4dlin} (which conserve the $H_k \times L^2$ norm), as opposed to the free wave equation in $4d$ with data in $\dot H^1 \times L^2$.  However, a proof identical to the one in~\cite{BG} can be used to establish Lemma~\ref{l:bg}. Alternatively, one can establish Lemma~\ref{l:bg} by conjugating~\eqref{eq:4dlin} to the free wave equation in dimension $d = 2k+2$ via the map $v(r) \mapsto r^{-k+1} v(r) = u$. This map induces an isometry $H_k(\R^4) \to \dot{H}^1(\R^{2k+2})$; see~\eqref{eq:free}. Then the usual Bahouri-G\'erard profile decomposition in $d=2k+2$ induces a profile decomposition in $H_k$.  Once must check that the errors $w_{n, L}^J$ can be made to vanish as in~\eqref{eq:w-in-strich}, but this follows by combining the vanishing of $r^{-k+1}w_{n, L}^J$ in appropriate dim $=2k+2$ Strichartz norms with the Strauss estimate, 
\EQ{
 \sup_{t \in \R, r >0} \abs{r w_{n, L}^J(t, r) }\lesssim \|w_{n, L}^J \|_{L^{\infty}_t H_k(\R^4)}, 
 }
 and interpolation. 

\end{rem} 

A direct consequence of Lemma~\ref{l:bg} and~\eqref{eq:2-4}  with the identifications 
\EQ{  \label{ident}
& \psi_n(r) := r u_n(r), \quad  \ga_n^J(r)  := r w_n^J, \\
& \fy^j_L( -t_{n, j}/ \la_{n, j}, r/ \la_{n, j}) := \frac{r}{\la_{n, j}}V^j_L( -t_{n, j}/ \la_{n, j}, r/ \la_{n, j}),
}
is the following profile decomposition for bounded sequences $\vec \psi_n \in \HH_0$.

\begin{cor}[Linear profile decomposition]\label{c:bg} Consider a sequence $\vec \psi_n \in \HH_0$ that is uniformly bounded in $\HH_0$. Then, up to passing to a subsequence,  there exists a sequence of solutions $\vec \fy^j_L \in \HH_0$ to~\eqref{eq:2dlin},  sequences of times $\{t_{n, j}\}\subset \R$,  sequences of scales $\{\la_{n, j}\}\subset (0, \infty)$, and errors $\vec \ga_n^J$ defined by 
\EQ{
\vec \psi_n = \sum_{j=1}^J (\fy^j_L( -t_{n, j}/ \la_{n, j}, \cdot/ \la_{n, j}), \frac{1}{\la_{n, j}}\p_t\fy^j_L( -t_{n, j}/ \la_{n, j}, \cdot/ \la_{n, j})) + (\ga_{n, 0}^J, \ga_{n, 1}^J)
}
so that the following statements hold: Let $\ga_{n, L}^J(t) \in \HH_0$ denote the linear evolution, (i.e., solution to \eqref{eq:2dlin}) of the data $\vec \ga_n^J \in \HH_0$. Then,  for any $j \le \ell$, 
\EQ{ \label{eq:ga-weak} 
(\ga_n^\ell(  t_{n, j},  \la_{n, j}\cdot) , \la_{n, j} \ga_n^\ell(  t_{n, j},  \la_{n, j}\cdot)) \rightharpoonup 0\quad \textrm{weakly in} \quad \HH_0. 
}
In addition, for any $j\neq \ell$ we have
\EQ{ \label{eq:po}
\frac{\la_{n, j}}{\la_{n, \ell}} + \frac{\la_{n, \ell}}{\la_{n, j}} + \frac{\abs{t_{n, j}-t_{n, \ell}}}{\la_{n, j}} + \frac{\abs{t_{n, j}-t_{n, \ell}}}{\la_{n, \ell}} \to \infty \quad \textrm{as} \quad n \to \infty.
}
Moreover, the errors $\vec \ga_n^J$ vanish asymptotically in the sense that  
\EQ{
\limsup_{n \to \infty} \left\|\frac{1}{r} \ga_{n, L}^J\right\|_{L^{\infty}_tL^4_x \cap L^3_tL^6_x( \R \times \R^4)}  \to 0 \quad \textrm{as} \quad J \to \infty.
}
Finally, we have almost-orthogonality of the $\HH_0$ norms of the decomposition: 
\EQ{ \label{ort H} 
\|\vec \psi_n\|_{\HH_0}^2 = \sum_{1 \le j \le J} \| \vec \fy_L^j( - t_{n, j}/ \la_{n, j}) \|_{\HH_0}^2  +  \|\vec \ga_n^J\|_{\HH_0}^2 + o_n(1) \mas n \to \infty
} 
\end{cor}
Our applications of the concentration-compactness techniques developed by Kenig and Merle in~\cite{KM06, KM08} requires a ``Pythagorean decomposition'' of the nonlinear energy proved in~\cite{CKLS1}. 
\begin{lem}\emph{\cite[Lemma $2.16$]{CKLS1} }\label{l:enorth}
Let  $\vec \psi_n \in \HH_0$ be a bounded sequence with a linear profile decomposition as in Corollary~\ref{c:bg}. Then the following Pythagorean decomposition holds for the nonlinear energy of the sequence: 
\EQ{\label{eq:enorth} 
\E(\vec \psi_n) = \sum_{j=1}^J \E(\vec \fy_L^j(-t_{n, j}/ \la_{n, j})) + \E(\vec \ga_n^J) + o_n(1)  \mas n \to \infty.
} 
 \end{lem}
 We will also require the following nonlinear profile decomposition analogous to~\cite[Proposition 2.17]{CKLS1}, or~\cite[Proposition 2.8]{DKM1}. We'll use the following notation: Given a linear profile decomposition as in Corollary~\ref{c:bg} with profiles $\{\fy^j_L\}$ and parameters $\{t_{n,j}\}, \{ \la_{n, j}\}$ we denote by $\{\fy^j\}$ the nonlinear profile associated to $\{\fy^j_L(-t_n^j/ \la_n^j), \dot{\fy}^j_L(-t_n^j/ \la_n^j)\}$, i.e., the unique solution to \eqref{eq:wmk} so that for all $-t_n^j/ \la_n^j \in I_{\max}(\vec \fy^j)$ we have
\EQ{
\lim_{n \to \infty} \| \vec \fy^j(-t_{n, j}/\la_{n, j}) - \vec \fy_L^j(-t_{n, j}/ \la_{n,j})\|_{\HH_0} =0.
}
The existence of a non-linear profile is immediate from the local well-posedness theory for~\eqref{eq:wmk} in the case that $-t_{n, j}/ \la_{n,j} \to \tau_{\infty, j} \in \R$. If $-t_{n, j}/ \la_{n, j}  \to \pm \infty$ then the existence of the nonlinear profile follows from the existence of wave operators for~\eqref{eq:wmk} and it follows that the maximal forward/backward time of existence  $T_{\pm}(\vec \fy) = \infty$. Each of these facts are now standard consequences of the Strichartz estimates in Lemma~\ref{l:strich}. 

\begin{lem}[Nonlinear Profile Decomposition]\emph{\cite[Proposition 2.17]{CKLS1}\cite[Proposition 2.8]{DKM1}\cite{BG}}\label{p:nlprof} Let $\vec \psi_n(0) \in \HH_0$ be a uniformly bounded sequence with a profile decomposition as in Corollary~\ref{c:bg}. Assume that the nonlinear profile $\fy^j$ associated to the linear profile $\fy^j_L$ has maximal forward time of existence $T_+(\vec \fy^j)$. Let $s_n \in (0, \infty)$ be any sequence such that for all $j$ and for all $n$, 
\EQ{
\frac{s_n - t_{n, j}}{\la_{n, j}} < T_+(\vec \fy^j), \quad \limsup_{n \to \infty} \|\fy^j/ r\|_{L^3_t([-\frac{t_{n, j}}{\la_{n, j}}, \frac{s_n - t_{n, j}}{\la_{n, j}}); L^6_x(\R^4))} <\infty.
}
Let $\vec \psi_n(t)$ denote the solution of \eqref{eq:wmk} with initial data $\vec \psi_n(0)$.  Then for $n$ large enough  $\vec \psi_n(t)$ exists on the interval $s \in (0, s_n)$ and satisfies, 
\EQ{
\limsup_{n \to \infty} \|\psi_n/r\|_{L^3_t([0, s_n); L^6_x(\R^4))} < \infty.
}
Moreover, the following non-linear profile decomposition holds for all $s \in [0, s_n)$,  
\EQ{
\vec \psi_n(s, r) = \sum_{j=1}^J \left(\fy^j\left( \frac{s- t_{n, j}}{\la_{n, j}}, \frac{r}{\la_{n, j}}\right), \frac{1}{\la_n^j}\p_t \fy^j\left(\frac{s-t_{n, j}}{\la_{n, j}}, \frac{r}{\la_{n,j}}\right) \right) + \vec \ga_{n, L}^{J}(s, r)+ \vec \te_n^J(s, r)
}
with  $\ga_{n, L}^J(t)$ as in \eqref{eq:w-in-strich} and 
\EQ{\label{eq:nlerror}
\lim_{J \to \infty} \limsup_{n\to \infty} \left( \|\te_n^J/r\|_{L^3_t([0, s_n); L^6_x(\R^4))} + \|\vec \te_n^J\|_{L^{\infty}_t ([0, s_n); \HH_0)} \right) =0.
}
The analogous statement holds for sequences $s_n\in (-\infty, 0)$. 
\end{lem}
 
Our main application of these ideas can be summarized in the following compactness lemma.

  \begin{lem} \label{l:1profile} Let $\vec \psi(t) \in \HH_0$ be a solution to~\eqref{eq:wmk} defined on its forward maximal interval of existence $[0, T_+(\vec \psi))$. Suppose that $\E(\vec \psi) = 2 \E(Q_k)$ and that $\vec \psi(t)$ does not scatter as $t \to T_+(\vec \psi)$. Then the following holds: Suppose that  $t_n \to T_+ $ is any sequence of times such that 
  \EQ{ \label{eq:Hbounded} 
  \sup_n \| \vec\psi(t_n)\|_{\HH_0} \le C < \infty
  }
  Then, up to passing to a subsequence of the $t_n$, there exists scales $\nu_n>0$ and a nonzero $\vec \fy \in \HH_0$ such that 
  \EQ{
   \vec \psi(t_n)_{\frac{1}{\nu_n}} \to \vec \fy \in \HH_0
  }
  strongly in $\HH_0$ and $ \E(\vec \fy ) = 2\E(Q_k)$. Moreover, the nonlinear evolution $\vec \fy(s)$ of the data $\vec \fy(0) = \vec \fy$   is non-scattering in both forwards and backwards time. 
 \end{lem} 
 
 \begin{rem}
One consequence the main result,  Theorem~\ref{t:main}, is that the hypothesis of Lemma~\ref{l:1profile} are not satisfied by any solution! However, we'll use Lemma~\ref{l:1profile} in the context of a contradiction argument in the proof of Proposition~\ref{p:psi_t} in Section~\ref{s:dynamics}. Since the proof of the lemma uses only standard facts about profile decompositions, the local Cauchy theory, and the Threshold Theorem~\ref{t:2EQ} we include it here in Section~\ref{s:cc}. 
 \end{rem} 
 
\begin{proof}[Proof of Lemma~\ref{l:1profile}]
By~\eqref{eq:Hbounded} we can perform a linear profile decomposition as in Corollary~\ref{c:bg} on $\vec \psi(t_n)$. 

First we observe that there can only be one non-zero profile $\vec \fy = \vec\fy^1$ and that the errors $\vec \gamma_{n, L}^J$ must vanish strongly $\HH_0$ as $n \to \infty$. Indeed, if there were two non-trivial profiles, or if the errors did not vanish strongly in $\HH_0$, then~\eqref{eq:enorth} along with our hypothesis that $\E(\vec \psi) = 2 \E(Q)$  imply that every nonzero profile must have energy $< 2\E(\vec Q)$. Thus each non-zero nonlinear profile scatters in both directions by the Threshold Theorem~\ref{t:2EQ}. A now standard argument based on the nonlinear Perturbation Lemma~\ref{l:pert}, and the orthogonality of the parameters in~\eqref{eq:po} implies that $\vec \psi(t)$ must also scatter in forward time, a contradiction.

Thus,  there exists times $t_{n, 1}$ and scales $\nu_{n, 1}$, and a single limiting profile $\vec \fy = (\fy_0, \fy_1)$ so that 
\EQ{ \label{eq:psim0} 
( \psi( t_n + \nu_n t_{n, 1}, \nu_{n, 1} \cdot ),   \nu_{n, 1}\psi_t( t_n + \nu_n t_{n, 1}, \nu_{n, 1} \cdot  ) \to  \vec \fy \in \HH_0 \mas n \to \infty
}
Next we claim both $-\frac{t_{n, 1}}{\nu_{n_1}} \to \pm \infty$ are impossible and we can therefore assume without loss of generality that $t_{n, 1} = 0$ for all $n$. 
To see this, first assume first   $-\frac{t_{n, 1}}{\nu_{n, 1}} \to +\infty$. Then $\vec \fy$ scatters in forward time and we can deduce that
 \EQ{
\| \fy/ r\|_{L^3_t L^6_x([ -\frac{t_{n, 1}}{ \nu_{n, 1}}, \infty) \times \R^4)} \to 0 \mas n \to \infty
} 
by the definition of the nonlinear profile. But then the Nonlinear Perturbation Lemma~\ref{l:pert} implies that $ \vec \psi(t)$ must also scatter in forward time, which contradicts our initial assumptions on $\vec \psi(t)$. 
 
Now assume that $-t_{n, 1}/ \nu_{n, 1} \to -\infty$. Then the nonlinear profile $\vec \fy(s)$ scatters in backwards time, and the Nonlinear Perturbation Lemma~\ref{l:pert}  implies that 
\EQ{
 \| \psi/r\|_{L^3_t L^6_x(([0,t_n]) \times \R^4)} = \| \varphi / r \|_{L^3_t L^6_x([ \frac{-t_n-t_{n, 1}}{ \nu_{n, 1}}, -\frac{t_{n, 1}}{ \nu_{n, 1}}] \times \R^4)} + o_n(1) \to 0, 
 }
 a contradiction. Thus we can assume that $ t_{n, 1}  \equiv 0$
and we simply write $\nu_{n, 1} = \nu_n$. At this point we've shown that up to passing to a subsequence in $t_n$ we have 
\EQ{
\vec \psi(t_n)_{\frac{1}{\nu_{n}}} \to \fy \in \HH_0, \quad \E(\vec \fy) = 2\E(Q)
}
We can now run a nearly identical argument to show that  nonlinear evolution $\vec \fy(s) \in \HH_0$ can not scatter in either time direction. To see this, first suppose that $\vec \fy$ scatters in forward time. Then,  the Nonlinear Perturbation Lemma~\ref{l:pert} implies that $\vec \psi(t)$ must also scatter as $t \to \infty$. If $\vec \fy(s)$ were to scatter as $s \to -\infty$, then, 
\EQ{
\| \psi/r \|_{L^3_t L^6_x([0,t_n] \times \R^4)} = \|  \vec \fy/ r \|_{L^3_t L^6_x([ (-t_n)/ \nu_{n, 1}, 0]) \times \R^4)} + o_n(1) \le C< \infty.
}
Letting limit $n \to \infty$, we see that $\| \psi /r\|_{L^3_t L^6_x([0, T_+(\vec \psi)) \times \R^4)} \le C$,  which again means that $\vec \psi(t)$ scatters in forward time, a contradiction. Hence $\vec \fy(s)$ does not scatter in either direction. 
\end{proof}

 
 \subsection{The harmonic maps $Q = Q_k$} \label{s:hm} 
 
 We record a few properties about the unique (up to scaling) $k$-equivariant harmonic map $Q = Q_k(r) = 2 \arctan r^k$ and some consequences of the fact that each $Q_k$ minimizes the energy functional amongst all  $k$-equivariant  maps. 
  
First observe  that $Q$ satisfies 
 \EQ{
 r \p_r Q(r) = k \sin Q(r), \quad Q(0) = 0, \quad Q(\infty) = \pi.
 }
Recall that  $\HH_\pi$ the set of all finite energy $k$-equivariant maps, with $\phi_0(0) = 0$ and  $\phi_0(\infty) = \pi$, 
  \EQ{
 \HH_\pi:= \{ (\phi_0, \phi_1) \mid \E(\vec \phi)< \infty, \quad \phi_0(0) = 0, \quad \lim_{r \to \infty} \fy_0(r) = \pi\}
 }
 The fact that $Q$ minimizes the energy in $\HH_\pi$ can be easily seen from the following Bogomol'nyi factorization of the energy:
\EQ{ \label{eq:bog}
  \E( \fy_0, \fy_1)  &=   \pi \| \fy_1 \|_{L^2}^2 +  \pi \int_0^{\infty} \left(\p_r \fy_0  - k\frac{\sin(\fy_0)}{r}\right)^2 \, r\, dr +  2\pi k \int_0^{\infty} \sin(\fy_0) \p_r\fy_0 \, \dr\\
&= \pi \| \fy_1 \|_{L^2}^2  + \pi \int_0^{\infty} \left(\p_r \fy_0  - k\frac{\sin(\fy_0)}{r}\right)^2 \, \rdr + 2 \pi k \int_{\fy_0(0)}^{\fy_0(\infty)} \sin(\rho)  \, \ud\rho \\
&  =  \pi \| \fy_1 \|_{L^2}^2  +  \pi \int_0^{\infty} \left(\p_r \fy_0  - k\frac{\sin(\fy_0)}{r}\right)^2 \, \rdr  + 4\pi k
}
Hence, 
\begin{align} \label{var char}
\E(\fy_0, \fy_1) \ge\pi  \| \fy_1 \|_{L^2}^2 + 4 \pi k =  \pi \| \fy_1 \|_{L^2}^2 + \E(Q_k, 0)
\end{align}
where the inequality in the last line above is in fact strict if $\fy_0 \neq Q_k$.

We define a functional on maps $\Phi: \R^2 \to \Sp^2$ of finite energy. Let $\om_{\Sp^2}$ denote the volume form on $\Sp^2$.  Given $\Om  \subset \R^2$ set 
\EQ{
G( \Phi, \Om):=  \int_{\Phi(\Om)} \om_{\Sp^2} =  \int_{\Om} \Phi^*(\om_{\Sp^2}) 
}
where $\Phi^*(\om_{\Sp^2})$ denotes the pull-back. Given $k$-equivariant $\Phi$ with polar angle $\phi$, this reduces to 
\EQ{
G(\phi_0(r)):= 2 \pi  \int_{\phi_0(0)}^{\phi_0(r)}k \abs{\sin \rho} \, \ud\rho
}
Observe that for any $(\phi, 0)$ with $\E(\vec \phi)< \infty$  and for any $R \in[0, \infty)$ we have 
\EQ{\label{b R}
\abs{G(\phi_0(R))}    =\abs{2 \pi \int_{\phi_0(0)}^{\phi_0(R)} k\abs{\sin \rho} \, \ud\rho} 
= \abs{2 \pi \int_0^R \abs{k\sin(\phi_0(r))} \p_r\phi_0(r) \, dr}
 \le \E_0^R(\phi_0, 0)
}
where for any $0 \le a <b$ we define the localized energy $\E^a_b$ by
\EQ{
\E_a^b( \phi_0, \phi_1) := 2 \pi \int_a^b \frac{1}{2} \left(\phi_1^2 + ( \p_r \phi_0)^2 + k^2 \frac{\sin^2 \phi_0}{ r^2} \right) \, r \, \ud r. 
}
The same argument shows that 
\EQ{ \label{a R} 
\abs{G(\phi_0(R))} \le  \E_R^{\infty}(\phi_0, 0)
}
On the other hand, since  $Q$ satisfies $r \p_r Q(r) = k\sin(Q)$,  for any $0\le a\le  b < \infty$ we see that 
\EQ{\label{eq:GQ}
G(Q(b))-G(Q(a)) = 2 \pi \int_{a}^{b} \abs{\sin(Q(r))} Q_r(r) \, dr =  \E_a^b(Q, 0)
}
Letting $a \to 0$ and $b \to \infty$ we recover the fact that $\E(Q, 0) = G(\pi) = 4 \pi k $.

We recall the following variational characterization of $Q$ in  $\HH_\pi$ from~\cite{Cote}, which amounts to the coercivity of the energy functional near $Q$. 

\begin{lem}\emph{\cite[Proposition $2.3$]{Cote}}\label{l:decQ} There exists a function $c: [0, \infty) \to [0, \infty)$ such that $c(\al) \to 0$ as  $\al \to 0$ and such that the following holds: Let  $( \phi_0, 0) \in \HH_\pi$. Suppose 
\ant{
\alpha :=\E( \phi_0, 0)- \E(Q, 0) \ge 0
}
Then for $\la>0$ defined so that  $\E_0^{\la}(\phi_0, 0) = \E_0^1(Q)= \E(Q)/2$, we have 
\ant{
\|\phi_0 - Q_\la\|_{H} \le c(\alpha)
}
Moreover,  $\al = 0$ if and only if $\phi_0(r)  = Q(r/ \la)$ for some $\la>0$. 
\end{lem}


 \subsection{Threshold solutions near a $2$-bubble configuration} 
 The goal of this section is to relate the proximity of a map $\phi \in \HH_0$ to a $2$-bubble configuration to the size of the $\HH_0$-norm of $\vec \phi$. With this in mind we make the following definition. 

 \begin{defn}[Proximity to a $2$-bubble]\label{d:ddef} Given a map $ \vec\phi =  (\phi_0, \phi_1) \in \HH_0$ we define its proximity $\bfd(\vec \phi)$ to a pure $2$-bubble by 
\EQ{ \label{eq:ddef} 
\bfd(\vec \phi) := \inf_{\la, \mu >0, \iota \in \{+1, -1\}}  \Big(  \| (\phi_0 - \iota (Q_\la - Q_\mu), \phi_1) \|_{\HH_0}^2 + \left( \la/\mu \right)^k \Big)
}
\end{defn}

The proof of Theorem~\ref{t:main} will require a few technical lemmas concerning $\bfd$. We'll state the lemmas first and postpone the proofs until the end of this section. 
 
 \begin{lem} \label{l:d-size}
 Suppose that $\vec \phi = (\phi_0, \phi_1) \in \HH_0$ is $k$-equivariant and satisfies, 
\EQ{
&\E( \vec \phi)  = 2 \E(\vec Q_k). 
 }
  Then for each $\be>0$ there exists a there exists a constant $C(\be)>0$ such that 
 \EQ{ \label{eq:d-big} 
  \bfd(\vec \phi) \ge \be \Longrightarrow  \|(\phi_0, \phi_1) \|_{\HH_0} \le C(\be) 
   } 
Conversely, for each $A>0$ we can find $\al  = \al(A)$ such that 
\EQ{\label{eq:d-small}
 \bfd(\vec \phi) \le \al(A) \Longrightarrow \|(\phi_0, \phi_1) \|_{\HH_0}  \ge A
}
%
 \end{lem}
 
 Note that $\bfd$ is small when $\vec \phi$ is close to either a bubble/anti-bubble ($\iota= +$ in the definition of $\bfd$) or anti-bubble/bubble configuration ($\iota = -$). The next lemma makes precise the intuitive notion that a map $\vec \phi$ cannot be simultaneously close to both configurations. With this in mind we define 
 \EQ{\label{eq:dpm} 
 \bfd_\pm(\vec \phi) := \inf_{\la, \mu >0}  \Big(  \| (\phi_0  \mp  (Q_\la - Q_\mu), \phi_1) \|_{\HH_0}^2 + \left( \la/\mu \right)^k \Big)
 }
 
 \begin{lem}\label{l:dpm} 
 There exists $\al_0>0$ with the following property: Let $\vec \phi \in \HH_0$. Then, 
 \EQ{
 \bfd_{\pm}(\vec \phi) \le \al_0 \Longrightarrow \bfd_{\mp}(\vec \phi)  \ge \al_0. 
 }
 \end{lem} 
 
 We begin by proving Lemma~\ref{l:d-size}. 
 
 \begin{proof}[Proof of Lemma~\ref{l:d-size}]
 It suffices to consider $\vec \phi$ of the form $\vec \phi = (\phi, 0)$. First we prove~\eqref{eq:d-big}. To see this we'll first show that for each $\be >0$ there exists a constant $\de = \de(\be)$ so that for any $\vec \phi \in \HH_0$ with $\E(\vec \phi) = 2\E(Q)$ we have 
 \EQ{ \label{eq:d-big1}
 \bfd(\vec \phi) \ge \be \Longrightarrow \| \phi \|_{L^\infty} \le \pi -  \de(\be),
 }
 Suppose~\eqref{eq:d-big1} fails. Then we can find $\be>0$, a sequence $\vec \phi_n = (\phi_n, 0) \in \HH_0$ with $\E(\vec \phi_n) = 2\E(Q)$, and numbers $r_n >0$ so that 
 \EQ{ \label{eq:phinpi} 
 \bfd(\phi_n, 0) \ge \be \mand \abs{ \phi_n(r_n) - \pi}  = o_n(1) \mas n \to \infty
 }
 Define scales $\la_n$ and $\mu_n$ by 
 \EQ{ \label{eq:lanmun} 
 \E_0^{\la_n} (\vec \phi_n) =  \E(Q)/2, \quad  \E_{\mu_n}^\infty(\vec \phi_n) = \E(Q)/2
 }
 Then, by~\eqref{eq:GQ} we see that for $n $ large enough $\la_n < r_n$ and $\mu_n >r_n$. 
 Now define $\phi_{n, 1}$ and $\phi_{n, 2}$ as follows 
\EQ{
&\phi_{n, 1}(r)  = \begin{cases}  \phi_n(r)  \mif 0 \le r \le  r_n  \\  \pi + \frac{ \pi - \phi(r_n )}{r_n  } (r - 2r_n) \mif  r \in [r_n, 2r_n] \\  \pi \mif r \ge 2r_n\end{cases}  
\\
&  \phi_{n, 2}( r)  = \begin{cases}  \pi + \frac{\phi_{n}(r_n )- \pi}{ r_n } r  \mif r \le r_n  \\ \phi_n( r)  \mif r \ge  r_n   \end{cases}
 }
 And define $\eta_n(r)$ by 
 \EQ{
  \eta_n( r)  :=  \phi_n( r) - \phi_{n, 1}( r) - \phi_{n, 2}( r) + \pi
 }
 We claim that 
 \begin{align}
 &\E( \phi_{n, 1}, 0) = \E(Q, 0) + o_n(1) \mas n \to \infty \label{eq:phi1en} \\
 &\E(\phi_{n, 2}, 0) = \E(Q, 0) + o_n(1) \mas n \to \infty  \label{eq:phi2en}\\
 &\| \eta_n\|_{H} \to 0 \mas t \to \infty \label{tieta1} 
 \end{align}
First we prove~\eqref{eq:phi1en}~\eqref{eq:phi2en}. Since $\phi_n(r_n) \to \pi$ we have 
 \ant{
 &\E_0^{r_n}(\phi_{n, 1}, 0)  =  \E_0^{r_n}( \phi_{n}, 0)  \ge G( \phi_{n}(r_n)) \to G( \pi) = \E(Q, 0) \mas n \to \infty \\
 &\E_{r_n}^\infty(\phi_{n}, 0)  =  \E_{r_n}^\infty( \phi_{n}, 0)  \ge G( \phi_{n, 2}(r_n)) \to G( \pi) = \E(Q, 0) \mas n \to \infty 
}
From the above and the fact that $\E( \vec \phi) = 2 \E(Q)$ we see that in fact
\EQ{ \label{e1g}
&\E_0^{r_n}(\phi_{n, 1}, 0) = \E(Q, 0) + o_n(1) \mas t \to \infty \\
&\E_{r_n}^\infty(\phi_{n, 2}, 0) = \E(Q, 0) + o_n(1) \mas t \to \infty
}
Direct computations using the definitions of $\phi_{n, 1}, \phi_{n, 2}$  then show that 
\EQ{ \label{e2g}
&\E_{r_n}^{\infty}(\phi_{n, 1}, 0)  \lesssim  \left(  \pi - \phi_n(r_n) \right)^2 \to 0 \mas n \to \infty  \\
&\E_0^{r_n}(\phi_{n, 2}, 0)  \lesssim \left(  \pi -\phi_n(r_n) \right)^2 \to 0 \mas n \to \infty 
}
Combining~\eqref{e1g} and~\eqref{e2g} gives~\eqref{eq:phi1en} and~\eqref{eq:phi2en}. By construction
 \EQ{
 &\eta_n( r) =  \pi - \phi_{n, 2}( r) \mif r \le r_n, \quad  \eta_n( r)  =  \pi - \phi_{n, 1}( r) \mif r \ge r_n
 }
 A direct computation using the above and the definitions of $\phi_{n, 1}, \phi_{n, 2}$ on the relevant intervals then yields
  \EQ{
   \|\eta_n \|_{H}^2 \lesssim  \left(  \pi - \phi_n(r_n) \right)^2 \to 0 \mas n \to \infty
   }
   By~\eqref{eq:phi1en}, and~\eqref{eq:phi2en}, and $\la_n, \mu_n$ defined in~\eqref{eq:lanmun} we  use Lemma~\ref{l:decQ} to find $\eta_{n, 1}, \eta_{n, 2} \in H$ so that  
   \EQ{
   & \phi_{n, 1}( r) =  Q_{\la_n} + \eta_{n, 1}(r) , \quad  \phi_{n, 2}( r) = \pi -   Q_{\mu_n} - \eta_{n, 2}(r) \\
      &\| \eta_{n, j} \|_{H} \to 0 \mas n \to \infty
      }
      for 
      $j  = 1, 2$. 
  Thus, 
 \EQ{\label{eq:partd}
 \| \phi_{n} - Q_{\la_n} + Q_{\mu_n} \|_{H}  = \| \eta_n + \eta_{n, 1} -  \eta_{n, 2} \|_H \to 0 \mas n \to \infty
 }
 Moreover, we must have $\la_n/ \mu_n \to 0 \mas n \to \infty$. To see this, simply note that if $\la_n/ \mu_n \simeq 1$ then $Q_{\la_n} - Q_{\mu_n}$ stays bounded away from $\pi$. But this contradicts~\eqref{eq:partd} and the assumption that $\phi_n(r_n) \to \pi$ as $n \to \infty$. Hence, 
 \EQ{
  \| \phi_{n} - Q_{\la_n} + Q_{\mu_n} \|_{H}^{2} +\left( \frac{\la_n}{\mu_n}\right)^k \to 0
 }
 and thus $\bfd(\phi_n, 0) \to 0$, which contradicts~\eqref{eq:phinpi}. To finish the proof, note that~\eqref{eq:d-big1}  implies the estimate 
 \EQ{
 \phi^2(r) \le C(\be) \sin^2 \phi(r)
 }
 which means we can control the $H$ norm of $\phi$ by a constant (which depends only on $\be$) times $\E(\phi) = 2\E(Q)$.

 Lastly, we prove~\eqref{eq:d-small}. Suppose that $\bfd(\vec \phi) \le \al$. Then we can find, say, $\la_0, \mu_0$ such that 
 \EQ{
\al \le  \|( \phi_0 - Q_{\la_0} + Q_{\mu_0}, \phi_1) \|_{\HH_0} + \left(\la_0/\mu_0 \right)^k  \le 2 \al
 }
 A direct computation then shows, 
 \EQ{
 \| \phi_0 \|_H  &\ge  \| Q_{\la_0} - Q_{\mu_0} \|_H  - \| \phi_0 - Q_{\la_0} + Q_{\mu_0} \|_{H} \\
 & \gtrsim   \abs{\log(\la_0/ \mu_0)} - 2\al  \to  \infty \mas \al   \to 0
 }
 which completes the proof. 
 \end{proof}

 We next prove Lemma~\ref{l:dpm}. 
 \begin{proof}[Proof of Lemma~\ref{l:dpm}]
If the conclusion fails we could find a sequence  $\phi_n \in H$, and two sequences of scales $\la_n^+, \mu_n^+, \la_n^-, \mu_n^-$ so that 
\EQ{
\| \phi_n - Q_{\la_n^+} + Q_{\mu_{n}^+} \|_H + \frac{\la_{n}^{+}}{\mu_{n}^{+}}  \to 0 \mas n \to \infty \\
\| \phi_n + Q_{\la_n^-} - Q_{\mu_{n}^-} \|_H + \frac{\la_{n}^{-}}{\mu_{n}^{-}}  \to 0 \mas n \to \infty
}
It follows that 
\begin{align} 
0   &= \|( \phi_n - Q_{\la_n^+} + Q_{\mu_{n}^+} ) - (\phi_n  +Q_{\la_n^-} -Q_{\mu_{n}^-}) +(Q_{\la_n^+} - Q_{\mu_{n}^+} + Q_{\la_n^-} - Q_{\mu_{n}^-})\|_H \\
& \ge  \| Q_{\la_n^+} - Q_{\mu_{n}^+} + Q_{\la_n^-} - Q_{\mu_{n}^-} \|_H - o_n(1) \mas n \to \infty\label{eq:dpm-bad}
\end{align}
Passing to subsequences if necessary,  relabeling $\pm$, or rescaling, we can assume that $\la_{n}^+ \le \la_{n}^-$ for all $n$  and that one of  the following three possibilities holds
\EQ{
\frac{\la_{n}^{-}}{\mu_{n}^+} \to 0 , \mor \frac{\la_{n}^{-}}{\mu_{n}^+} \to \infty, \mor \frac{\la_{n}^{-}}{\mu_{n}^+} \to 1>0, \mas n \to \infty
}
Assume we are in the first situation. Then, we can choose $n$ large enough so that 
\EQ{
Q_{\la_n^+}(r) + Q_{\la_n^-}(r) \ge  \pi \quad \forall r \in [\la_{n}^-,  2 \la_{n}^-] \\ 
 Q_{\mu_{n}^+}(r) +  Q_{\mu_{n}^-}(r) \le \frac{\pi}{2}  \quad \forall r \in [\la_{n}^-,  2 \la_{n}^-]
}
and thus 
\EQ{
 \| Q_{\la_n^+} - Q_{\mu_{n}^+} + Q_{\la_n^-} - Q_{\mu_{n}^-} \|_H^2 \ge \frac{\pi^2}{4}\int_{\la_n^-}^{2 \la_n^-} \frac{\ud r}{r} \ge  \frac{\pi^2}{4} \log 2
}
for all $n$ large enough, which is impossible by~\eqref{eq:dpm-bad}. Now suppose we are in the second case $\frac{\la_{n}^{-}}{\mu_{n}^+} \to \infty$. This means that 
\EQ{
\la_n^+ \ll \mu_{n}^+ \ll \la_{n}^- \ll \mu_n^-
}
and so for large enough $n$ we have 
\EQ{
(Q_{\la_n^+} - Q_{\mu_{n}^+} + Q_{\la_n^-} - Q_{\mu_{n}^-})(r) \ge \frac{\pi}{4}, \quad \forall r \in  [\la_{n}^+,  2 \la_{n}^+]
}
which similarly leads~\eqref{eq:dpm-bad} into a contradiction. Finally, if $ \frac{\la_{n}^{-}}{\mu_{n}^+} \to 1$ we have 
\EQ{
\| Q_{\la_n^+} - Q_{\mu_{n}^+} + Q_{\la_n^-} - Q_{\mu_{n}^-} \|_H  \ge \| Q_{\la_n^+} - Q_{\mu_{n}^-} \|_H - o_n(1)
}
Then setting $\fy_n:=Q_{\la_n^+} - Q_{\mu_{n}^-}$ we see that $\bfd((\fy_n, 0)) \to 0$ and hence the right-hand-side above is bounded below by a fixed constant by~\eqref{eq:d-small} in Lemma~\ref{l:d-size}. This again leads to a contradiction in~\eqref{eq:dpm-bad}, which completes the proof. 
 \end{proof}

%
 
 \subsection{Virial identity} 
 
 In this section we record a \emph{nonlinear} estimate related to a virial-type identity that will be used in the proof of Theorem~\ref{t:main}. 
 
 We begin with a virial-type  identity for solutions to~\eqref{eq:wmk}. In what follows we fix a smooth radial cut-off function  $\chi \in C^{\infty}_{\mathrm{rad}}(\R^2)$, so that, writing $\chi = \chi(r)$ we have 
 \EQ{
  \chi(r) = 1 \mif r \le 1 \mand  \chi(r) = 0 \mif r \ge 3 \mand \abs{\chi'(r)} \le 1\quad \forall r \ge 0
 }
 For each $R>0$ we then define 
 \EQ{
 \chi_R(r) := \chi(r/ R)
 }
 
 \begin{lem} \label{l:vir} 
 Let $\vec \psi(t)$ be a solution to~\eqref{eq:wmk} on a time interval $I$. Then for any time $t \in I$  and $R>0$ fixed we have 
 \EQ{\label{eq:vir}
 \frac{\ud}{\ud t} \ang{ \psi_t \mid \chi_R \, r \p_r \psi}_{L^2}(t)  = - \int_0^\infty \psi_t^2(t, r) \, \rdr + \Om_R(\vec \psi(t))
 }
 where 
 \EQ{ \label{eq:OmRdef} 
 \Om_R(\vec \psi(t)) &:=  \int_0^\infty \psi_t^2(t)(1 - \chi_R) \, \rdr   \\
 & \quad -\frac{1}{2} \int_0^\infty  \Big( \psi_t^2(t) + \psi_r^2(t) - k^2 \frac{\sin^2 \psi(t)}{r^2} \Big)   \frac{r}{R} \chi'(r/R) \rdr
 }
 satisfies 
 \EQ{ \label{eq:OmRest} 
 \abs{\Om_R(\vec \psi(t))} &\lesssim \int_{R}^\infty \psi_t^2(t, r) \, r \ud r  \, \ud t + \int_{R}^{\infty} \abs{ \psi_r^2 - k^2 \frac{\sin^2 \psi}{r^2} } \,r  \ud r \ud t  \\
 &\lesssim  \E_{R}^{\infty}(\vec \psi(t))
 }
 \end{lem}
 \begin{proof}
 By direct calculation, using~\eqref{eq:wmk} we have 
 \ant{
 \frac{\ud}{\ud t} \ang{ \psi_t \mid \chi_R \, r \p_r \psi}_{L^2}(t) &=  - \int_0^\infty \psi_t^2(t) \, \rdr   + \int_0^\infty \psi_t^2(t)(1 - \chi_R) \, \rdr \\
 & \quad -\frac{1}{2} \int_0^\infty  \Big( \psi_t^2(t) + \psi_r^2(t) - k^2 \frac{\sin^2 \psi(t)}{r^2} \Big)   \frac{r}{R} \chi'(r/R) \rdr
 }
  \end{proof}
 
  We show below how the quantities appearing on the right hand side of the virial identity can be estimated in terms of $\bfd(\vec \psi)$
 in the vicinity of a two-bubble.
 \begin{lem}\label{l:error-estim}
 There exists a number $C_0 > 0$ depending only on $k$ such that for all $\vec \phi = (\phi_0, \phi_1) \in \HH_0$ with $\E(\vec\phi) = 2\E(Q)$
 and all $R > 0$ there holds
 \begin{gather}
 |\ang{\phi_1, \chi_R r\partial_r \phi_0}| \leq C_0 R \sqrt{\bfd(\vec \phi)}, \label{eq:virial-end} \\
 \Omega_R(\vec \phi) \leq C_0 \sqrt{\bfd(\vec \phi)}. \label{eq:virial-err}
 \end{gather}
 \end{lem}
 \begin{proof}
 By Cauchy-Schwarz, we get
 \EQ{
 |\ang{\phi_1, \chi_R r\partial_r \phi_0}| \lesssim R \|\phi_1\|_{L^2}\|\partial_r \phi_0\|_{L^2}. \label{eq:virial-end-1}
 }
 We have $\|\partial_r(Q_\lambda - Q_\mu)\|_{L^2} \lesssim 1$ for all $\lambda$ and $\mu$, hence by the triangle inequality
 $\|\partial_r \phi_0\|_{L^2} \lesssim 1 + \sqrt{\bfd(\vec \phi)}$. If $\bfd(\vec\phi) \leq 1$, then we obtain $\|\partial_r \phi_0\|_{L^2} \lesssim 1$.
 If $\bfd(\vec \phi) \geq 1$, then Lemma~\ref{l:d-size} gives $\|\vec\phi\|_{\HH_0} \lesssim 1$,
 in particular again $\|\partial_r \phi_0\|_{L^2} \lesssim 1$. Thus \eqref{eq:virial-end-1} yields \eqref{eq:virial-end}.
 
To prove \eqref{eq:virial-err}, we write
 \EQ{
   \big|\Omega_R(\vec\phi)\big| \lesssim \int_0^{+\infty}\phi_1^2r\ud r + \Big|(\partial_r \phi_0)^2 - k^2\frac{\sin^2\phi_0}{r^2}\Big|r\ud r.
 }
 Again, the conclusion is clear if $\bfd(\vec\phi) \geq 1$, we can assume $\bfd(\vec\phi) \leq 1$.  Find $\la, \mu>0$ such that, say, 
 \EQ{
 \left( \frac{\la}{\mu} \right)^k  \le 2\bfd( \vec \phi) \mand \| (\phi_0 - Q_\la + Q_\mu), \phi_1 \|_{\HH_0}^2  \le 2\bfd(\vec \phi)
 } 
 By the above it suffices to show that for $g:= \phi_0 - Q_\la + Q_\mu$ we have 
 \EQ{
 \int_0^{+\infty} \Big|(\partial_r \phi_0)^2 - k^2\frac{\sin^2\phi_0}{r^2}\Big|r\ud r   \lesssim  \left( \left( \frac{\la}{\mu} \right)^{k/2} + \| g \|_{H} \right)
 }
 Using trigonometric identities we expand 
  \ant{
 \sin^2( Q_\la - Q_\mu + g)  & =     \sin^2 Q_\la + \sin^2 Q_\mu  - \frac{1}{2} \sin2Q_\la \sin2 Q_\mu - 2 \sin ^2Q_\la \sin^2 Q_\mu \\
 & \quad +\frac{1}{2} \sin 2g \sin 2(Q_\la - Q_\mu) + \sin^2 g \cos2(Q_\la - Q_\mu) 
 }
Then, since $\La Q_{\la}:= r \p_r Q_\la  = k \sin Q_\la$ we have 
 \EQ{
 \int_0^\infty \Big|(r\partial_r \phi_0)^2 - k^2\sin^2\phi_0\Big|  \frac{\ud r}{r} &\lesssim  \int_0^\infty   \abs{ \La Q_\la \La Q_\mu} + \abs{ \La Q_\la  r\p_r g} + \abs{ \La Q_\mu r\p_r g}  \frac{\ud r}{r}\\
  & \quad + \int_0^\infty \abs{ g \La Q_\la  } + \abs{ g \La Q_\mu r } + \abs{r \p_r g}^2+ \abs{g}^2 \, \frac{\ud r}{r}
 }
 To estimate the first term above we see that setting $\s = \la/ \mu$ we have 
 \EQ{
 \int_0^\infty   \abs{ \La Q_\la \La Q_\mu} \frac{\ud r}{r}  &\lesssim  \int_0^\infty  \frac{(r/\la)^k (r/\mu)^k}{ (1+ (r/\la)^{2k})(1+ (r/\mu)^{2k})} \, \frac{\ud r }{r} \\
 &= \s^{k} \int_0^\infty\frac{r^{2k-1}}{ (\s^{2k} + r^{2k})( 1+ r^{2k})} \, \ud r \lesssim \s^{k} \abs{\log \s} \lesssim \left(\frac{\la}{\mu}\right)^{\frac{k}{2}}
 }
The remaining terms can be controlled by $\| g \|_H$ by Cauchy-Schwarz. 
  \end{proof}


\section{The modulation method:  analysis of $2$-bubble collisions} \label{s:mod}

In this section we give a careful analysis of the modulation equations that govern the
evolution of $2$-bubble configurations.
The intuition is that the less concentrated bubble does not change its scale and influences
the dynamics of the more concentrated bubble. We will quantify this influence.

\subsection{Modulation Equations}
We consider solutions $\vec \psi(t)$ to~\eqref{eq:wmk} that are close to a $2$-bubble configuration on a time interval $J$ in the sense that $ \bfd( \vec \psi(t))$,
defined in \eqref{eq:ddef}, is small for all $t \in J$.
Recall that $\bfd(\vec\psi(t))$ is the smaller of the numbers $\bfd_+(\vec\psi(t))$ and $\bfd_-(\vec\psi(t))$ defined in \eqref{eq:dpm}.

Linearizing \eqref{eq:wmk} about $Q_\la$ leads to the Schr\"odinger operator
\EQ{
\LL_\la:= - \p_r^2 - \frac{1}{r} \p_r + k^2 \frac{\cos 2Q_\la}{r^2} 
}
We write $\LL := \LL_1$. Recall from~\eqref{eq:LaLa0} that $\La = r \p_r$ is the infinitesimal generator of dilations in $\dot{H}^1(\R^2)$. One can check that  $\La Q$ is a zero energy eigenfunction for  $\LL$, i.e., 
\EQ{
\LL \La Q = 0, \mand \La Q  \in L^2_{\textrm{rad}}(\R^2).
}
When $k=1$, $\LL \La Q = 0$ still holds but in this case $\La Q \not \in L^2$ due to slow decay as $r \to \infty$ and is $0$ is referred to as a threshold resonance.

In fact, $\La Q$ spans the kernel of $\LL$. This can be seen using the following well known factorization of $\LL$, 
\EQ{ \label{eq:LLfact} 
\LL = A^* A  \mwhere   A^* = \partial_r + \frac{1+k\cos (Q)}{r}, \quad A = - \partial_r + \frac{k\cos(Q)}{r}
}
together with the observation that  $A(\La Q) = 0$;  we note that~\eqref{eq:LLfact} is a consequence of the Bogomol'nyi factorization~\eqref{eq:bog}; see~\cite{RS, RR} for more.  

The fact that $\LL_\la \Lambda Q_\la = 0$ will play an important role in the modulation estimates.

We fix a radial function $\ZZ \in C^{\infty}_0(\R^2)$ so that 
\EQ{ \label{eq:ZZdef} 
\int_0^\infty  \ZZ(r)  \cdot \La Q(r) \,  r \, \dr >0, \quad \abs{ \frac{\ZZ(r)}{r^k}} \lesssim 1  \quad  \forall \, r \le 1. 
}
%
%

\begin{lem}[Modulation Lemma]  \label{l:modeq} There exist $\eta_0>0$ and $C > 0$
 with the following property:   Let $J\subset \R$ be a time interval, $\vec \psi(t)$ a solution to~\eqref{eq:wmk} defined on $J$,  and assume that
 \EQ{
 \bfd_+(\vec\psi(t)) \leq \eta_0\qquad \forall t\in J.
 }
 Then, there exist unique $C^1(J)$ functions $\la(t), \mu(t)$ so that, defining $g(t) \in H$ by
\EQ{ \label{eq:gdef1} 
g(t):= \psi(t) - Q_{\la(t)} + Q_{\mu(t)} , 
}
we have, for each $t \in J$,
\begin{gather} 
\ang{ \cZ_{\uln{\lambda(t)}} \mid  g(t)}  = 0  \label{eq:ola},   \\
\ang{\ZZ_{\uln{\mu(t)}} \mid g(t) } = 0  \label{eq:omu},  \\
\bfd_+(\vec \psi(t)) \le \| (g(t), \psi_t(t)) \|_{\HH_0}^2 +\big(\lambda(t)/\mu(t)\big)^k \le C \bfd_+(\vec\psi(t))\label{eq:gdotgd} . 
\end{gather} 
Moreover, 
\begin{align} 
\| (g(t), \psi_t(t) ) \|_{\HH_0} \le C\left( \frac{\la(t)}{\mu(t)} \right)^{\frac{k}{2}} \label{eq:gH} , 
\end{align} 
and hence 
\EQ{ \label{eq:lamud1}
\bfd_+(\vec \psi(t)) \simeq \left( \frac{\la(t)}{\mu(t)} \right)^{k}. 
}
\end{lem} 

\begin{rem} \label{r:ift}  The following version of the implicit function theorem will be used in the proof.

 \emph{
Let $X, Y, Z$ be Banach spaces. Let $(x_0, y_0) \in X \times Y$, let $\de_1, \de_2>0$ and consider a mapping  $G: B(x_0, \de_1) \times B(y_0, \de_2) \to Z$ that is continuous in $x$ and $C^1$ in $y$. Suppose that $G(x_0, y_0) = 0$ and $(D_y G)(x_0, y_0)$ has bounded inverse $L_0$. Moreover, suppose that 
\EQ{ \label{eq:ift} 
&\| L_0 - D_y G(x, y) \|_{\LL(Y, Z)} \le \frac{1}{3 \| L_0^{-1} \|_{\LL(Z, Y)}}  \\ 
& \| G(x, y_0) \|_Z \le \frac{\de_2}{ 3 \| L_0^{-1} \|_{\LL(Z, Y)}} 
}
for all $\| x - x_0 \|_{X} \le \de_1$ and $\| y - y_0 \|_Y \le \de_2$. 
Then, there exists a continuous function $\si: B(x_0, \de_1)  \to B(y_0, \de_2)$ such that for all $x \in B(x_0, \de_1)$, $y = \si(x)$ is the unique solution of $G(x, \si(x)) = 0$ in $B(y_0, \de_2)$. 
}

The above is proved in the same fashion as the usual implicit function theorem, see, e.g.,~\cite[Section 2.2]{ChowHale}. The key point is that the bounds~\eqref{eq:ift} give uniform control on the size of the open set on which the Banach contraction mapping theorem can be applied. 
\end{rem} 

%

\begin{proof}
The proof 
follows by standard techniques that we outline below; we refer the reader to~\cite[Lemma 3.3]{JJ15}  for a detailed proof of a similar statement. 

 We begin by showing that for each $t\in J$ there exist unique $\la(t), \mu(t)$, and $g(t)$ that satisfy~\eqref{eq:gdef1},~\eqref{eq:gdotgd} and the orthogonality conditions~\eqref{eq:ola}~\eqref{eq:omu} using an argument based on the implicit function theorem stated in Remark~\ref{r:ift}. That $\la(t)$ and $\mu(t)$ are actually $C^1(J)$ is then proved via a standard ODE argument, which we postpone until Remark~\ref{r:ODE} in Section~\ref{s:modest}. 
 
 To establish the former statement let $\vec \phi \in \HH_0$ be such that $\bfd_+(\vec \phi) \le \eta_0$. This means we can find $\la_0, \mu_0>0$ such that for $g_0 \in H$ defined by 
\EQ{ \label{eq:g0def1} 
g_0 := \phi_0 - (Q_{\la_0}  - Q_{\mu_0})
}
we have 
\EQ{
\| (g_0, \phi_1) \|_{H\times L^2}^2 + \left( \frac{\la_0}{\mu_0} \right)^k  \le 2 \eta_0
}

Define the mapping 
$
F: H \times (0, \infty) \times (0, \infty) \to H$ by 
\EQ{
F( g, \la, \mu) :=  g - (Q_{\la} - Q_\mu)  + (Q_{\la_0} - Q_{\mu_0})
}
Note that $F(0, \la_0, \mu_0) = 0$ and moreover that 
\EQ{ \label{eq:FH}
\| F( g, \mu, \la) \|_{H} \lesssim  \|g \|_{H} + \abs{(\la/\la_0)-1}^{\frac{1}{2}} + \abs{( \mu/ \mu_0) - 1}^{\frac{1}{2}} 
}
Next define a mapping $G: H \times (0, \infty) \times (0, \infty) \to \R^2$  by 
\EQ{
G(g, \la, \mu)  :=  \pmat{ \ang{ \frac{1}{\la} \ZZ_{\ula} \mid F(g, \la, \mu)} , \, \ang{ \frac{1}{\mu} \ZZ_{\U\mu}  \mid F(g, \la, \mu)} }
}
For $g \in H$ we have 
\EQ{
\frac{1}{\la}\ang{  \ZZ_{\ula}  \mid g } \le \| r\la^{-1} \ZZ_{\ula} \|_{L^2} \| r^{-1} g \|_{L^2} \lesssim   \|g \|_H , \quad
\frac{1}{\mu}\ang{ \ZZ_{\U\mu}  \mid g } \lesssim  \|g \|_H
}
which ensures that the mapping $G$ is well-defined and continuous. Taking the $\la, \mu$ derivatives of $G$, we have 
\EQ{ \label{eq:pA1} 
\frac{\ud}{ \ud \la} \ang{ \frac{1}{\lambda}\ZZ_{\ula} \mid F(g, \la, \mu)} &=  \frac{1}{\la}\ang{ \ZZ_{\ula} \mid \La Q_{\ula}} - \frac{1}{\la^2} \ang{ (\La_0+1) \ZZ_{\ula} \mid F(g, \la, \mu)} \\
&  =: \frac{1}{\la} A_{11}(g, \la, \mu)\\
\frac{\ud}{ \ud \mu} \ang{\frac{1}{\lambda}\ZZ_{\ula} \mid F(g, \la, \mu)} &=  -\frac{1}{\la} \ang{  \ZZ_{\ula} \mid \La Q_{\umu}} =: \frac{1}{\la} A_{1 2}(g, \la,\mu)  
}
and 
\EQ{\label{eq:pA2}
\frac{\ud}{ \ud \la} \ang{ \frac{1}{\mu}\ZZ_{\umu} \mid F(g, \la, \mu)} &= \frac{1}{\mu} \ang{ \ZZ_{\umu} \mid \La Q_{\ula}} =: \frac{1}{\mu} A_{2 1}(g, \la, \mu) \\
\frac{\ud}{ \ud \mu} \ang{ \frac{1}{\mu}\ZZ_{\umu} \mid F(g, \la, \mu)} &=  - \frac{1}{\mu} \ang{ \ZZ_{\umu} \mid \La Q_{\umu}} - \frac{1}{\mu^2} \ang{ (\La_0+1) \ZZ_{\umu} \mid F(g, \la, \mu)}  \\
&=:  \frac{1}{\mu} A_{22} (g, \la, \mu)
}
For convenience in applying the implicit function theorem we change variables, setting $\ell := \log \la$ and $m := \log \mu$. In the new variables we write 
\EQ{
\ti G( g,  \ell, m) = G( g, \la,\mu), \quad \ti F ( g, \ell, m) = F(g, \la, \mu)
}
We now check that the conditions~\eqref{eq:ift} are satisfied for $x_0 = 0 \in H$, $y_0 = (\ell_0, m_0) \in \R^2$ and $\ti G: B_H(0, 2 \eta_0) \times  B_{\R^2}( (\ell_0, m_0), C_0\eta) \to \R^2$, for $ \de_1 = 2\eta_0>0$ small enough and $C_0$ a uniform constant.  Since $\p_\ell = \la \p_\la$ and $\p_m = \mu \p_\mu$ we deduce using~\eqref{eq:pA1},~\eqref{eq:pA2} that 
\EQ{
D_{\ell, m} \ti G( g,  \ell, m) = \pmat{ A_{11}(g, \la, \mu) & A_{12}(g,\la, \mu) \\ A_{21}(g, \la, \mu) & A_{22}(g, \la, \mu)}
}

Restricting to $(g, \la, \mu) = (0, \la_0, \mu_0)$, this yields  
\EQ{
L_0:= D_{\ell, m} \ti G\rest_{(g=0, \ell = \ell_0, m=m_0)} =  \pmat{ \ang{ \ZZ_{\U{\la_0}}\mid \La Q_{\U{\la_0}}} &  {-}\ang{\ZZ_{\U{\la_0}} \mid \La Q_{\U{\mu_0}}} \\  \ang{ \ZZ_{\U{\mu_0}} \mid \La Q_{ \U{\la_0}}} & {-}\ang{ \ZZ_{\U{\mu_0}} \mid \La Q_{\U{\mu_0}}}} =: \pmat{ A_{11} & A_{12} \\ A_{21} & A_{22}}
}
The diagonal terms in the matrix $L_0$ above are size $O(1)$ by~\eqref{eq:ZZdef}. We can estimate the off-diagonal terms as follows:
  \begin{claim} \label{c:M12est}
  For $\lambda \ll \mu$ we have
 \begin{align} \label{eq:M12}
\abs{\ang{ \cZ_{\U{\lambda}}\mid \Lambda Q_{\U{\mu}}}}  \lesssim ({\la/\mu})^{k+1}, \quad 
\abs{\ang{ \cZ_{\U{\mu}}\mid \Lambda Q_{\U{\lambda}}}} \lesssim ({\la/\mu})^{k-1}
 \end{align} 
 \end{claim}
To prove the Claim, without loss of generality we can assume $\mu = 1$ and $\la \ll 1$. Let $B>0$ be such that $\supp \ZZ \subset \{ r \le  B\}$. Then, using~\eqref{eq:ZZdef} we have 
\EQ{
\abs{\ang{ \cZ_{\U{\lambda}}\mid \Lambda Q}}& \lesssim \frac{1}{\la}  \int_0^\infty \ZZ(r/ \la) \La Q(r) \, r \, \ud r \\
 & \lesssim \frac{1}{\la} \int_0^\la (r/ \la)^k \frac{r^{k+1}}{1 + r^{2k}} \, \ud r + \frac{1}{\la} \int_\la^{B \la}  \frac{r^{k+1}}{1 + r^{2k}} \, \ud r   \lesssim \la^{k+1} 
}
Similarly, 
\EQ{
\abs{\ang{ \cZ\mid \Lambda Q_{\U{\lambda}}}}  & \lesssim \frac{1}{\la} \int_0^\infty \La Q( r/ \la) \ZZ(r) \, \ud r \\
& \lesssim  \la^{k-1} \int_0^\la \frac{ r^k}{ \la^{2k} + r^{2k} } r^{k+1} \, \ud r +  \la^{k-1}  \int_\la^B \frac{\ZZ(r)}{r^k} \frac{ r^{2k+1}}{ \la^{2k} + r^{2k} } \,\ud r \\
& \lesssim \la^{-k-1} \int_0^\la r^{2k+1} \, \ud r  +  \la^{k-1}  \lesssim \la^{k+1} + \la^{k-1}  \lesssim \la^{k-1} 
}
which proves Claim~\ref{c:M12est}. 

This proves that the off diagonal terms in $L_0$ are of size $O((\la_0/\mu_0)^{k-1})$.
Hence for $k \ge 2$ the matrix $L_0$ is invertible as long as $(\la_0/\mu_0)^{k-1}$  is small enough.

The second condition in~\eqref{eq:ift} is clear since $F(g, \la_0, \mu_0) = g$ and hence
\EQ{
| G( g, \la_0, \mu_0)| = \abs{  \pmat{ \frac{1}{\la} \ang{ \ZZ_{\ula} \mid g} , \, \frac{1}{\mu} \ang{ \ZZ_{\U\mu}  \mid g} }} \lesssim   \| g\|_H
}
The first condition in~\eqref{eq:ift} follows from a direct computation checking that for $1 \le j, k \le2$, 
\EQ{
\abs{  A_{ij}(g, \la, \mu) - A_{ij} } \lesssim ( \abs{ \la/\la_0 - 1}^{\frac{1}{2}} + \abs{\mu/ \mu_0 -1}^{\frac{1}{2}} + \|g \|_H)   \ll 1 
} 
as long as $\eta_0>0$ is chosen small enough. Here let us just remark that the factors involving $\la/\la_0$ and $\mu/\mu_0$ on the right-hand-side above appear from the estimates 
\EQ{
 \| \La Q_{\U \s} - \La Q_{\U \s_0} \|_{L^2} + \| \ZZ_{\U \s} - \ZZ_{\U \s_0} \|_{L^2} \lesssim \abs{\s/ \s_0 - 1}^{\frac{1}{2}}
}
An application of Remark~\ref{r:ift}  yields the following: There exists $\eta_0>0$ small enough 
 and a continuous mapping $\si: B_{H}(0, 2\eta_0) \to B_{\R^2}((\ell_0, m_0), C_0 \eta_0)$
 so that for all $(g, \ell, m) \in B_{H}(0, \eta_0) \times B_{\R^2}( (\ell_0, m_0), C_0 \eta_0)$ we have 
\EQ{
G( g, \la, \mu) \equiv 0 \Longleftrightarrow  (\ell, m) = \si(g), \quad \la = e^\ell, \, \, \mu = e^m
}
Finally, we observe that if we let $ g_0$ be as in~\eqref{eq:g0def1}, and define $(\la, \mu) = (e^{\ell}, e^{m})$ and $ g \in H$ by 
\EQ{
(\ell, m):=  \si( g_0), \quad g:= F(  g_0, \la, \mu)
}
we see that $ \phi_0 = Q_{\la} - Q_{\mu} + g$ and moreover that 
\EQ{
\ang{ \ZZ_{\ula}  \mid g} = 0 \mand \ang{ \ZZ_{\umu} \mid g} = 0. 
}
Lastly, since $\abs{\ell - \ell_0} \le C_0 \eta_0$ and $\abs{m - m_0} \le C_0 \eta_0$, we have
\begin{equation}
\label{eq:g-small-666}
\abs{\lambda/\lambda_0 - 1} + \abs{\mu/\mu_0 - 1} \lesssim \eta_0 \ll 1,\qquad \|g\|_H \lesssim \sqrt{\eta_0} \ll 1.
\end{equation}
In particular, we have
\begin{equation}
\label{eq:g-small-6666}
(\lambda/\mu)^k \lesssim (\lambda_0/\mu_0)^k  \leq \bfd_+(\vec \phi).
\end{equation}

Now we establish the estimate~\eqref{eq:gH}. This follows by expanding the nonlinear energy. We'll make use of trigonometric identities here for simplicity but note that the following computation relies only on the fact that the nonlinearity in~\eqref{eq:wmk} is smooth, that $Q$ is a solution, and the orthogonality conditions~\eqref{eq:ola}~\eqref{eq:omu}.  
\EQ{ \label{enexp1}
\frac{2}{\pi} \E(Q) &= \frac{1}{ \pi}\E( \vec \phi) = \frac{2}{\pi} \E(Q) + \int_0^\infty g_r^2    \rdr  + \int_0^\I \psi_t^2 \rdr \\
& - 2 \int_0^{\infty}  \p_r Q_\la \p_r Q_\mu \rdr + 2 \int_0^\infty  \p_r Q_\la g_r \rdr - 2 \int_0^\infty  \p_r Q_\mu g_r \rdr \\
&  + k^2 \int_0^\I  \frac{ \sin^2( Q_\la- Q_\mu + g)}{r^2} \rdr - k^2 \int_0^\I  \frac{ (\sin^2 Q_\la + \sin^2 Q_\mu )}{r^2} \rdr 
}
We expand the nonlinear terms on the last line using trigonometric  identities  
\ant{
 \sin^2( Q_\la - Q_\mu + g)  &= \sin^2 (Q_\la - Q_{\mu})    + \frac{1}{2}\sin 2 g \sin2(Q_\la- Q_\mu)    \\
  &\quad + \sin^2 g  \cos 2(Q_\la - Q_\mu)\\
 & =     \sin^2 Q_\la + \sin^2 Q_\mu  - \frac{1}{2} \sin2Q_\la \sin2 Q_\mu - 2 \sin ^2Q_\la \sin^2 Q_\mu \\
 & \quad +  g \sin 2(Q_\la - Q_\mu) + g^2 \cos2(Q_\la - Q_\mu) +O(\abs{g}^3)
 }
 which further reduces to 
\EQ{  \label{sin2Q}
 & =  \sin^2 Q_\la + \sin^2 Q_\mu  +  g^2 \cos2(Q_\la - Q_\mu)   - 2 \sin ^2Q_\la \sin^2 Q_\mu \\
 & \quad +  g \sin 2Q_\la  - g \sin 2 Q_\mu    -  \sin2Q_\la Q_\mu  + \frac{1}{2} \sin 2Q_{\la}  \big[2Q_\mu - \sin 2Q_\mu]\\
 & \quad - g (2\sin 2Q_\la \sin^2Q_\mu - 2\sin 2Q_\mu \sin^2Q_\la)  + O(\abs{g}^3) 
}
Next we observe that the first three terms in the second line of~\eqref{sin2Q} will give exact cancelations with the terms in the second line of~\eqref{enexp1}. Indeed, using the identity 
\EQ{
 \frac{1}{r} \p_r( r \p_r Q_\la) = k^2\frac{\sin 2 Q_\la}{2r^2}
 }
 we  integrate by parts to obtain
 \ant{
 &k^2 \int_0^{\infty}  \frac{  \sin 2Q_{\la}}{r^2}Q_\mu \rdr  =  2\int_0^\I \frac{1}{r} \p_r( r \p_r Q_\la) Q_\mu \rdr = - 2\int_0^{\infty}  \p_r Q_\la \p_r Q_\mu \rdr\\
 & k^2 \int_0^{\infty}  \frac{  \sin 2Q_{\la}}{r^2}g \rdr  =  2\int_0^\I \frac{1}{r} \p_r( r \p_r Q_\la) g \rdr = - 2\int_0^{\infty}  \p_r Q_\la g_r  \rdr\\
 & k^2 \int_0^{\infty}  \frac{  \sin 2Q_{\mu}}{r^2}g \rdr  =  2\int_0^\I \frac{1}{r} \p_r( r \p_r Q_\mu) g \rdr = - 2\int_0^{\infty}  \p_r Q_\mu g_r  \rdr
}
We can use the same identity to integrate by parts the terms arising from the rest of~\eqref{sin2Q}. 
\ant{
-2 k^2 \int_0^\I \frac{ \sin 2Q_\la}{r^2} & \sin^2Q_\mu g \rdr = -4 \int_0^\I \frac{1}{r} \p_r( r \p_r Q_\la) \sin^2 Q_\mu g  \rdr \\
& = 4 \int_0^\I \p_r Q_\la  \p_r Q_\mu \sin 2Q_\mu g \rdr + 4 \int_0^\I \p_r Q_\la  \sin^2 Q_\mu g_r \rdr
}
and the same for the symmetric term in $\la, \mu$. Finally, we have 
\ant{
k^2 \int_0^\I \frac{ \sin2 Q_\la}{ 2r^2}& \big[2Q_\mu - \sin 2Q_\mu] \rdr =  \int_0^\I\frac{1}{r}\p_r( r \p_r Q_\la)  \big[2Q_\mu - \sin 2Q_\mu] \rdr \\
& =  - \int_0^\I \p_r Q_\la \p_r  \big[2Q_\mu - \sin 2Q_\mu] \rdr  = - 4\int_0^\I \p_r Q_\la \p_r Q_\mu \sin^2 Q_\mu \rdr
}
Therefore, using the above,~\eqref{sin2Q}, the identity $\La Q = k \sin  Q$, and the fact that $\E( \vec \psi) = 2 \E(Q)$, we can deduce from~\eqref{enexp1} that 
\EQ{ \label{eq:enexp2}
 \int_0^\I &\psi_t^2 \rdr +  \int_0^\infty g_r^2 \rdr + k^2 \int_0^{\infty}  \frac{ \cos 2(Q_\la - Q_\mu)}{r^2} g^2  \rdr \\
 & =   \frac{4}{k^2} \int_0^\I \La Q_\la (\La Q_\mu)^3 \, \frac{\dr}{r} - \frac{2}{k^2}\int_0^\I (\La Q_\la)^2 (\La Q_\mu)^2 \, \frac{\dr}{r}  \\
 &  \quad  - 4 \int_0^\I \La Q_\la  \La Q_\mu \sin 2Q_\mu g \frac{\ud r}{r} -  \frac{4}{k^2} \int_0^\I \La Q_\la  (\La Q_\mu)^2 (rg_r) \frac{\ud r}{r} \\
 & \quad + 4 \int_0^\I \La Q_\la  \La Q_\mu \sin 2Q_\la  g \frac{\ud r}{r} + \frac{4}{k^2} \int_0^\I \La Q_\mu  (\La Q_\la)^2 (r g_r) \frac{\ud r}{r} \\
 &\quad  
 + O\left( \int_0^\infty \abs{g}^3 \, \frac{ \ud r}{r}\right) 
 }
 Next, we estimate each of the terms on the right-hand-side of~\eqref{eq:enexp2}. Denote $\s := \la / \mu$. We claim that first term on the right-hand-side of~\eqref{eq:enexp2} gives the leading order, i.e., we claim that 
 \EQ{ \label{eq:lead}
  \frac{4}{k^2} \int_0^\I \La Q_\la (\La Q_\mu)^3 \, \frac{\dr}{r} = 16k \s^k( 1+ O(\s^{2k})) 
 } 
 We compute, using the identity $k \sin Q_\la =\La Q_\la$, and setting $\s = \la/ \mu$, 
 \begin{align}
  \frac{4}{k^2} \int_0^\I \La Q_\la (\La Q_\mu)^3 \, \frac{\dr}{r} &=\frac{4}{k^2} \int_0^\I \La Q_\s (\La Q)^3 \, \frac{\dr}{r}  \\
 & = 64k^2 \s^k \int_0^\I \frac{ r^{4k-1}}{( \s^{2k} + r^{2k})(1+ r^{2k})^3} \dr \label{lo1}
 \end{align}
 First we estimate the contribution of the integral on the interval $[0, \s]$. Since we can assume that $\s \ll 1$, we have 
 \ant{
 \int_0^\s \frac{ r^{4k-1}}{( \s^{2k} + r^{2k})(1+ r^{2k})^3} \dr \simeq \s^{-2k}\int_0^\s r^{4k-1} \dr   \simeq  \s^{2k}
 }
 Next, we estimate the integral on $[\s, \infty]$. For $\s <r$ we have 
 \ant{
  \frac{1}{ \s^{2k} + r^{2k}}  &= \frac{1}{r^{2k}}+  \left( \frac{1}{ \s^{2k} + r^{2k}} - \frac{1}{r^{2k}} \right) \\
  & = \frac{1}{r^{2k}} + \frac{1}{r^{2k}}\left( \frac{1}{ 1 + (\s/r)^{2k}} - 1 \right) \\
  & =  \frac{1}{r^{2k}}+ \frac{1}{r^{2k}}\left( - (\s/r)^{2k} + O( (\s/r)^{4k}) \right) \\
  }
  Hence, 
  \ant{
   \int_\s^\I \frac{ r^{4k-1}}{( \s^{2k} + r^{2k})(1+ r^{2k})^3} \dr &=  \int_0^\I  \frac{r^{2k-1}}{ (1+ r^{2k})^3} \dr -  \int_0^\s  \frac{r^{2k-1}}{ (1+ r^{2k})^3} \dr + O( \s^{2k}) \\
   & =  \int_0^\I  \frac{r^{2k-1}}{ (1+ r^{2k})^3} \dr + O( \s^{2k}) \\
   & = \frac{1}{4k} + O( \s^{2k})
      }
      where the integral on the second to  last line can be computed explicitly by contour integration. Inserting the above into the last line of~\eqref{lo1} yields~\eqref{eq:lead}.  

Next we observe that all of the remaining terms on the right-hand-side of~\eqref{eq:enexp2} are $o( \s^{k})$. Indeed, a similar computation to the one performed above yields, 
\EQ{
\int_0^\I (\La Q_\la)^2 (\La Q_\mu)^2 \, \frac{\dr}{r} \lesssim \s^{2k} \abs{\log \s}
}
Moreover since $\|g \|_{L^\infty} \lesssim \|g \|_H$ 
we have 
\EQ{
\int_0^\infty \abs{g}^3 \, \frac{ \ud r}{r} \lesssim \|g\|_{H}^{3}   \lesssim  \sqrt{\eta_0} \| g\|_{H}^2 
}
And the remaining terms in~\eqref{eq:enexp2} can be controlled by a combination of these last two estimates together with Cauchy-Schwarz. Therefore as long as $\eta_0$ is small enough and since  $\s = \la/ \mu \lesssim \eta_0$ we have 
\EQ{ \label{eq:enexp3} 
\int_0^\infty &\psi_t^2\,   r \, \ud r +  \int_0^\infty g_r^2 \rdr + k^2 \int_0^{\infty}  \frac{ \cos 2(Q_\la - Q_\mu)}{r^2} g^2  r \ud r  = 16k \s^k - O( \s^{\frac{3k}{2} }\abs{\log \s})
}
To complete the proof of~\eqref{eq:gH} we claim the following coercivity statement: there exists a uniform constant  $c >0$ so that 
\EQ{ \label{eq:coercivity} 
 \int_0^\infty g_r^2 \rdr + k^2 \int_0^{\infty}  \frac{ \cos 2(Q_\la - Q_\mu)}{r^2} g^2  r \ud r  \ge c  \| g \|_H^2
}
for all $g \in H$ such that~\eqref{eq:ola}~\eqref{eq:omu} hold and such that ~$\|g \|_H$ is small enough. This is a standard consequence of the orthogonality conditions~\eqref{eq:ola},~\eqref{eq:omu} and the smallness of $\la/ \mu, \|g\|_H$ and we refer the reader to~\cite[Lemma 5.4]{JJ-AJM} for a detailed proof.

The left inequality in \eqref{eq:gdotgd} is trivial and the right inequality follows from \eqref{eq:gH} and \eqref{eq:g-small-6666}.
\end{proof}

\subsection{Dynamical control of the modulation parameters}\label{s:modest} 

In this section we obtain precise control of the evolution of the modulation parameters $\la(t), \mu(t)$ on any time interval $J$ on which $\bfd_+(\vec \psi(t))$ is small. We'll show that any solution $\vec \psi(t)$ that lies within a small enough $\eps$-neighborhood of a $2$-bubble at some time $t_0$, must be ejected from this $\eps$-neighborhood in at least one time direction.

This ejection happens by a defocalisation of the more concentrated bubble $Q_\lambda$
until its scale becomes comparable with the less concentrated bubble $Q_\mu$
(which does not change in the process).
The influence of the bubble $Q_\mu$ on the evolution of $Q_\lambda$ is reflected
in the time derivative of the function $b(t)$ defined in \eqref{eq:bdef} below.
Indeed, the main term of $b'(t)$ is given precisely by the interaction between
the two bubbles, see \eqref{eq:b'lb}. 
The main term of $b(t)$ is related to $\lambda'(t)$.
Hence $b'(t)$ is related to $\lambda''(t)$
so that the interaction influences the acceleration,
as it should be expected.

In this subsection we define a truncated virial functional and state some estimates related to it. The same functional was used crucially in the two-bubble construction by the first author in~\cite{JJ-AJM}.
For the proofs of the following statements we refer the reader to~\cite[Lemma 4.6]{JJ-AJM} and~\cite[Lemma 5.5]{JJ-AJM}. 

\begin{lem} \emph{\cite[Lemma 4.6]{JJ-AJM}}
  \label{lem:fun-q}
  For each $c, R > 0$ there exists a function $q(r) = q_{c, R}(r) \in C^{3,1}((0, +\infty))$ with the following properties:
  \begin{enumerate}[label=(P\arabic*)]
    \item $q(r) = \frac{1}{2} r^2$ for $r \leq R$, \label{enum:approx-q}
    \item there exists $\ti R = \ti R(R, c)> R$ such that $q(r) \equiv \tx{const}$ for $r \geq \ti R$, \label{enum:support-q}
    \item $|q'(r)| \lesssim r$ and $|q''(r)| \lesssim 1$ for all $r > 0$, with constants independent of $c, R$, \label{enum:gradlap-q}
    \item $q''(r) \geq -c$ and $\frac 1r q'(r) \geq -c$, for all $r > 0$, \label{enum:convex-ym}
    \item $(\frac{\vd^2}{\vd r^2} + \frac 1r \dd r)^2 q(r) \leq c\cdot r^{-2}$, for all $r > 0$, \label{enum:bilapl-ym}
    \item $\big|r\big(\frac{q'(r)}{r}\big)'\big| \leq c$, for all $r > 0$. \label{enum:multip-petit-ym}
  \end{enumerate}
\end{lem}
For each $\lambda > 0$ we define the operators $\A(\lambda)$ and $\A_0(\lambda)$ as follows:
\begin{align}
  [\A(\lambda)g](r) &:= q'\big(\frac{r}{\lambda}\big)\cdot \p_r g(r), \label{eq:opA-wm} \\
  [\A_0(\lambda)g](r) &:= \big(\frac{1}{2\lambda}q''\big(\frac{r}{\lambda}\big) + \frac{1}{2r}q'\big(\frac{r}{\lambda}\big)\big)g(r) + q'\big(\frac{r}{\lambda}\big)\cdot\p_r g(r). \label{eq:opA0-wm}
\end{align}
Note the similarity between $\A$ and $\frac{1}{\la} \La$ and between $\A_0$ and $\frac{1}{\la} \La_0$. 
For technical reasons we introduce the space 
\EQ{
X:= \{ g \in H \mid \frac{g}{r}, \p_r g \in H\}
}
Let 
$
f(\rho):= \frac{k^2}{2} \sin 2 \rho
$
denote the nonlinearity in~\eqref{eq:wmk}. 
\begin{lem} \emph{\cite[Lemma 5.5]{JJ-AJM}}
  \label{lem:op-A-wm}
  Let $c_0>0$ be arbitrary. There exists $c>0$ small enough and $R, \ti R>0$ large enough in Lemma~\ref{lem:fun-q} so that the operators $\A(\lambda)$ and $\A_0(\lambda)$ defined in~\eqref{eq:opA-wm} and~\eqref{eq:opA0-wm} have the following properties:
  \begin{itemize}[leftmargin=0.5cm]
    \item the families $\{\A(\lambda): \lambda > 0\}$, $\{\A_0(\lambda): \lambda > 0\}$, $\{\lambda\partial_\lambda \A(\lambda): \lambda > 0\}$
      and $\{\lambda\partial_\lambda \A_0(\lambda): \lambda > 0\}$ are bounded in $\mathscr{L}(H; L^2)$, with the bound depending only on the choice of the function $q(r)$,
    \item 
    For all $\lambda > 0$ and $g_1, g_2 \in X$  there holds
      \begin{multline}  \label{eq:A-by-parts-wm}
      \Big| \ang{ \A(\lambda)g_1\mid  \frac{1}{r^2}\big(f(g_1 + g_2) - f(g_1) - f'(g_1)g_2\big)}  \\ +\ang{ \A(\lambda)g_2\mid \frac{1}{r^2}\big(f(g_1+g_2) - f(g_1) -k^2 g_2\big)}\Big| 
        \leq \frac{c_0}{\lambda} \|g_2\|_H^2, 
      \end{multline}
    \item For all $g \in X$ we have  
\EQ{
        \label{eq:A-pohozaev-wm}
        \ang{\A_0(\lambda)g | \big(\partial_r^2 + \frac 1r\partial_r - \frac{k^2}{r^2}\big)g} \leq \frac{c_0}{\lambda}\|g\|_{H}^2 - \frac{1}{\lambda}\int_0^{R\lambda}\Big((\partial_r g)^2 + \frac{k^2}{r^2}g^2\Big) \udr, 
        }
        \item Moreover, for $\la, \mu >0$ with $\la/\mu \ll 1$, 
 \begin{gather}
      \label{eq:L0-A0-wm}
      \|\Lambda_0 \Lambda Q_\uln{\lambda} - \A_0(\lambda)\Lambda Q_{\lambda}\|_{L^2} \leq c_0, \\
      \label{eq:L-A-wm}
      \|\Lambda Q_\uln\lambda - \A(\lambda)Q_\lambda\|_{L^\infty} \leq \frac{c_0}{\lambda},  \\
    \| \A(\la) Q_\mu \|_{L^\infty} + \| \A_0(\la) Q_\mu \|_{L^\infty}  \lesssim \frac{1}{\la} (\la/ \mu)^k,   \label{eq:Ainfty} 
     \end{gather} 
    and, for any $g \in H$, 
    \begin{multline}   \label{eq:approx-potential-wm}
        \bigg|\int_0^{+\infty}\frac 12 \Big(q''\big(\frac{r}{\lambda}\big) + \frac{\lambda}{r}q'\big(\frac{r}{\lambda}\big)\Big)\frac{1}{r^2}\big(f({-}Q_\mu + Q_\lambda + g) - f({-}Q_\mu + Q_\lambda)-k^2g\big)g \udr \\
        - \int_0^{+\infty} \frac{1}{r^2}\big(f'(Q_\lambda)-k^2\big)g^2 \udr\bigg| \leq c_0(\|g\|_H^2 + (\lambda/\mu)^k).  
  \end{multline} 
  \end{itemize}
\end{lem}
\begin{rem}
  The conditions $g, g_1, g_2 \in X$ is required only to ensure that the left-hand-side of \eqref{eq:A-by-parts-wm} and~\eqref{eq:A-pohozaev-wm} are well defined,
  but do not appear on the right-hand-side of the estimates. Note also that in \eqref{eq:A-by-parts-wm}, \eqref{eq:A-pohozaev-wm} and \eqref{eq:approx-potential-wm}
  we have extracted the linear part of $f$. Lastly, the estimate~\eqref{eq:Ainfty} is not stated in~\cite{JJ-AJM} but follows immediately from  $P$2, $P3$ in Lemma~\ref{lem:fun-q} and the explicit formula for $Q$. 
\end{rem}


We are now ready to state the main modulation estimates. Proofs are given in Section~\ref{s:modproof}. Our first  estimate is a consequence of the orthogonality conditions~\eqref{eq:ola} and~\eqref{eq:omu}.  

\begin{prop}[Modulation Control Part 1] \label{p:modp} 
Let $\eta_0>0$ be as in Lemma~\ref{l:modeq}, let $J \subset \R$ be a time interval, and let
$\vec \psi(t)$ be a solution to~\eqref{eq:wmk} on $J$ such that 
\EQ{
\bfd(\vec \psi(t))\le \eta_0 \quad \forall t \in J.
}
Let $\la(t), \mu(t)$ be given by Lemma~\ref{l:modeq}. Then the following estimates hold for $t \in J$:
\begin{align}
\abs{\la'(t)} &\lesssim \la(t)^{\frac{k}{2}}/\mu(t)^\frac k2, \label{eq:la'} \\ 
\abs{ \mu'(t)}& \lesssim \la(t)^{\frac{k}{2}}/\mu(t)^\frac k2,  \label{eq:mu'} 
\end{align} 
\end{prop} 

The control we obtain on $\la(t), \mu(t)$ above is not sufficient for our purposes. In particular, we'd like to show that the ratio $\la(t)/ \mu(t)$ grows in a controlled fashion away from any small enough local minimum value. For this purpose we introduce a virial-type correction  $b(t)$ to $\la'(t)$.
The idea of modifying a modulation parameter by a virial term was used in \cite{JJ-NLS}
and, in a different context of minimal mass blow-up for non-homogeneous $L^2$-critical NLS,
in an earlier work of Rapha\"el and Szeftel \cite{RaSz11}. 

Given scaling parameters $\la(t), \mu(t)$ we write 
\EQ{ \label{eq:gdef} 
&g(t) := \psi(t) - Q_{\la(t)} + Q_{\mu(t)} \\
&\dot g(t):= \psi_t(t)
}
so that the vector $\vec g:= (g, \dot g)$ satisfies the system of equations
\begin{align} \label{eq:ptg} 
&\p_t g = \dot{g} + \la' \La Q_{\ula} - \mu' \La Q_{\umu} \\
& \p_t \dot g = \p_r^2 g + \frac{1}{r} \p_r g  - \frac{1}{r^2} \left( f( Q_\la - Q_\mu + g) - f(Q_\la) + f(Q_\mu)\right) \label{eq:ptgdot} 
\end{align} 
We then define the auxiliary function $b(t)$ by  
\EQ{ \label{eq:bdef} 
b(t):= - \ang{ \La Q_{\underline{\la(t)}}  \mid \dot g(t)}  - \ang{ \dot g(t) \mid \A_0(\la(t)) g(t)}
}
We'll show below that we can think of $b(t)$ as a subtle monotonic correction to the derivative $\la'(t)$.

Before stating the estimates satisfied by $b(t)$ we  record the following numbers, which can be computed using contour integration: 
\EQ{
 \| \La Q \|_{2}^2  = \frac{2 \pi i \textrm{Res}[(\La Q(z))^2 z; \, \om_k]}{1- \om_k^4}, \quad \om_k:= \exp( 2\pi i/ 4k)
}
which means that 
\EQ{
 \| \La Q \|_{2}^2  = \frac{2\pi}{\sin(\pi/k)} =: \ka = \kappa(k)>0
}
We will also use  fact that 
\EQ{ \label{eq:LaQ3}
\int_0^\I (\La Q(r))^3 \, r^{k-1} \ud r = 2k^2
}

Lastly, the modulation parameter $\la(t)$ itself is an imprecise proxy for the true dynamics because it was defined with respect to a somewhat arbitrary function $\ZZ$ as in~\eqref{eq:ZZdef}. To account for this imprecision we introduce a correction to $\la(t)$ as follows. Fix a radial cutoff  $\chi\in C^{\infty}_0(\R^2)$ such that $\chi(r) = 1$ if $r \le 1$, $\supp( \chi) \in B(0, 2)$. Define 
\EQ{ \label{eq:zetadef} 
\zeta(t) := \lambda(t) - \frac{1}{\kappa}\langle \chi_{\mu(t)}\Lambda Q_{\uln{\lambda(t)}} \mid g(t)\rangle
}
Note that $\zeta(t)$ is $C^1$ (because $\partial_t g(t)$ is continuous in $L^2$ with respect to $t$).

\begin{prop}[Modulation Control Part 2] \label{p:modp2} Fix $k \ge 2$. 
Assume the same hypothesis as in  Proposition~\ref{p:modp}. Let $0<\de< 1- 2^{-\frac{1}{k-1}}$ be arbitrary and let $\eta_0$ be as in Lemma~\ref{l:modeq}.  Let $b(t)$ be as in~\eqref{eq:bdef} and let $\zeta(t)$ be as in~\eqref{eq:zetadef}. 
Then, there exists $\eta_1 = \eta_1(\de) < \eta_0$ such that  if $\bfd_+(\vec \psi(t)) \le \eta_1$ for all $t \in J$ we have 
\begin{align}
&\abs{\zeta(t)/\lambda(t) - 1} \le \de , \label{eq:bound-on-l} \\ 
& \abs{\zeta'(t) - b(t) } \le \de \la(t)^{\frac k2}/\mu(t)^{\frac k2}  \le 2 \de \zeta(t)^{\frac k2}/\mu(t)^{\frac k2},  \label{eq:kala'}  \\ 
&\abs{b(t)} \le 4\sqrt{\kappa k }  (\la(t)/\mu(t))^{\frac{k}{2}} + \de \la^{\frac{k}{2}}(t)/\mu(t)^\frac k2  \le 10 \sqrt{\kappa k } \zeta(t)^{\frac k2}/\mu(t)^{\frac k2},\label{eq:b-bound}  \
\end{align}
In addition, $b(t)$ is locally Lipschitz and the derivative $b'(t)$ satisfies
\begin{align}
&|b'(t)| \leq C_0 \lambda(t)^{k-1}/\mu(t)^k \leq 2C_0 \zeta(t)^{k-1}/ \mu(t)^k \label{eq:b'} \\
&b'(t) \ge 8k^2 \la^{k-1}(t)/\mu^k(t) - \de \la^{k-1}(t)/\mu(t)^k \ge 2k^2 \zeta(t)^{k-1}/ \mu(t)^k \label{eq:b'lb}  
\end{align} 
where $C_0 > 0$ depends only on $k$. 
\end{prop} 

\begin{rem}
If $k \geq 3$, then we can take $\cZ = \Lambda Q$ in Lemma~\ref{l:modeq} and use no cut-off function in the definition of $\zeta$. 
Then $\zeta(t) \equiv \lambda(t)$.
This fails for $k = 2$, which was pointed out to us by one of the referees.
\end{rem}

We'll deduce the following consequence of Proposition~\ref{p:modp} and Proposition~\ref{p:modp2}. 

\begin{prop}
\label{prop:modulation}
Let $C > 0$. For any $\eps_0>0$ small enough, and for all $\eps>0$ sufficiently small relative to $\eps_0$ the following conclusions hold true.   
Let $\vec\psi(t): [T_0, T_+) \to \HH_0$ be a solution of \eqref{eq:wmk}. 
Assume that $t_0 \in [T_0, T_+)$ is such that $\bfd(\vec \psi(t_0)) \leq \eps$
and $\dd t(\zeta(t)/\mu(t))\vert_{t = t_0} \geq 0$.
Then there exist $t_1$ and $t_2$, $T_0 \leq t_0 \leq t_1 \leq t_2 < T_+$,
such that
\begin{align}
\bfd(\vec \psi(t)) &\geq 2\eps,\qquad \text{for }t \in [t_1, t_2], \label{eq:d-t1-t2}\\
\bfd(\vec \psi(t)) &\leq \frac 14\eps_0,\qquad \text{for }t \in [t_0, t_1], \label{eq:d-t0-t1}\\
\bfd(\vec \psi(t_2)) &\geq 2\eps_0, \label{eq:d-at-t2} \\
\int_{t_1}^{t_2} \|\partial_t\psi(t)\|_{L^2}^2\ud t &\geq C\int_{t_0}^{t_1}\sqrt{\bfd(\vec \psi(t))}\ud t
\label{eq:err-absorb}
\end{align}
Analogous statements hold with times  $t_2  \le  t_1 \le t_0$ if $\dd t(\zeta(t)/\mu(t))\vert_{t = t_0} \le 0$. 
\end{prop}
\begin{rem}
We will take $\eps < \eta_0$, so that $\bfd(\vec\psi(t_0)) \leq \eps$
implies that the modulation parameters $\lambda(t), \mu(t)$ and also $\zeta(t)$ are well-defined $C^1$ functions in a neighborhood of $t= t_0$.

\end{rem}
\begin{rem}
We will actually deduce \eqref{eq:err-absorb} from the following stronger statement.
There exist $\eps_0, C_k$ depending only on $k$ such that
for any $\eps > 0$ small enough, Proposition~\ref{prop:modulation} holds and additionally
\begin{align}
\int_{t_0}^{t_1}\sqrt{\bfd(\vec\psi(t))}\,\ud t &\leq C_k \eps^{\frac 1k}, \label{eq:uni-borne} \\
\int_{t_1}^{t_2}\|\partial_t \psi(t)\|_{L^2}^2\,\ud t &\geq \frac{1}{C_k}.
\end{align}
\end{rem}

\subsection{Proofs of the Modulation Estimates}  \label{s:modproof} 

We first assume the conclusions of Propositions~\ref{p:modp} and ~\ref{p:modp2} and prove Proposition~\ref{prop:modulation}. 
We record here a few useful formulae:
%
\begin{align}
&\La Q:= r \p_r Q = k \sin Q  =  \frac{2k r^k}{ 1+ r^{2k}}  \label{eq:LaQ}\\
& \La^2 Q = \frac{k^{2}}{2} \sin 2 Q = 2k^2 r^k \left(\frac{ 1-r^{2k}}{ (1 + r^{2k})^2} \right) \label{eq:La2Q} \\ 
&\La^3 Q  =2k^3 r^k \left( \frac{1+ r^{2k} - 5r^{4k} - r^{6k}}{(1+r^{2k})^4}\right) \label{eq:La3Q} \\
& \La_0 \La Q = ( r\p_r +1)(r \p_r Q) = 2 \La Q + r^2 \p_r^2 Q \label{eq:DLa} \\
\end{align}

\begin{proof}[Proof of Proposition~\ref{prop:modulation}]

From \eqref{eq:gH} and~\eqref{eq:bound-on-l} it follows that there exists $\eps_1 > 0$ such that if
$\zeta(t)/\mu(t) \leq \eps_1$,  the modulation estimates hold
in a neighborhood of time $t$.
If needed, we will assume that $\eps_1$ is sufficiently small, but depending only on $k$.
Let $t_2$ be the first time $t_2 \geq t_0$ such that $\zeta(t_2) / \mu(t_2) = \eps_1$
(if there is no such time, we set  $t_2 = T_+$).
By the estimate~\eqref{eq:lamud1} in Lemma~\ref{l:modeq} along with~\eqref{eq:bound-on-l}, which accounts for the difference between $\la(t)$ and $\zeta(t)$, there exists $\eps_0$ sufficiently small such that $\zeta(t_2) / \mu(t_2) = \eps_1$
implies \eqref{eq:d-at-t2}. 
The number $\eps > 0$ will be chosen later in the proof and should
be thought of as being much smaller than $(\eps_1)^\frac k2$,
whereas we can think of $\eps_0$ as comparable to $(\eps_1)^\frac k2$.

Without loss of generality we can assume that $\mu(t_0) = 1$.
Let $t_3 \leq t_2$ be the last time such that $\mu(t) \in [\frac 12, 2]$ for all $t \in [t_0, t_3]$. If there is no such final time we set $t_3= t_2$. 
(Later we will see that we can always take $t_3 = t_2$ as long as $\eps_1>0$ is small enough.)

For $t \in [t_0, t_3]$, from \eqref{eq:b'lb} we obtain
\begin{equation}
\label{eq:b'lb2}
b'(t) \geq \frac{k^2}{2^{k-1}}\zeta(t)^{k-1}.
\end{equation}
We also obtain from \eqref{eq:kala'}
\begin{equation}
\zeta'(t) \geq \frac{1}{\kappa}b(t) - \sqrt{\frac{k}{2^{k-1}\kappa}}\zeta(t)^\frac k2.
\end{equation}
Let $\kappa_1 := \sqrt{\frac{k\kappa}{2^{k-1}}}$ and consider $\xi(t) := b(t) + \kappa_1 \zeta(t)^\frac k2$.
Using the two inequalities above we obtain
\begin{equation}
\label{eq:psi'}
\begin{aligned}
\xi'(t) &\geq \frac{k^2}{2^{k-1}}\zeta(t)^{k-1} + \kappa_1 \frac k2 \zeta(t)^{\frac k2 - 1}
\Big(\frac{1}{\kappa}b(t) - \sqrt{\frac{k}{2^{k-1}\kappa}}\zeta(t)^\frac k2\Big) \\
& = \frac{\kappa_1 k}{2\kappa}\zeta(t)^{\frac k2 - 1}b(t)
+ \Big(\frac{k^2}{2^{k-1}}-\frac{k\kappa_1}{2}\sqrt{
\frac{k}{2^{k-1}\kappa}}\Big)\zeta(t)^{k-1} \\
&= k\sqrt{\frac{k}{\kappa 2^{k+1}}}\zeta(t)^{\frac k2 - 1}b(t)
+ \frac{k^2}{2^k}\zeta(t)^{k-1} \\
&= k\sqrt{\frac{k}{\kappa 2^{k+1}}}\zeta(t)^{\frac k2 - 1}\xi(t).
\end{aligned}
\end{equation}
It is easy to compute that \eqref{eq:b-bound} yields $|b(t)| \leq \kappa_1 2^{k+4}\zeta(t)^\frac k2$,
so we have
\begin{equation}
\label{eq:psi-bound}
\xi(t) \leq \kappa_1 2^{k+5}\zeta(t)^\frac k2
\end{equation}
and \eqref{eq:psi'} leads to
\begin{equation}
\label{eq:psi'2}
\xi'(t) \geq \kappa_2 \xi(t)^\frac{2k-2}{k},
\end{equation}
for some constant $\kappa_2 > 0$ depending only on $k$.

Let $\xi_1(t) := b(t) + \frac{\kappa_1}{2}\zeta(t)^\frac k2 = \frac 12 b(t) + \frac 12 \xi(t)$.
Since $b'(t) \geq 0$, we have
\begin{equation}
\label{eq:psi1'}
\xi_1'(t) \geq \frac 12 \xi'(t) \geq \frac k2 \sqrt{\frac{k}{\kappa 2^{k+1}}}\zeta(t)^{\frac k2 - 1}\xi(t) \geq \frac k2 \sqrt{\frac{k}{\kappa 2^{k+1}}}\zeta(t)^{\frac k2 - 1}\xi_1(t).
\end{equation}

Since $\mu(t_0) = 1$, we have $0 \leq \dd t(\lambda(t)/\mu(t))\vert_{t = t_0} = \zeta'(t_0) - \zeta(t_0)\mu'(t_0)$, so \eqref{eq:mu'} and~\eqref{eq:bound-on-l} imply that $\kappa \zeta'(t_0) \geq -\frac{\kappa_1}{4}\zeta(t_0)^\frac k2$ as long as $\eps$ is taken small enough. Now \eqref{eq:kala'} gives $b(t_0) \geq -\frac{\kappa_1}{3}\zeta(t_0)^\frac k2$,
so $\xi_1(t_0) > 0$ and \eqref{eq:psi1'} yields $\xi_1(t) > 0$ for all $t \in [t_0, t_3]$.
Thus
\begin{equation}
\label{eq:psi-lbound}
\xi(t) \geq \frac{\kappa_1}{2}\zeta(t)^\frac k2, \qquad \text{for }t \in [t_0, t_3].
\end{equation}
In particular, \eqref{eq:psi'2} implies that $\xi(t)$ is strictly increasing on $[t_0, t_3]$
and by \eqref{eq:psi-bound} we see that  $\zeta(t)$ and thus $\la(t)$  is far from $0$ on $[t_0, t_3]$.

Bounds \eqref{eq:psi-bound} and~\eqref{eq:bound-on-l} imply that there exists a constant $\kappa_3$ depending only on $k$
such that $\xi(t) \geq \kappa_3 \sqrt{\eps}$ forces $\bfd(\vec \psi(t)) \geq 2\eps$.
Let $t_1 \in [t_0, t_3]$ be the last time such that $\xi(t_1) = \kappa_3 \sqrt{\eps}$
(set $t_1 = t_3$ if no such time exists).
Then by \eqref{eq:psi-lbound} and~\eqref{eq:bound-on-l}  we have
\EQ{
\frac{1}{2} \lambda(t)^\frac k2  \le \zeta(t)^{\frac{k}{2}} \leq \frac{2\kappa_3}{\kappa_1}\sqrt\eps\quad\text{for }t\in[t_0, t_1],
}
which yields \eqref{eq:d-t0-t1} if $\eps$ is small enough.

\textbf{Case $k = 2$.}
In this case \eqref{eq:psi'2} reads
\begin{equation}
\label{eq:psi'3}
\xi'(t) \geq \kappa_2 \xi(t).
\end{equation}
Integrating between $t$ and $t_3$ we get $\xi(t) \leq \eee^{\kappa_2(t - t_3)}\xi(t_3)$.
Thus \eqref{eq:psi-lbound} and \eqref{eq:psi-bound} yield
\begin{equation}
\zeta(t) \leq \kappa_4 \eee^{\kappa_2(t - t_3)}\zeta(t_3) \leq 2 \kappa_4 \eee^{\kappa_2(t - t_3)}\eps_1,
\end{equation}
with a universal constant $\kappa_4 > 0$. Thus integrating \eqref{eq:mu'} and using $\mu(t_0) = 1$
we get $\mu(t_3) \in [2/3, 3/2]$ if $\eps_1$ is small enough, which implies that $t_3 = t_2$.
Also, suppose that there is no $t_2 \geq t_0$ such that $\zeta(t_2)/\mu(t_2) = \eps_1$.
Then, since $\zeta(t)$ (and hence $\la(t)$) is far from $0$, by known arguments,
see for instance \cite[Corollary A.4]{JJ-AJM},
the solution is global and \eqref{eq:psi'2} implies that $\xi(t)$ is unbounded.
Thus $\lambda(t)$ is also unbounded, which is a contradiction.
We infer that there must be $t_2 < T_+$ such that $\zeta(t_2)/\mu(t_2) = \eps_1$,
which implies \eqref{eq:d-at-t2} by choosing $\eps_0$ comparable to $(\eps_1)^{\frac k2}$.

We have $|\zeta'(t)| \lesssim |\zeta(t)|$, see \eqref{eq:kala'} and~\eqref{eq:b-bound},
hence there exists a constant $\kappa_5$
such that $\zeta(t) \geq \frac 14 \eps_1$ for $t \in [t_2 - \kappa_5, t_2]$.
Thus \eqref{eq:b'lb2} yields
\begin{equation}
b(t) - b(t_0) \geq \kappa_6 (t - (t_2 - \kappa_5)) \eps_1, \qquad \text{for }t \in [t_2 - \kappa_5, t_2].
\end{equation}
Thus, if $\eps$ is small enough, we get
$b(t) \geq \kappa_7 \eps_1$ for $t \in [t_2 - \frac 12 \kappa_5, t_2]$.
Note that $\kappa_7$ is independent of $\eps_1$.
From the definition of $b(t)$ and the Cauchy-Schwarz inequality we can deduce, if $\eps_1$ is small enough,
that $\|\dot g(t)\|_{L^2} \geq \kappa_8 \eps_1$ for $t \in [t_2 - \frac 12 \kappa_5, t_2]$, which leads to
\begin{equation}
\label{eq:err-lbound}
\int_{t_2 - \frac 12 \kappa_5}^{t_2} \|\dot g(t)\|_{L^2}^2 \ud t \geq \kappa_9 \eps_1^2.
\end{equation}
Integrating \eqref{eq:psi'3} between $t$ and $t_1$ and using \eqref{eq:psi-lbound}, \eqref{eq:psi-bound}
and the definition of $t_1$ we obtain
\begin{equation}
\frac{1}{2} \la(t) \le \zeta(t) \leq \kappa_{10} \eee^{\kappa_2(t-t_1)}\sqrt{\eps}, \qquad \text{for }t \in [t_0, t_1].
\end{equation}
Thus
\begin{equation}
\label{eq:err-bound}
\int_{t_0}^{t_1}\sqrt{\bfd(\vec \psi(t))}\ud t \leq \kappa_{11}\sqrt{\eps}.
\end{equation}
Comparing this bound with \eqref{eq:err-lbound} and choosing $\eps$ small enough,
we get \eqref{eq:err-absorb}.


\textbf{Case $k \geq 3$.}
Most of the argument can be repeated without major changes. We can rewrite \eqref{eq:psi'2} as
\begin{equation}
\label{eq:psi'3k}
\big(\xi(t)^{\frac{2-k}{k}}\big)' \leq -\frac{(k-2)\kappa_2}{k}.
\end{equation}
Integrating and using \eqref{eq:psi-lbound}, \eqref{eq:psi-bound} we obtain
\begin{equation}
\label{eq:lambda-bound-k}
\zeta(t)^\frac k2 \leq \kappa_4 \big(\zeta(t_3)^\frac{2-k}{2} + (t_3 - t)\big)^\frac{k}{2-k},
\end{equation}
with $\kappa_4$ only depending on $k$. Thus
\begin{equation}
\frac{1}{2}\int_{t_0}^{t_3}\lambda(t)^\frac k2 \ud t  \le\int_{t_0}^{t_3}\zeta(t)^\frac k2 \ud t \leq \kappa_4 \int_{\zeta(t_3)^{\frac{2-k}{2}}}^{+\infty}
\tau^\frac{k}{2-k}\ud \tau \leq \frac{k-2}{2}\kappa_4 \lambda(t_3) \leq (k-2)\kappa_4 \eps_1.
\end{equation}
As in the case $k = 2$, we can deduce that $t_3 = t_2$ and that $\zeta(t_2)/\mu(t_2) = \eps_1$,
which implies \eqref{eq:d-at-t2}.

The proof of \eqref{eq:err-lbound} applies without significant changes and yields
\begin{equation}
\label{eq:err-lbound-k}
\int_{t_2 - \frac 12 \kappa_5}^{t_2} \|\dot g(t)\|_{L^2}^2 \ud t \geq \kappa_9 \eps_1^{2k-2}.
\end{equation}
The proof of \eqref{eq:lambda-bound-k} yields
\begin{equation}
\frac{1}{2} \lambda(t)^\frac k2 \le \zeta(t)^{\frac{k}{2}} \leq \kappa_{10}\big(\eps^\frac{2-k}{2k} + (t_1 - t)\big)^\frac{k}{2-k}, \qquad \text{for }t \in [t_0, t_1].
\end{equation}
After integrating, this implies
\begin{equation}
\label{eq:err-bound-k}
\int_{t_0}^{t_1}\sqrt{\bfd(\vec \psi(t))}\ud t \leq \kappa_{11}\eps^\frac 1k,
\end{equation}
With $\kappa_{11}$ depending only on $k$.
Comparing this bound with \eqref{eq:err-lbound-k} and choosing $\eps$ small enough,
we get \eqref{eq:err-absorb}.
\end{proof}

\begin{proof}[Proof of Proposition~\ref{p:modp}]
Let $t_0 \in J$ be any point in $J$. By rescaling  $ \vec \psi(t) \mapsto \vec \psi(t)_{\mu(t_0)^{-1}}$ we can assume without loss of generality that $\mu(t_0) = 1$.
We can also assume that
\EQ{ \label{eq:mu12} 
\frac{1}{2} \le \mu(t) \le 2 
}
for all $t \in J$ (we work in a small neighborhood of $t_0$).

We begin by differentiating the orthogonality conditions~\eqref{eq:ola} and~\eqref{eq:omu} to derive a linear system for $(\la', \mu')$.

Differentiating~\eqref{eq:ola} yields 
 \EQ{
  0 &= \dd t \ang{ \ZZ_{\ula} \mid g} = - \lambda' \ang{\frac{1}{\la} [\La_0 \ZZ]_{\ula} \mid  g} + \ang{\ZZ_{\ula} \mid  \p_t g}  
  }
  Plugging in~\eqref{eq:ptg}  above and rearranging  we have 
  \EQ{
 - \ang{ \ZZ_{\ula} \mid \dot g} =    \la' \left(\ang{\ZZ_{\ula} \mid \La Q_{\ula}} -  \ang{\frac{1}{\la} [\La_0 \ZZ]_{\ula} \mid  g} \right) -  \mu' \ang{ \ZZ_{\ula} \mid \La Q_{\umu}} 
  }

Differentiating~\eqref{eq:omu} yields 
 \EQ{
  0 &= \dd t \ang{ \ZZ_{\umu} \mid g} = - \mu' \ang{\frac{1}{\mu} [\La_0 \ZZ]_{\umu} \mid  g} + \ang{\ZZ_{\umu} \mid  \p_t g}  
  }
  Plugging in~\eqref{eq:ptg}  above and rearranging  we have 
  \EQ{
 - \ang{ \ZZ_{\umu} \mid \dot g} =    \la' \ang{\ZZ_{\umu} \mid \La Q_{\ula}} +   \mu' \left({-} \ang{ \ZZ_{\umu} \mid \La Q_{\umu}} -  \ang{\frac{1}{\mu} [\La_0 \ZZ]_{\umu} \mid  g} \right) 
  }
  We then arrive at the following linear system for $(\la', \mu')$, 
  \EQ{
  \pmat{ M_{11} & M_{12} \\ M_{21} & M_{22}} \pmat{ \la' \\ \mu'} =  \pmat{ {-}\ang{\ZZ_{\ula} \mid  \dot g} \\ -  \ang{ \ZZ_{\umu} \mid \dot g} }
  }
  where
  \EQ{ \label{eq:Mmat} 
 & M_{11} :=  \ang{\ZZ_{\ula} \mid \La Q_{\ula}} -  \ang{\frac{1}{\la} [\La_0 \ZZ]_{\ula} \mid  g}  \\
  & M_{12} := -\ang{ \ZZ_{\ula} \mid \La Q_{\umu}}  \\ 
  &M_{21}:= \ang{\ZZ_{\umu} \mid \La Q_{\ula}}  \\
  & M_{22}:= - \ang{ \ZZ_{\umu} \mid \La Q_{\umu}} -  \ang{\frac{1}{\mu} [\La_0 \ZZ]_{\umu} \mid  g}
  }
 We first claim that $M = (M_{ij})$ is diagonally dominant with coefficients of size $\simeq 1$ on the diagonal. This will allow us to invert $M$ and estimate $\la', \mu'$.  Indeed, in Claim~\ref{c:M12est} we showed that the off-diagonal terms $M_{12}$ and $M_{21}$ satisfy $\abs{M_{12}} \lesssim \la^{k+1}$ and $\abs{M_{21}} \lesssim \la^{k-1}$. To estimate the diagonal terms define $ \be := \ang{  \ZZ \mid \La Q}$ and note that $\be>0$ is a fixed positive number by~\eqref{eq:ZZdef}.  Then 
  \EQ{
  \abs{M_{11} - \be } + \abs{ M_{22} + \be } \lesssim \la^{\frac{k}{2}}
  }
  To see this, note that by the definitions of $M_{11}, M_{22}$ and the fact that $\ZZ \in C^{\infty}_0$ we have 
  \EQ{
   \abs{M_{11}-\be} + \abs{ M_{22} + \be } \le  \abs{\ang{ \frac{r}{\la} [\La_0 \ZZ]_{\ula} \mid r^{-1} g}} + \abs{\ang{ \frac{r}{\mu}[\La_0 \ZZ]_{\umu} \mid r^{-1}g}} \lesssim \|g \|_{H} \lesssim \la^{\frac{k}{2}}
   }
   where the last inequality above follows from~\eqref{eq:gH}.

    We solve for  $(\la', \mu')$, by inverting $M$, 
   \EQ{
   \pmat{  \la' \\ \mu ' } =  \frac{1}{\det M}  \pmat{ - M_{22} \ang{ \ZZ_{\ula} \mid  \dot g} + M_{12} \ang{\ZZ_{\umu} \mid  \dot g}  \\  M_{21}  \ang{\ZZ_{\ula} \mid  \dot g}- M_{11}  \ang{\ZZ_{\umu} \mid  \dot g}}
   }
   Now note that by~\eqref{eq:gH} we have 
   \EQ{ \label{eq:rsest} 
   &\abs{ \ang{\ZZ_{\ula} \mid  \dot g}} + \abs{ \ang{\ZZ_{\umu} \mid  \dot g}} \lesssim \| \dot g \|_{L^2} \lesssim \la^{\frac{k}{2}}
   }
   Our estimates for the coefficients of $M$ imply  that 
   \EQ{
   \det M =   M_{11}M_{22} + O( \la^{2k}) , \quad
   \frac{1}{\det M}  = \frac{1}{M_{11} M_{22}} + O( \la^{2k})
   }
   Using the above we now write $\la'$ and $\mu'$ as follows, 
   \EQ{
   \la' =  \left( \frac{1}{M_{11} M_{22}} + O( \la^{2k}) \right) \left( - M_{22} \ang{ \ZZ_{\ula} \mid  \dot g} + M_{12} \ang{\ZZ_{\umu} \mid  \dot g}  \right)  \\
   }
and thus using~\eqref{eq:M12} and~\eqref{eq:rsest} we conclude that 
\EQ{ \label{la'est}
\abs{\la'} \lesssim \la^{\frac{k}{2}}
}
Similarly, for $\mu'$ we have 
\EQ{
 \mu' =  \left( \frac{1}{M_{11} M_{22}} + O( \la^{2k}) \right) \left(  M_{21} \ang{ \ZZ_{\ula} \mid  \dot g} - M_{11} \ang{\ZZ_{\umu} \mid  \dot g}  \right) 
}
and hence, 
\EQ{ 
\abs{ \mu'} \lesssim \la^{\frac{k}{2}}
}
which proves~\eqref{eq:mu'} and  completes the proof of Proposition~\ref{p:modp}. 
 \end{proof} 
 
%

\begin{rem}\label{r:ODE} 
We remark here that $\la(t)$, $\mu(t)$ obtained in the proof of Lemma~\ref{l:modeq} can be easily seen to be  $C^1$ functions. Recall the ODE 
\EQ{\label{eq:lamuODE} 
  \pmat{ M_{11} & M_{12} \\ M_{21} & M_{22}} \pmat{ \la' \\ \mu'} =  \pmat{ -\ang{\ZZ_{\ula} \mid  \psi_t(t)} \\ -  \ang{ \ZZ_{\umu} \mid \psi_t(t)} }, 
  }
 obtained by formally differentiating the orthogonality conditions~\eqref{eq:ola} and~\eqref{eq:omu}; the coefficients $M_{ij}$ are given explicitly in~\eqref{eq:Mmat}. For any $t_0 \in J$ the smallness of $\la(t_0)/ \mu(t_0)$ guarantees the existence of a unique $C^1$ solution  $(\ti \la(t), \ti\mu(t))$ with initial data $(\ti \la, \ti \mu)(t_0) = (\la(t_0), \mu(t_0))$ in a neighborhood of $t_0$. Because of how the system~\eqref{eq:lamuODE} was derived, $\ti \la(t), \ti \mu(t)$ and $g(t) := \psi(t) - (Q_{\ti \la(t)} + Q_{\ti \mu(t)})$ satisfy~\eqref{eq:ola}~\eqref{eq:omu} and~\eqref{eq:gdotgd} in a small enough neighborhood of $t_0$. Since the $\la(t), \mu(t)$ obtained by the implicit value theorem are unique with these properties we have $\la(t) = \ti \la(t)$ and $\mu(t) = \ti \mu(t)$ proving that $\la(t), \mu(t)$ are indeed $C^1$. 
%
  
\end{rem}

\begin{proof}[Proof of Proposition~\ref{p:modp2}]

As in the proof of Proposition~\ref{p:modp} we can assume that $\frac 12 \le \mu(t) \le 2$ below. 

We first prove \eqref{eq:bound-on-l}. From \eqref{eq:gH} we have $\|g\|_{L^\infty} \lesssim \lambda^\frac k2$, so it suffices
to check that
\begin{equation}
\label{eq:bound-on-l-2}
\|\chi_\mu \Lambda Q_{\ula}\|_{L^1(r\ud r)}  \ll \lambda^{1-\frac k2} \mas \la \to 0 
\end{equation}
which follows from
\begin{equation}
\|\chi_\mu \Lambda Q_{\ula}\|_{L^1(r\ud r)} \leq \lambda \int_0^{4/\lambda}\Lambda Q(r)r\ud r
\lesssim  \la \int_0^{4/\lambda} (1+r)^{-k+1}\ud r \lesssim  \la \abs{\log(4/\lambda)} 
\end{equation}

Now we prove~\eqref{eq:kala'}.
From \eqref{eq:ptg} we have
\begin{equation}
\label{eq:bound-on-l-666}
\begin{aligned}
\dd t\ang{\chi_\mu \Lambda Q_{\ula} \mid g} &= \ang{\chi_\mu \Lambda Q_{\ula}\mid \dot g} + \lambda'\ang{ \chi_\mu \Lambda Q_{\ula}\mid \Lambda Q_{\ula}}
- \mu'\ang{\chi_\mu \Lambda Q_{\ula}\mid \Lambda Q_{\umu}} \\
&-\frac{\lambda'}{\lambda}\ang{\chi_\mu \Lambda_0\Lambda Q_{\ula}, g} - \mu'\ang{ \Lambda \chi_{\umu}\Lambda Q_{\ula}\mid g},
\end{aligned}
\end{equation}

Most terms on the right hand side are negligible (that is, $\ll \lambda^\frac k2$).
  Since $\la \ll 1$, we have
  \EQ{
  \begin{aligned}
  \label{eq:1-chi-LQ}
 \|(1- \chi_\mu) \La Q_{\ula}\|_{L^2}^2 &\lesssim \int_{1/2}^\infty (\La Q_{\ula})^2 \, r \ud r  = \int_{\frac 12 \la^{-1}}^{\infty}  (\La Q)^2 \, r \ \ud r \\
 &\lesssim \int_{\frac 12 \la^{-1}}^\infty r^{-2k+1} \, \ud r \lesssim \la^{2k-2}.
 \end{aligned}
  }
  Together with \eqref{eq:gH} this yields
  \begin{equation}
  \abs{ \ang{(1-\chi_\mu) \La Q_{\ula} \mid \dot g}} \lesssim \la^{k-1}\la^\frac k2 \ll \lambda^\frac k2,
  \end{equation}
  so in the first term we can erase $\chi_\mu$.
 Similarly, from \eqref{eq:1-chi-LQ} and \eqref{eq:la'} we have $\abs{\lambda'\ang{(1-\chi_\mu)\Lambda Q_{\ula}\mid
\Lambda Q_{\ula}}} \ll \lambda^\frac k2$, so $\chi_\mu$ can be erased also in the second term.

Regarding the third term,
 since we are assuming  $\frac 12 \le  \mu \leq 2$, we have 
  \EQ{
  \abs{\ang{\chi_\mu \Lambda Q_{\ula}\mid \Lambda Q_{\umu}}}& \lesssim \int_0^4 \frac{1}{\la} \frac{r}{\la} Q_r(r/ \la) \frac{1}{\mu} \frac{r}{\mu}  Q_{r}(r/ \mu) \, r \ud r \\
  & \lesssim \frac{1}{\la} \int_0^4  \frac{ (r/ \la)^{k}}{ 1+ (r/\la)^{2k}} \frac{r^{k+1}}{ 1+ r^{2k}} \, \ud r \\
  & \lesssim \la^{k-1} \int_0^\la  \frac{ r^{2k+1}}{ \la^{2k} + r^{2k}} \frac{1}{ 1+ r^{2k}} \, \ud r  +  \la^{k-1} \int_\la^4  \frac{ r^{2k+1}}{ \la^{2k} + r^{2k}} \frac{1}{ 1+ r^{2k}} \, \ud r\\
  }
  To estimate the first integral on the right above on the interval $[0, \la]$ we note that since $\la \ll 1$ we have 
  \EQ{
   \la^{k-1} \int_0^\la  \frac{ r^{2k+1}}{ \la^{2k} + r^{2k}} \frac{1}{ 1+ r^{2k}} \, \ud r \lesssim \la^{k-1}  \la^{-2k} \int_0^\la r^{2k+1} \,  \dr \lesssim \la^{k+1} 
  }
  On the interval $[\la, 4]$ we write 
  \EQ{
  \frac{1}{\la^{2k} + r^{2k}}  &= \frac{1}{r^{2k}} +  \left( \frac{1}{\la^{2k} + r^{2k}}  - \frac{1}{r^{2k}} \right) \\
  & =  \frac{1}{r^{2k}} +  \frac{1}{r^{2k}} \left( \frac{1}{1 + (\la/r)^{2k} }  - 1\right)  \\ 
  &= \frac{1}{r^{2k}} + \frac{1}{r^{2k}} \left( -\la^{2k}r^{-2k}   + O( \la^{4k}r^{-4k})\right)
  }
  This yields, 
  \EQ{
  \la^{k-1} \int_\la^4  \frac{ r^{2k+1}}{ \la^{2k} + r^{2k}} \frac{1}{ 1+ r^{2k}} \, \ud r & =   \la^{k-1} \int_0^{\infty} \frac{r}{1+ r^{2k}} \, \dr \\
  & \quad  - \la^{k-1}\int_0^\la   \frac{r}{1+ r^{2k}} \, \dr + O( \la^{k+1}) \\
    & = C \la^{k-1} + O( \la^{k+1})
  }  
 This and \eqref{eq:mu'} imply that the third term of the right hand side in \eqref{eq:bound-on-l-666}
 is negligible.

As for the fourth term, we have
\begin{equation}
\Big|\frac{\lambda'}{\lambda}\ang{\chi_\mu \Lambda_0\Lambda Q_{\ula}, g}\Big| \lesssim |\lambda'|\|g\|_{L^\infty}\|\chi_{\mu/\lambda}\Lambda_0\Lambda Q\|_{L^1} \lesssim \lambda^k\|\chi_{2/\lambda}\Lambda_0\Lambda Q\|_{L^1},
\end{equation}
which is $\ll \lambda^\frac k2$, see the proof of \eqref{eq:bound-on-l-2}.
The fifth term is even smaller (we gain an additional factor $\lambda$).

Summarizing, from the definition of $\zeta(t)$ and \eqref{eq:bound-on-l-6666}, we obtain
\begin{equation}
\label{eq:bound-on-l-6666}
|\kappa\zeta' - \ang{\Lambda Q_{\ula}\mid \dot g}| \ll \lambda^\frac k2.
\end{equation}

Recall that 
 \EQ{
 b(t):= - \ang{ \La Q_{\ula} \mid  \dot g}  - \ang{ \dot g \mid \A_0(\la) g}
}
 By~\eqref{eq:gH} and the fact that $\A_0: H \to L^2$ is bounded independently of $\la$ -- see Lemma~\ref{lem:op-A-wm} --  we have
\EQ{  \label{eq:dotgAg} 
 \ang{ \dot g \mid \A_0(\la) g} \lesssim  \| \dot g \|_{L^2} \| \A_0(\la) g \|_{L^2} \lesssim \| (g, \dot g) \|_{\HH_0}^2 \lesssim \la^k \ll \lambda^\frac k2.
 }
Together with~\eqref{eq:bound-on-l-6666} this  means that 
\EQ{
\abs{\kappa\zeta' - b}  \ll \lambda^\frac k2,
}
which is~\eqref{eq:kala'}. 

Arguing as above we have 
\EQ{
\abs{b(t)} \le \| \La Q_{\ula}\|_2 \| \psi_t\|_{2}  - O(\la^k) = \sqrt{\ka} \| \psi_t \|_2 - O(\la^k)
} 
From the expansion of the nonlinear energy in the proof of~\eqref{eq:gH} and our assumption that $\mu(t) \simeq  1$ we see that 
\EQ{
\| \psi_t(t) \|_2^2 \le 16k (\la/ \mu)^k  + o(\la^k)
}
Plugging this in above yields~\eqref{eq:b-bound}.

Finally, we  begin the delicate proof of~\eqref{eq:b'lb}; we note that~\eqref{eq:b'} will also be a consequence of this analysis. It is sufficient to prove the result for smooth solutions.
Indeed, we can then use a standard approximation procedure.
We approximate a solution $\vec\psi:J\to \HH_0$ by a sequence of smooth solutions $\vec\psi_n$.
Then $b_n(t)$ converges to $b(t)$ uniformly for $t\in J$, and we can pass to a limit in \eqref{eq:b'} and \eqref{eq:b'lb}. Differentiating $b(t)$ and recalling the formulae~\eqref{eq:ptg},~\eqref{eq:ptgdot} we have 
\EQ{ \label{eq:b'1} 
b'(t) &= \frac{\la'}{\la} \ang{[ \La_0 \La Q]_{\ula} \mid \dot g }   - \ang{ \La Q_{\ula}  \mid \p_t \dot g} 
 - \ang{\p_t  \dot g \mid \A_0( \la) g}  \\ 
& \quad - \frac{\la'}{\la}\ang{ \dot g \mid  (\la \p_\la \A_0(\la)) g}  -  \ang{ \dot g \mid \A_0(\la) \p_t g} \\ 
& = \frac{\la'}{\la} \ang{[ \La_0 \La Q]_{\ula} \mid \dot g }   \\
 &\quad -  \ang{ \La Q_{\ula} \mid  \p_r^2 g + \frac{1}{r} \p_r g  - \frac{1}{r^2} \left( f( Q_\la - Q_\mu + g) - f(Q_\la) +f(Q_\mu)\right)} \\
& \quad - \ang{ \p_r^2 g + \frac{1}{r} \p_r g  - \frac{1}{r^2} \left( f( Q_\la - Q_\mu + g) - f(Q_\la) + f(Q_\mu)\right) \mid \A_0(\la) g}  \\
& \quad - \frac{\la'}{\la}\ang{ \dot g \mid  (\la \p_\la \A_0(\la)) g} - \ang{ \dot g \mid \A_0(\la) \dot g} -  \la' \ang{ \dot g \mid \A_0(\la) \La Q_{\ula}} \\
&\quad + \mu' \ang{ \dot g \mid \A_0(\la) \La Q_{\umu}}
}
Let us first identify terms above that we've already established to be $\ll \la^{k-1}$ and discard them. 
First note that since $(\la \p_\la \A_0(\la)): H \to L^2$ is bounded, and since we've already shown $\abs{\la'} \lesssim \la^{\frac{k}{2}}$ we have 
\EQ{
\frac{\la'}{\la}\ang{ \dot g \mid  (\la \p_\la \A_0(\la)) g} \lesssim  \|(g, \dot g) \|_{\HH_0}^2 \lesssim \la^k \ll \la^{k-1}
}
Then we note that 
\EQ{
\ang{ \dot g \mid \A_0(\la) \dot g} = 0, 
}
which can be shown directly by integration by parts. 
Next, using the fact that $\mu   \simeq1$, along with the boundedness of $\A_0(\la): H \to L^2$  we have 
\EQ{
\mu' \ang{ \dot g \mid \A_0(\la) \La Q_{\umu}} &= \frac{\mu'}{\mu} \ang{ \dot g \mid \A_0(\la) \La Q_{\mu}} \\
& \lesssim \abs{ \mu'} \|  \dot g\|_{L^2}  \| \La Q_{\mu} \|_{H} \lesssim  \abs{ \mu'} \| \dot g \|_{L^2} \lesssim  \la^{k} \ll \la^{k-1}
}
Next, the combination of the first and sixth terms on the right-hand-side of~\eqref{eq:b'1} can be estimated using~\eqref{eq:L0-A0-wm}, 
\EQ{ \label{eq:A0cancel}
\Big|\frac{\la'}{\la}  \ang{[ \La_0 \La Q]_{\ula} \mid \dot g } - & \la' \ang{ \dot g \mid \A_0(\la) \La Q_{\ula}}\Big| = \abs{\frac{\la'}{\la} \ang{[ \La_0 \La Q]_{\ula} -  \A_0(\la) \La Q_{\la} \mid  \dot g} } \\
& \lesssim  \la^{\frac{k}{2}-1} \| [\La_0 \La Q]_{\ula} -  \A_0(\la) \La Q_{\la} \|_{L^2} \| \dot g \|_{L^2}  \\
& \lesssim c_0\la^{\frac{k}{2}-1}  \la^{\frac{k}{2}}   \ll \la^{k-1}
 }
 where in the last line above we rely on our ability to take $c_0$ as small as we like in the estimate~\eqref{eq:L0-A0-wm} from Lemma~\ref{lem:op-A-wm}. 
 
Thus we've show that up to terms of order $\ll \la^{k-1}$, which can be absorbed into the error,  we have 
\EQ{ \label{eq:b'2}
b'(t) & \simeq   -  \ang{ \La Q_{\ula} \mid  \p_r^2 g + \frac{1}{r} \p_r g  - \frac{1}{r^2} \left( f( Q_\la - Q_\mu + g) - f(Q_\la) + f(Q_\mu)\right)} \\
& \quad - \ang{ \p_r^2 g + \frac{1}{r} \p_r g  - \frac{1}{r^2} \left( f( Q_\la - Q_\mu + g) - f(Q_\la) + f(Q_\mu)\right) \mid \A_0(\la) g}  
}
Next, rescaling the equation $\LL \La Q = 0$, we see that 
\EQ{
 \LL_\la \La Q_{\ula} := (- \p_{rr} - \frac{1}{r} \p_r + \frac{f'(Q_\la)}{r^2}) \La Q_{\ula} = 0 
}
And since $\LL_{\la}$ is symmetric we have 
\EQ{
 \ang{ \La Q_{\ula} \mid  \p_r^2 g + \frac{1}{r} \p_r g } =  \ang{ \La Q_{\ula} \mid \frac{f'(Q_\la)}{r^2} g}
}
We thus rewrite~\eqref{eq:b'2} as 
\EQ{ \label{eq:b'3} 
b'(t) & \simeq     \ang{ \La Q_{\ula} \mid  \frac{1}{r^2} \Big( f( Q_\la - Q_\mu + g) - f(Q_\la) + f(Q_\mu)- f'(Q_\la)g\Big)} \\
& \quad - \ang{ \p_r^2 g + \frac{1}{r} \p_r g  - \frac{1}{r^2} \Big( f( Q_\la - Q_\mu + g) - f(Q_\la) + f(Q_\mu)\Big) \mid \A_0(\la) g}  
}
where the symbol $\simeq$ above means ``up to terms of order $\ll \la^{k-1}$". 
Adding and subtracting we have 
\begin{align}  \label{eq:lead3}
b'(t)  &\simeq 
 \ang{ \La Q_{\ula} \mid \frac{1}{r^2}  \Big( f(Q_{\la} - Q_\mu) - f(Q_\la) + f(Q_\mu) \Big)} \\ 
  &+  \ang{ \La Q_{\ula} \mid \frac{1}{r^2} \Big( f'(Q_\la - Q_\mu) - f'(Q_\la) \Big) g}  \label{eq:2ndterm} \\ 
& +  \ang{ \La Q_{\ula} \mid \frac{1}{r^2} \Big( f(Q_\la - Q_\mu + g) - f(Q_\la - Q_\mu) - f'(Q_\la - Q_\mu) g \Big)} \label{eq:3rdterm}  \\
 & - \ang{ \p_r^2 g + \frac{1}{r} \p_r g  - \frac{1}{r^2} \Big( f( Q_\la - Q_\mu + g) - f(Q_\la) + f(Q_\mu)\Big) \mid \A_0(\la) g}  \label{eq:4thterm}
\end{align}
Let's begin by estimating the first term on the right-hand-side above, which we'll show contributes the leading order: 
\begin{claim} \label{c:lead} 
\EQ{
\ang{ \La Q_{\ula} \mid \frac{1}{r^2}  \Big( f(Q_{\la} - Q_\mu) - f(Q_\la) + f(Q_\mu) \Big)}  \simeq 8k^2 \frac{\la^{k-1}}{\mu^k}
}
where again $\simeq$ means ``up to terms of order $\ll \la^{k-1}$."
\end{claim} 
\begin{proof}
Let's prove the claim. Recall that the nonlinearity $f(\rho)$ is given by 
\EQ{
f(\rho) := \frac{k^2}{2} \sin(2 \rho), \quad f'(\rho) = k^2 \cos(2 \rho)
}
Using trigonometric identities we can write 
\begin{align} 
f(Q_{\la} - Q_\mu) - f(Q_\la) + f(Q_\mu)  &=\frac{k^2}{2} ( \sin2 Q_\la( \cos 2Q_\mu - 1) + \sin 2 Q_\mu( 1 -  \cos 2 Q_\la)) \\
& = - k^2  \sin 2 Q_\la \sin^2 Q_\mu +  k^2 \sin 2 Q_\mu  \sin^2 Q_\la \\
& =  - \sin 2 Q_\la (\La Q_\mu)^2 + \sin 2 Q_\mu (\La Q_\la)^2.     \label{eq:trig1} 
\end{align}
We show that the leading order contribution comes from the second term above. Indeed, writing $\s= \la/ \mu$ and changing variables we have 
\EQ{ \label{eq:lead1} 
\big\langle \La Q_{\ula} &\mid  \frac{1}{r^2}(\La Q_{\la})^2 \sin 2Q_\mu \big\rangle= \frac{1}{ \la} \int_0^\infty ( \La Q_{\s}(r))^3 \sin 2 Q(r) \frac{ \ud r }{r}  \\
& = \frac{1}{\la}  \int_0^{\sqrt{\s}} ( \La Q_{\s}(r))^3 \sin 2 Q(r) \frac{ \ud r }{r}   + \frac{1}{\la}  \int_{\sqrt{\s}}^\infty ( \La Q_{\s}(r))^3 \sin 2 Q(r) \frac{ \ud r }{r} 
}
Since $\s = \la/ \mu  \ll 1$, on the interval $[0, \sqrt{\s}]$ we write 
\EQ{ \label{eq:sin2Q-small} 
\sin 2 Q = \frac{2}{k^2} 2 k^2 r^k \frac{ 1 - r^{2k}}{ (1+ r^{2k})^2} = 4 r^k + O(r^{3k})
}
Changing variables again we have 
\EQ{
\frac{1}{\la} \int_0^{\sqrt{\s}} ( \La Q_{\s})^3 4r^{k} \, \frac{\ud r}{r} &= 4 \frac{\s^k}{\la} \int_0^{\frac{1}{\sqrt{\s}}}  (\La Q)^3(r) r^{k-1} \, \ud r \\
& = 4 \frac{\s^k}{\la} \int_0^\infty  (\La Q)^3(r) r^{k-1} \, \ud r -  4 \frac{\s^k}{\la} \int_{\frac{1}{\sqrt{\s}}}^{\infty}  (\La Q)^3(r) r^{k-1} \, \ud r \\
& = 8k^2\frac{ \la^{k-1} }{\mu^k} + O( \la^{-1}\s^{2k})
}
where we used~\eqref{eq:LaQ3} in the last line above. Moreover, using~\eqref{eq:sin2Q-small} we estimate 
\EQ{
\frac{1}{\la}  \int_0^{\sqrt{\s}}( \La Q_{\s})^3( \sin 2Q - 4r^{k}) \, \frac{\ud r}{r}& \lesssim \frac{1}{\la}  \int_0^{\sqrt{\s}}( \La Q_{\s})^3 r^{3k-1} \, \ud r \\
& = \frac{\s^{3k}}{\la} \int_0^{\frac{1}{\sqrt{\s}}}( \La Q)^3 r^{3k-1} \, \ud r  \lesssim \frac{\s^{3k} \abs{\log \s}}{\la} 
}
Finally, we estimate the second term on the last line~\eqref{eq:lead1} by 
\EQ{
\abs{\frac{1}{\la}  \int_{\sqrt{\s}}^\infty ( \La Q_{\s}(r))^3 \sin 2 Q(r) \frac{ \ud r }{r} }&\lesssim \frac{1}{\la} \int_{\sqrt{\s}}^\infty ( \La Q_{\s}(r))^3 \, \frac{\ud r}{r}  =  \frac{1}{\la}\int_{\frac{1}{\sqrt{\s}}} ( \La Q(r))^3 \frac{ \ud r }{r} \\
&  \lesssim \frac{1}{\la}\int_{\frac{1}{\sqrt{\s}}}^{\infty} r^{-3k-1} \, \ud r \lesssim \frac{ \s^{\frac{3}{2} k}}{\la}
}
Next we estimate the contribution of the first term in~\eqref{eq:trig1}. Recalling that $\sin 2 Q_\la = \frac{2}{k^2} \La^2 Q_\la = \frac{2}{k^2} r \p_r (\La Q_{\la})$ we integrate by parts to obtain
\EQ{
-\ang{ \La Q_{\ula}  \mid\frac{1}{r^2} \sin 2 Q_\la (\La Q_\mu)^2} &= \frac{2}{k^2 \la} \ang{ (\La Q_{\la})^2 \mid \frac{1}{r^2} \La Q_{\mu} \La^2 Q_\mu } \\
& = \frac{2}{k^2 \la} \ang{ (\La Q_{\s})^2 \mid \frac{1}{r^2} \La Q \La^2 Q }
}
where $\s = \la/ \mu$ as before. We first estimate the last line above on the interval $[0, \s]$ and $[\sigma, 1]$ using \eqref{eq:LaQ}~\eqref{eq:La2Q} to obtain the bound 
\EQ{
\abs{\La Q(r)\La^2 Q(r)} \lesssim r^{2k} 
}
which gives 
\EQ{
\frac{2}{\la k^2} \int_0^\s (\La Q_{\s})^2 \La Q\La^2 Q \frac{\ud r}{r}&  \lesssim\frac{1}{\la} \int_0^\s \La Q_{\s}^2 r^{2k-1} \, \ud r = \frac{\s^{2k}}{\la} \int_0^1 \La Q^2  \, r^{2k-1} \, \ud r 
 \lesssim \frac{\s^{2k}}{\la}
}
and 
\EQ{
\frac{2}{\la k^2} \int_{\s}^1 (\La Q_{\s})^2 \La Q\La^2 Q \frac{\ud r}{r}& \lesssim \frac{\s^{2k}}{\la} \int_\s^1  \frac{r^{4k-1}}{ (\s^{2k} + r^{2k})^2} \, \ud  r \lesssim \frac{\s^{2k}\abs{\log \s}}{\la}
}
On the interval $[1, \infty]$ we use the formulas~\eqref{eq:LaQ}~\eqref{eq:La2Q} to estimate 
\EQ{
\abs{\La Q(r)\La^2 Q(r)} \lesssim \frac{r^{4k}}{(1+ r^{2k})^3}
}
which means 
\EQ{
\frac{2}{\la k^2} \int_{\s}^\infty (\La Q_{\s})^2 \La Q(r)\La^2 Q(r) \frac{\ud r}{r}&  \lesssim \frac{\s^{2k}}{\la} \int_{\s}^\infty \frac{r^{6k-1}\,  \ud r}{ (\s^{2k}+ r^{2k})^2(1+r^{2k})^3}  \\
& \lesssim  \frac{\s^{2k}}{\la} \int_0^\infty \frac{ r^{2k-1}}{(1+r^{2k})^3} \, \ud r \lesssim \frac{\s^{2k}}{\la}
}
Putting this all together, we've shown that 
\EQ{
\ang{\La Q_{\ula} \mid  \Big( f(Q_{\la} - Q_\mu) - f(Q_\la) + f(Q_\mu) \Big)}   =  8k^2 \frac{\la^{k-1}}{\mu^k} + O\left(\frac{\la^{k-1}}{\mu^k} \left(\frac{\la}{\mu}\right)^{\frac{k}{2}} \right)
}
which is precisely Claim~\ref{c:lead}. 
\end{proof} 

Next, we claim the second term~\eqref{eq:2ndterm} in our expansion of $b'(t)$ satisfies
\EQ{ \label{eq:2term} 
 \ang{ \La Q_{\ula} \mid \frac{1}{r^2} \Big( f'(Q_\la - Q_\mu) - f'(Q_\la) \Big) g}   = o(\la^{k-1})
}
and can thus be absorbed into the error. 
First note that the we have 
\EQ{
 f'(Q_\la - Q_\mu) - f'(Q_\la)& = k^2 \sin 2Q_\la \sin 2Q_\mu -2k^2 \cos 2Q_\la \sin^2Q_\mu \\
&=  \frac{4}{k^2} \La^2 Q_\la \La^2 Q_\mu- (\La Q_\mu)^2\cos 2Q_\la 
}
For the contribution from the first term above we integrate by parts, change variables, and use the explicit formulae \eqref{eq:LaQ}~\eqref{eq:La2Q}~\eqref{eq:La3Q} to estimate 
\EQ{
&\abs{\ang{\La Q_{\ula} \mid \frac{4}{k^2r^2} (\La^2 Q_\la \La^2 Q_\mu \Big)g}} \lesssim \frac{1}{\la} \abs{ \int_0^\infty (\La Q_\la)^2 \La^3 Q_\mu g \frac{\ud r }{r}} + \abs{\int_0^\infty (\La Q_\la)^2 \La^2Q_\mu r \p_r g \frac{\ud r}{r}}\\
& \lesssim \frac{1}{\la}  \| g\|_{H} \left[\left( \int_0^\infty (\La Q_\s)^4 (\La^3 Q)^2 \frac{\ud r}{r} \right)^{\frac{1}{2}} + \left( \int_0^\infty (\La Q_\s)^4 (\La^2 Q)^2 \ \frac{\ud r}{r} \right)^{\frac{1}{2}} \right] = o(\s^{k}/\la)
}
where $\s = \la/ \mu$ as before.  For the second term we write 
\EQ{
\abs{\ang{ \La Q_{\ula} \mid \frac{1}{r^2}( (\La Q_\mu)^2\cos 2Q_\la)  g}} &\lesssim \| g\|_{H} \left( \int_0^\infty  (\La Q_{\s})^2 (\La Q)^4 \frac{ \ud r}{r} \right)^{\frac{1}{2}} \\
& \lesssim \frac{1}{\la} \s^{\frac{k}{2}} \s^k \abs{\log \s}^{\frac{1}{2}}  = o( \s^k/ \la)
}
which finishes the proof of~\eqref{eq:2term}. 

Finally, we consider the last two terms~\eqref{eq:3rdterm}~\eqref{eq:4thterm}. We will reorganize these terms in anticipation of applications of Lemma~\ref{lem:op-A-wm}. First we rewrite~\eqref{eq:3rdterm} as follows: 
\begin{multline}
\Big\langle \Lambda Q_\uln\lambda \mid
\frac{1}{r^2}\big(f({-}Q_\mu + Q_\lambda + g) - f({-}Q_\mu + Q_\lambda) - f'({-}Q_\mu + Q_\lambda)g\big)\Big\rangle \\ 
=  -  \ang{ \A(\la) g \mid \frac{1}{r^2}\big(f({-}Q_\mu + Q_\lambda + g) - f({-}Q_\mu + Q_\lambda) -k^2g\big)} \\
+\ang{ \A(\la) g \mid \frac{1}{r^2}\big(f({-}Q_\mu + Q_\lambda + g) - f({-}Q_\mu + Q_\lambda) -k^2g\big)} \\
+\ang{ \A(\la)  (Q_\la- Q_\mu) \mid \frac{1}{r^2}\big(f({-}Q_\mu + Q_\lambda + g) - f({-}Q_\mu + Q_\lambda) - f'({-}Q_\mu + Q_\lambda)g\big)}  \\
+ \ang{ \A(\la) Q_\mu \mid \frac{1}{r^2}\big(f({-}Q_\mu + Q_\lambda + g) - f({-}Q_\mu + Q_\lambda) - f'({-}Q_\mu + Q_\lambda)g\big)}  \\
+ \ang{ \La Q_{\ula} - \A(\la)  Q_\la  \mid  \frac{1}{r^2}\big(f({-}Q_\mu + Q_\lambda + g) - f({-}Q_\mu + Q_\lambda) - f'({-}Q_\mu + Q_\lambda)g\big)}
\end{multline}
The second two terms on the right-hand-side can be controlled by setting  $g_1 = Q_{\la} - Q_\mu$ and $g_2 = g$ in~\eqref{eq:A-by-parts-wm}: 
\begin{multline}
\bigg|\ang{ \A(\la) g \mid \frac{1}{r^2}\big(f({-}Q_\mu + Q_\lambda + g) - f({-}Q_\mu + Q_\lambda) -k^2g\big)} 
\\ +\ang{ \A(\la)  (Q_\la- Q_\mu) \mid \frac{1}{r^2}\big(f({-}Q_\mu + Q_\lambda + g) - f({-}Q_\mu + Q_\lambda) - f'({-}Q_\mu + Q_\lambda)g\big)} \bigg| \\
 \lesssim c_0 \la^{k-1}
\end{multline}
Using  the pointwise bound
\begin{multline}
\abs{ f(Q_\la - Q_\mu + g) - f(Q_\la - Q_\mu) - f'(Q_\la - Q_\mu) g }  \\
= \frac{k^2}{2}\abs{ \sin(2Q_\la - 2Q_\mu)[ \cos 2 g - 1] + \cos(2Q_\la - 2Q_\mu)[ \sin 2g -2 g]} 
\lesssim \abs{g}^2
\end{multline} 
along with~\eqref{eq:Ainfty} the second to last line of the above can be estimated by 
\EQ{
 \ang{ \A(\la) Q_\mu \mid \frac{1}{r^2}\big(f({-}Q_\mu + Q_\lambda + g) - f({-}Q_\mu + Q_\lambda) - f'({-}Q_\mu + Q_\lambda)g\big)}  \lesssim c_0 \la^{k-1}.
}
Similarly, the last line of the expansion of~\eqref{eq:3rdterm} can be controlled as follows 
\begin{multline}
\abs{\ang{ \La Q_{\ula} - \A(\la)  Q_\la  \mid  \frac{1}{r^2}\big(f({-}Q_\mu + Q_\lambda + g) - f({-}Q_\mu + Q_\lambda) - f'({-}Q_\mu + Q_\lambda)g\big)}}
\\ \lesssim \|\La Q_{\ula} - \A(\la)  Q_\la \|_{L^\infty} \| g\|^2_H  \le Cc_0 \la^{k-1} \ll \la^{k-1}
\end{multline} 
In the last line we used~\eqref{eq:L-A-wm} and the fact that $c_0$ can be taken small independently of $\la$ in Lemma~\ref{lem:op-A-wm}. 

Thus, up to terms of order $\ll \la^{k-1}$ we have 
\begin{multline}  \label{eq:3rdterm2} 
\Big\langle \Lambda Q_\uln\lambda \mid
\frac{1}{r^2}\big(f({-}Q_\mu + Q_\lambda + g) - f({-}Q_\mu + Q_\lambda) - f'({-}Q_\mu + Q_\lambda)g\big)\Big\rangle  \\ \simeq -  \ang{ \A(\la) g \mid \frac{1}{r^2}\big(f({-}Q_\mu + Q_\lambda + g) - f({-}Q_\mu + Q_\lambda) -k^2g\big)} 
\end{multline} 
We now transform the last line~\eqref{eq:4thterm} adding and subtracting terms as before. Using~\eqref{eq:A-pohozaev-wm} we have
\begin{equation}
\label{eq:mod-dtb-2nd-line}
\begin{aligned}
&{-} \Big\langle \partial_r^2 g + \frac 1r \partial_r g - \frac{1}{r^2}\big(f(Q_\la- Q_\mu + g) - f(Q_\la) +f(Q_\mu)\big) \mid  \A_0(\lambda)g\Big\rangle \\
& = -\ang{ \A_0(\la) g \mid  \partial_r^2 g + \frac 1r \partial_r g - \frac{k^2}{r^2}g} \\
&\quad + \ang{ \A_0(\la) g \mid \frac{1}{r^2} \Big( f(Q_\la - Q_\mu) - f(Q_\la) + f(Q_\mu) \Big)}\\
& \quad +\ang{ \A_0(\la) g \mid  \frac{1}{r^2} \Big( f(Q_\la - Q_\mu + g) - f(Q_\la-Q_\mu) - k^2 g \Big)} \\
&\geq -\frac{c_0}{\lambda}\|g \|_H^2 + \frac{1}{\lambda}\int_0^{R\lambda}\Big((\partial_r g)^2 + \frac{k^2}{r^2}g^2\Big)r\ud r \\
&+ \ang{ \A_0(\la) g \mid \frac{1}{r^2} \Big( f(Q_\la - Q_\mu) - f(Q_\la) + f(Q_\mu) \Big)}\\
&+ \Big\langle \A_0(\lambda)g \mid 
\frac{1}{r^2}\big(f({-}Q_\mu + Q_\lambda + g) + f(-Q_\mu + Q_\la) - k^2g\big)\Big\rangle
\end{aligned}
\end{equation}
where $R$ is as in Lemma~\ref{lem:op-A-wm}. 
Note that from~\eqref{eq:trig1} we have the pointwise inequality 
\EQ{
\abs{ f(Q_\la - Q_\mu) - f(Q_\la) + f(Q_\mu)} \lesssim (\La Q_\la)^2 (\La Q_\mu) + \La Q_\la (\La Q_{\mu})^2
}
Since $\| \A_0(\la) g \|_{L^2} \lesssim \|g \|_{H}$, and since $\A_0(\la) g$ is supported on a ball of radius $R \la$, the term on the second to last line above can be estimated as follows, 
\EQ{
\bigg| \bigg\langle &\A_0(\la) g \mid \frac{1}{r^2} \Big( f(Q_\la - Q_\mu) - f(Q_\la) + f(Q_\mu) \Big) \bigg \rangle \bigg| \\ & \lesssim \|g \|_H\left[ \left( \int_0^{R \s} r^{-2} (\La Q_\s)^4 (\La Q)^2 \,  \frac{\ud r}{r} \right)^{\frac{1}{2}} +  \left( \int_0^{R \s} r^{-2} (\La Q)^4 (\La Q_\s)^2 \,  \frac{\ud r}{r} \right)^{\frac{1}{2}} \right] \\
&\lesssim \s^{\frac{k}{2}} \s^{k-1} \ll \la^{k-1} 
}
where $\s = \la/ \mu$ as usual and $\mu \simeq 1$. 

Therefore, up to terms of order $\ll \la^{k-1}$, we can put together~\eqref{eq:3rdterm2} and~\eqref{eq:mod-dtb-2nd-line} to estimate the combination of~\eqref{eq:3rdterm} and~\eqref{eq:4thterm} from below by 
\EQ{ \label{eq:34terms} 
&\Big\langle \Lambda Q_\uln\lambda \mid
\frac{1}{r^2}\big(f({-}Q_\mu + Q_\lambda + g) - f({-}Q_\mu + Q_\lambda) - f'({-}Q_\mu + Q_\lambda)g\big)\Big\rangle \\
&{-} \Big\langle \partial_r^2 g + \frac 1r \partial_r g - \frac{1}{r^2}\big(f(Q_\la- Q_\mu + g) - f(Q_\la) +f(Q_\mu)\big) \mid  \A_0(\lambda)g\Big\rangle  \\
& \ge  \frac{1}{\lambda}\int_0^{R\lambda}\Big((\partial_r g)^2 + \frac{k^2}{r^2}g^2\Big)r\ud r  \\
& \quad +\Big\langle \big(\A_0(\lambda) - \A(\lambda)\big) g \mid \frac{1}{r^2}\big(f({-}Q_\mu + Q_\lambda + g) - f({-}Q_\mu+Q_\lambda)- k^2g\big)\Big\rangle
}
Since $\A_0(\lambda) - \A(\lambda)$ is the operator of multiplication by $\frac{1}{2\lambda} \Big(q''\big(\frac{r}{\lambda}\big) + \frac{\lambda}{r}q'\big(\frac{r}{\lambda}\big)\Big)$,
we can use \eqref{eq:approx-potential-wm} to estimate the last term above, 
\EQ{ \label{eq:34terms-error} 
 \Big\langle \big(\A_0(\lambda) - \A(\lambda)\big) g\mid \frac{1}{r^2}\big(f({-}Q_\mu + Q_\lambda + g) - f({-}Q_\mu+Q_\lambda)- k^2g\big)\Big\rangle   \\ 
 =  \frac{1}{\lambda}\int_0^{+\infty} \frac{1}{r^2}\big(f'(Q_\lambda)-k^2\big)g^2 \udr  + O(c_0 \la^k)
}
where $c_0>0$ is as in Lemma~\ref{lem:op-A-wm}. 

Putting together the estimates from Claim~\ref{c:lead},~\eqref{eq:2term},~\eqref{eq:34terms}, and~\eqref{eq:34terms-error} we obtain the estimate 
\EQ{
b'(t)  &\ge 8k^2 \frac{\la^{k-1}}{\mu^k}  + o(\la^{k-1}) \\
& \quad + \frac{1}{\lambda}\int_0^{R\lambda}\Big((\partial_r g)^2 + \frac{k^2}{r^2}g^2\Big)r\ud r  + \frac{1}{\lambda}\int_0^{+\infty} \frac{1}{r^2}\big(f'(Q_\lambda)-k^2\big)g^2 \udr  
}
Finally we conclude by using the following localized coercivity estimate, 
\EQ{
\frac{1}{\lambda}\int_0^{R\lambda}\Big((\partial_r g)^2 + \frac{k^2}{r^2}g^2\Big)r\ud r  + \frac{1}{\lambda}\int_0^{+\infty} \frac{1}{r^2}\big(f'(Q_\lambda)-k^2\big)g^2 \udr   \ge - \frac{c_1}{\la} \| g\|_H^2  
}
where we use again crucially here that $\ang{ \ZZ_{\ula} \mid g} = 0$; see~\cite[Lemma 5.4, eq. (5.28)]{JJ-AJM} for the proof. Above the constant $c_1>0$ can be made as small as we like by taking $R>0$ large enough, which we are free to do. This completes the proof. 
\end{proof} 

  
\section{Dynamics of Non-Scattering Threshold Solutions} \label{s:dynamics}
\subsection{Overall scheme}
In this section we prove the Main Theorem.
We deduce it from the following proposition, whose proof will be split into several lemmas. 



\begin{prop}\label{p:psi_t} 
Let $\psi(t) :(T_-, T_+) \to \HH_0$ be a solution to~\eqref{eq:wmk} with $\E(\vec \psi) = 2 \E(\vec Q)$ which does not scatter in forward time. Then
\EQ{ 
\lim_{t \to T_+} \bfd(\vec \psi(t)) = 0.\label{eq:psi_t}
}
\end{prop}
To begin with, note the following special case of Theorem~\ref{t:cjk}.
\begin{prop}\label{p:cjk}
Let $\vec \psi(t) : (T_-, T_+) \to \HH_0$ be a solution to \eqref{eq:wmk} with $\E(\vec\psi) = 2 \E( \vec Q)$
which does not scatter in forward time. Then
\EQ{ \label{eq:seqbub} 
\liminf_{t\to T_+} \bfd(\vec \psi(t)) = 0.
}
An analogous statement holds if $\vec \psi(t)$ does not scatter in backwards time.\qed
\end{prop}

Let us summarize the main idea of the proof of Proposition~\ref{p:psi_t}. We know from Proposition~\ref{p:cjk} that
\eqref{eq:psi_t} holds \emph{for a sequence of times}.
Thus, in order to obtain \eqref{eq:psi_t}, we should prove that $\vec\psi(t)$, after exiting a small neighborhood of a two-bubble configuration,
can never approach a two-bubble again. Such a result is similar in nature to the~\emph{no-return lemma} proved by Krieger, Nakanishi and Schlag \cite{KNS15} in their study of the dynamics near the ground state stationary solution for the energy critical NLW. 
Such results are usually obtained by means of a convexity argument based on monotonicity formulas, which is also the scheme that we adopt here.

Until the end of this section, $\vec \psi(t)$ always denotes a solution to~\eqref{eq:wmk}, $\vec\psi(t):(T_-, T_+) \to \HH_0$,
such that $\E(\vec \psi) = 2 \E(\vec Q)$
and $\vec\psi(t)$ does not scatter in forward time.
Let $T_- < \tau_1 \leq \tau_2 < T_+$. Integrating the virial identity from Lemma~\ref{l:vir} for $t \in [\tau_1, \tau_2]$ yields 
\EQ{ \label{eq:virtau}
\int_{\tau_1}^{\tau_2} \| \p_t \psi (t) \|_{L^2}^2 \,  \ud t  &\le  \abs{ \ang{ \p_t\psi \mid \chi_R r \p_r \psi}(\tau_1)}  + \abs{ \ang{ \p_t\psi \mid   \chi_R r \p_r \psi}(\tau_2)} \\
 &\quad +  \int_{\tau_1}^{\tau_2} \abs{\Om_{R}( \vec \psi(t))} \, \ud t
 }
 where $\Om_{R}(\vec \psi(t))$ is defined in~\eqref{eq:OmRdef}.
  Note that for any $R>0$ we can use Lemma~\ref{l:error-estim} to bound the first two terms on the right-hand-side above and obtain
 \EQ{ \label{eq:virR} 
 \begin{aligned}
 \int_{\tau_1}^{\tau_2} \| \p_t \psi (t) \|_{L^2}^2 \,  \ud t  &\leq C_0\Big(  R \sqrt{\bfd(\vec\psi(\tau_1))} +  R \sqrt{\bfd(\vec\psi(\tau_2))}\Big) \\ &+  \int_{\tau_1}^{\tau_2} \abs{\Om_{R}( \vec \psi(t))} \, \ud t.
 \end{aligned}
 }

 Our goal is to show that with a good choice of $R$, $\tau_1$ and $\tau_2$
 the right hand side can be absorbed into the left hand side.
 As mentioned in the Introduction, we use different arguments
 depending whether $\bfd(\vec\psi(t))$ is small or not.
  
\subsection{Splitting of the time axis}
We would like to divide the time axis into \emph{good intervals} where $\bfd(\vec\psi(t))$ is large
and \emph{bad intervals} where it is small.
We begin with a preliminary splitting, which will then need to be refined.
\begin{claim}\label{cl:pre-split}
Suppose that~\eqref{eq:psi_t} fails. Then for any $\eps_0 > 0$ sufficiently small there exist
sequences $p_n$, $q_n$ such that
\EQ{
T_- < p_0 < q_0 < p_1 < q_1 < \dots < p_{n-1} < q_{n-1} < p_n < q_n < \dots 
}
such that the following holds for all $n \in \{0, 1, 2, 3, \ldots\}$:
\begin{gather}
\forall t \in [p_n, q_n]: \bfd(\vec\psi(t)) \leq \eps_0, \label{eq:d-small-all-pq} \\
\forall t \in [q_n, p_{n+1}]: \bfd(\vec\psi(t)) \geq \frac 12\eps_0, \label{eq:d-large-all-pq} \\
\lim_{n \to +\infty}p_n = \lim_{n\to+\infty}q_n = T_+. \label{eq:pq-lim}
\end{gather}
\end{claim}
\begin{proof}
Suppose that \eqref{eq:psi_t} fails and let $\eps_0$ be any number such that
\EQ{\label{eq:eps-limsup}
0 < \eps_0 < \min(\limsup_{t\to T_+}\bfd(\vec\psi(t)), \eta_1)
}
(recall that $\bfd(\vec\psi(t_0)) < \eta_1$ guarantees that the modulation
estimates hold for $t$ in some neighborhood of $t_0$).
Let $T_0 \in (T_-, T_+)$ be such that $\bfd(\vec\psi(T_0)) > \eps_0$.
We set
\EQ{
p_0 := \sup\Big\{t: \bfd(\vec\psi(\tau)) \geq \frac 12 \eps_0, \forall \tau\in[T_0, t]\Big\}.
}
Proposition~\ref{p:psi_t} implies that $p_0 < T_+$ and $\bfd(\vec\psi(p_0)) = \frac 12 \eps_0$.
Then we define inductively for $n \geq 1$:
\begin{align}
q_{n-1} := \sup\Big\{t: \bfd(\vec\psi(\tau)) \leq \eps_0, \forall \tau\in[p_{n-1}, t]\Big\}, \\
p_n := \sup\Big\{t: \bfd(\vec\psi(\tau)) \geq \frac 12 \eps_0, \forall \tau\in[q_{n-1}, t]\Big\}.
\end{align}
By a simple inductive argument using \eqref{eq:eps-limsup} and Proposition~\ref{p:cjk}
we can show that for $n \in \{1, 2, \ldots\}$ there holds
\begin{gather}
p_{n-1} < q_{n-1} < T_+, \\
q_{n-1} < p_n < T_+, \\
\bfd(\vec\psi(p_n)) = \frac 12\eps_0, \label{eq:d-val-at-p} \\
\bfd(\vec\psi(q_n)) = \eps_0 \label{eq:d-val-at-q}.
\end{gather}
Bounds \eqref{eq:d-large-all-pq} and \eqref{eq:d-small-all-pq} follow directly from the definitions
of $p_n$ and $q_n$. Suppose that \eqref{eq:pq-lim} does not hold. Then, by monotonicity,
\EQ{
\lim_{n \to +\infty}p_n = \lim_{n\to+\infty}q_n = T_1 < T_+.
}
By the local well-posedness $\bfd(\vec\psi(t))$ has a limit as $t \to T_1$,
which is in contradiction with \eqref{eq:d-val-at-p} and \eqref{eq:d-val-at-q}.
\end{proof}
\begin{claim}\label{cl:split}
Let $\eps > 0$. There exist $\lambda_0, \eps' > 0$ having the following property.
Assume that $\bfd(\vec\psi(t)) < \eta_1$, with $\eta_1$ as in Proposition~\ref{p:modp2}, and let $\lambda(t)$, $\mu(t)$ be the modulation parameters
given by Lemma~\ref{l:modeq} and let $\zeta(t)$ be the correction to $\la(t)$ defined in~\eqref{eq:zetadef}. Then
\begin{align}
\frac{\zeta(t)}{\mu(t)} \geq \lambda_0 &\Rightarrow \bfd(\vec\psi(t)) > \eps', \label{eq:d-geq-epsp} \\
\frac{\zeta(t)}{\mu(t)} \leq \lambda_0 &\Rightarrow \bfd(\vec\psi(t)) < \eps \label{eq:d-leq-eps}.
\end{align}
\end{claim}
\begin{rem}
Note that $\eps' < \eps$.
\end{rem}
\begin{proof}

Lemma~\ref{l:modeq} yields $\bfd(\vec\psi(t)) \leq (C^2+1)\big(\lambda(t)/\mu(t))^k \le 2(C^2+1)\big(\zeta(t)/\mu(t))^k$,
so we get \eqref{eq:d-leq-eps} with any $\lambda_0 < \big(\eps/2(C^2+1)\big)^{\frac 1k}$.

In order to prove \eqref{eq:d-geq-epsp}, we notice that
from Lemma~\ref{l:modeq} and~\eqref{eq:bound-on-l} we get $\bfd(\vec\psi(t)) \geq \frac{1}{C}\big(\zeta(t)/\mu(t)\big)^k$,
hence it suffices to take $\eps' < \frac 1C \lambda_0^k$.
\end{proof}
\begin{lem}
Suppose that~\eqref{eq:psi_t} fails. Let $\eps_0>0$ be small enough so that Claim~\ref{cl:pre-split}
and Proposition~\ref{prop:modulation} hold. Then there exist $\eps, \eps' > 0$ with $\eps' < \eps$  and $\eps< \frac{1}{10} \eps_0$ as in Proposition~\ref{prop:modulation}, 
and a splitting of the time axis
\EQ{
T_- < a_1 < c_1 < b_1 < \dots < a_m < c_m < b_m < a_{m+1} < \dots 
}
such that the following holds for all $m \in \{2, 3, 4, \ldots\}$:
\begin{gather}
\forall t \in [b_{m}, a_{m+1}]: \bfd(\vec\psi(t)) \geq \eps', \label{eq:d-large-all} \\
\exists t \in [b_{m}, a_{m+1}]: \bfd(\vec\psi(t)) \geq 2\eps, \label{eq:d-large-some} \\
\bfd(\vec\psi(a_m)) = \bfd(\vec\psi(b_{m})) = \eps, \label{eq:d-small-some} \\
C_0 \int_{a_{m+1}}^{c_{m+1}}\sqrt{\bfd(\vec\psi(t))}\,\ud t \leq \frac{1}{10}\int_{b_{m}}^{a_{m+1}}\|\partial_t\psi(t)\|_{L^2}^2\,\ud t \label{eq:err-mod-contr-1} \\
C_0 \int_{c_{m}}^{b_{m}}\sqrt{\bfd(\vec\psi(t))}\,\ud t \leq \frac{1}{10}\int_{b_{m}}^{a_{m+1}}\|\partial_t\psi(t)\|_{L^2}^2\,\ud t \label{eq:err-mod-contr-2}
\end{gather}
and
\EQ{
\liminf_{m \to+\infty} \bfd(\vec\psi(c_m)) = 0. \label{eq:d-conv-0}
}
\end{lem}
\begin{proof}
We choose $\eps, \eps_0 > 0$ such that Claim~\ref{cl:pre-split}
and Proposition~\ref{prop:modulation} hold,
where the constant $C_0$ in Proposition~\ref{prop:modulation}
is given by Lemma~\ref{l:error-estim}.
We can assume that $\eps < \frac{1}{10}\eps_0$.
Let $\lambda_0$ and $\eps'$ be given by Claim~\ref{cl:split}.

We begin by defining the times $c_m$. Let $0 \leq n_1 < n_2 < ...$ be the sequence of
these indices $n_m$ for which
\EQ{\label{eq:inf-leq-eps}
\inf_{t\in[p_{n_m}, q_{n_m}]} \bfd(\vec\psi(t)) \leq \eps'.
}
Recall that the modulation parameters $\lambda(t)$, $\mu(t)$, and $\zeta(t) \simeq \la(t)$
are well defined on $[p_{n_m}, q_{n_m}]$.
Let $c_m \in [p_{n_m}, q_{n_m}]$ be such that
\EQ{
\zeta(c_m)/\mu(c_m) = \inf_{t \in [p_{n_m}, q_{n_m}]}\zeta(t)/\mu(t).
}
Claim~\ref{cl:split} and \eqref{eq:inf-leq-eps} imply that $\zeta(c_m) / \mu(c_m) < \lambda_0$,
which implies again by Claim~\ref{cl:split}
that $\bfd(\vec\psi(c_m)) < \eps < \frac{1}{10}\eps_0$.
Hence $c_m \in (p_{n_m}, q_{n_m})$ and
\EQ{
\dd{t}\Big|_{t = c_m}\Big(\frac{\zeta(t)}{\mu(t)}\Big) = 0.
}

We will use Proposition~\ref{prop:modulation} with various $t_0$, in forward and backward direction.
Thus the meaning of $t_0$, $t_1$ and $t_2$ will change depending on the context.

Using Proposition~\ref{prop:modulation} with $t_0$ = $c_m$ in the backward time direction
we obtain times $t_1 \leq c_m$ and $t_2 \leq t_1$.
Note that \eqref{eq:d-t0-t1} and \eqref{eq:d-large-all-pq} imply that $t_1 \in (p_{n_m}, c_m]$.
We set
\EQ{
a_m := \sup\{t \geq t_1: \bfd(\vec\psi(t)) \geq \eps,\ \forall \tau\in [t_1, t]\}.
}
By \eqref{eq:d-t1-t2} we have $\bfd(\vec\psi(t_1)) > \eps$.
Since $\bfd(\vec\psi(c_m)) < \eps$, we have $a_m \in (p_{n_m}, c_m)$
and $\bfd(\vec\psi(a_m)) = \eps$.

Denote $\sigma_m := t_2$. Then \eqref{eq:d-t1-t2} yields
$\bfd(\vec\psi(t)) \geq \eps$ for $t \in [\sigma_m, t_1]$.
By definition of $a_m$ we also have $\bfd(\vec\psi(t)) \geq \eps$ for $t \in [t_1, a_m]$, hence
\EQ{\label{eq:d-sigm-am}
\bfd(\vec\psi(t)) \geq \eps,\ \forall t \in [\sigma_m, a_m].
}
Bound \eqref{eq:d-at-t2} together with \eqref{eq:d-small-all-pq} yields $\sigma_m < p_{n_m}$,
so \eqref{eq:d-sigm-am} implies that
\EQ{\label{eq:d-pnm-am}
\bfd(\vec\psi(t)) \geq \eps,\ \forall t \in [p_{n_m}, a_m].
}
Finally, \eqref{eq:err-absorb} yields
\EQ{\label{eq:err-absorb-sig}
\int_{\sigma_m}^{a_m} \|\partial_t\psi(t)\|_{L^2}^2\ud t \geq C\int_{a_m}^{c_m}\sqrt{\bfd(\vec \psi(t))}\ud t.
}

Now using Proposition~\ref{prop:modulation} with $t_0$ = $c_m$ in the forward time direction
we obtain times $t_1 \geq c_m$ and $t_2 \geq t_1$.
Note that \eqref{eq:d-t0-t1} and \eqref{eq:d-large-all-pq} imply that $t_1 \in [c_m, q_{n_m})$.
We set
\EQ{
b_m := \inf\{t \leq t_1: \bfd(\vec\psi(t)) \geq \eps,\ \forall \tau\in[t, t_1]\}.
}
As in the case of $a_m$, we obtain $b_m \in (c_m, q_{n_m})$ and $\bfd(\vec\psi(b_m)) = \eps$.
Denote $\tau_m := t_2$. Adapting the proofs of \eqref{eq:d-sigm-am} and \eqref{eq:d-pnm-am} we get
\begin{gather}
\label{eq:d-bm-taum}
\bfd(\vec\psi(t)) \geq \eps,\ \forall t \in [b_m, \tau_m], \\
\label{eq:d-bm-qnm}
\bfd(\vec\psi(t)) \geq \eps,\ \forall t \in [b_m, q_{n_m}], \\
\label{eq:err-absorb-tau}
\int_{b_m}^{\tau_m} \|\partial_t\psi(t)\|_{L^2}^2\ud t \geq C\int_{c_m}^{b_m}\sqrt{\bfd(\vec \psi(t))}\ud t.
\end{gather}
We will prove that $\tau_m < a_{m+1}$. Suppose not. Since $\bfd(\vec\psi(\tau_m)) \geq 2\eps_0$,
see \eqref{eq:d-at-t2}, the fact that $a_{m+1} \in [p_{n_{m+1}}, q_{n_{m+1}}]$ would imply
that $\tau_m > q_{n_{m+1}}$. Thus by \eqref{eq:d-bm-taum} we would have
$\bfd(\vec\psi(t)) \geq \eps,\ \forall t \in [b_m, q_{n_{m+1}}]$. But $b_m < q_{n_m} < p_{n_{m+1}}$,
so we obtain
\EQ{
\bfd(\vec\psi(t)) \geq \eps,\ \forall t \in [p_{n_{m+1}}, q_{n_{m+1}}].
}
Since $\eps' < \eps$, this contradicts the definition of $n_{m+1}$.
Thus $\tau_m < a_{m+1}$, so \eqref{eq:err-absorb-tau} implies \eqref{eq:err-mod-contr-2}.
Also $\bfd(\vec\psi(\tau_m)) \geq 2\eps_0$ yields \eqref{eq:d-large-some}, since $\eps_0 \geq \eps$.

Analogously, we have $\sigma_{m+1} > b_m$, so \eqref{eq:err-absorb-sig} yields \eqref{eq:err-mod-contr-1}.

It remains to prove \eqref{eq:d-large-all}.
Take $t$ such that $\bfd(\vec\psi(t)) < \eps'$.
Then $t \in [p_{n_m}, q_{n_m}]$ for some $m$,
so \eqref{eq:d-pnm-am} and \eqref{eq:d-bm-qnm} yield $t\in[a_m, b_m]$,
which is exactly \eqref{eq:d-large-all}.

Finally, \eqref{eq:d-conv-0} follows from Proposition~\ref{p:cjk},~\eqref{eq:lamud1}, and~\eqref{eq:bound-on-l}.
\end{proof}
\begin{rem}\label{rem:eps-small}
It follows from the proof that $\eps$ can be taken as small as we wish.
\end{rem}

Until the end of the proof of Proposition~\ref{p:psi_t},
we fix $\eps, \eps' > 0$ and a partition of the time axis
given by the last lemma. In particular, all the constants are allowed to depend
on $\eps$ and $\eps'$. We denote
\EQ{
I_m := [b_{m-1}, a_m], \qquad I := \bigcup_{m\geq 1} I_m.
}
Then \eqref{eq:d-large-all} is equivalent to
\begin{equation}
\forall t\in I: \bfd(\vec\psi(t)) \geq \eps'. \label{eq:d-large-I}
\end{equation}

We will see that $\vec\psi(t)$ has a \emph{compactness property} for $t\in I$,
which allows to deal with the right hand side of \eqref{eq:virR} for $t \in I$. For $t \notin I$ we will rely on \eqref{eq:err-mod-contr-1}
and \eqref{eq:err-mod-contr-2}.

\subsection{Compactness on $I$}
The objective of this step is to deduce a compactness statement on $I$ that will allow us to obtain a lower bound for the left-hand-side of~\eqref{eq:virR} restricted to $I$ and to uniformly control the errors $\Om_R(\vec \psi(t))$ on  $I$, by choosing $R$ large enough.
\begin{lem} \label{l:Icompact} 
There exists a continuous function $\nu: I  \to (0, +\infty)$ such that the set 
\EQ{
\calK:=  \{  \vec \psi(t)_{1/\nu(t)} \mid t \in I \} \subset \HH_0 
}
is pre-compact in $\HH_0$. 
\end{lem} 

\begin{proof}
We will first prove that for any sequence $\{t_n \} \in I$ there exists a subsequence (still denoted by $t_n$) and a sequence of scales $\nu_n$,  so that 
\EQ{ 
\vec \psi(t_n)_{1/\nu_n} \to \vec \fy \in \HH_0 \label{eq:comp-seq}
} 
for some $\vec\fy \in \HH_0$. 

We observe that by Lemma~\ref{l:d-size} and \eqref{eq:d-large-I}  we have the uniform bound
\EQ{ \label{eq:IepsH} 
\| \vec \psi(t) \|_{\HH_0} \le C(\eps') \quad \forall t \in I,
}
which means in particular that 
\EQ{ \label{eq:tmH} 
\| \vec \psi(t_n) \|_{\HH_0} \le C(\eps') < \infty.
}
Thus \eqref{eq:comp-seq} follows from Lemma~\ref{l:1profile}.  

We are now ready to construct the function $\nu(t)$.
For each $t \in I$ let $\nu(t)$ be the unique number such that
\begin{equation}
\begin{gathered}
\int_0^\infty  \mathrm e^{-r}\left(( \p_t \psi_{1/\nu(t)}(t, r))^2+( \p_r \psi_{1/\nu(t)}(t, r))^2 +  k^2  \frac{(\psi_{1/\nu(t)}(t, r))^2}{r^2} \right) \rdr \\ = \frac 12 \| \vec \psi(t) \|_{\HH_0}^2
\end{gathered}
\end{equation}
(the function $\mathrm e^{-r}$ could be replaced by any continuous strictly decreasing function
whose value is $1$ for $r = 0$ and tending to $0$ as $r \to +\infty$).
We see that $\nu(t)$ is a continuous function.

Suppose that $\vec \psi(t)_{1/\nu(t)}$ is not pre-compact in $\HH_0$.
Thus there exists a sequence $\vec \psi(t_n)_{1/\nu(t_n)}$ which has no convergent subsequence.
But we know (by assumption) that there exist a subsequence (still denoted $t_n$)
and numbers $\nu_n$ such that $\vec \psi(t_n)_{1/\nu_n}$ converges in $\HH_0$
to some $\vec\fy = (\fy_0, \fy_1)$.
This implies that 
\begin{equation}
\int_0^\infty  \mathrm e^{-r}\left(( \p_t \psi_{1/\nu_n}(t_n, r))^2+( \p_r \psi_{1/\nu_n}(t_n, r))^2 +  k^2  \frac{(\psi_{1/\nu_n}(t_n, r))^2}{r^2} \right) \rdr
\end{equation}
converges to
\begin{equation}
\int_0^\infty  \mathrm e^{-r}\left(\fy_1(r)^2+( \p_r \fy_0(r))^2 +  k^2  \frac{\fy_0(r)^2}{r^2} \right) \rdr \in (0, \|\vec \fy\|_{\HH_0}).
\end{equation}
Since $\|\vec\psi(t_n)\|_{\HH_0}$ converges to $\|\vec\fy\|_{\HH_0}$,
we deduce that $\nu_n / \nu(t_n)$ is bounded.
This implies that $\nu_n / \nu(t_n)$ has a convergent subsequence,
hence $\vec \psi(t_n)_{1/\nu(t_n)}$ has a convergent subsequence,
so we have a contradiction. This completes the proof.
\end{proof}

For $m \in \{1, 2, 3, \ldots\}$ we define
\EQ{
\nu_m := |I_m| = a_m - b_{m-1}.
}
\begin{lem}\label{l:interval-bd}
There exists $C_1 > 0$ such that for all $m \geq 1$ and all $t \in I_m$ there holds
\EQ{
\frac{1}{C_1} \nu(t) \leq \nu_m \leq C_1 \nu(t).
}
\end{lem}
\begin{rem}
The last lemma tells us that $\nu(t)$ is comparable to $\nu_m$ for $t \in I_m$.
In particular, the set
\EQ{
\calK_1:=  \bigcup_{m \geq 1}\{  \vec \psi(t)_{1/\nu_m} \mid t \in I_m \} \label{eq:K1-comp}
}
is pre-compact in $\HH_0$. 
\end{rem}
\begin{proof}[Proof of Lemma~\ref{l:interval-bd}]
Suppose that there exists a sequence $m_\ell$ and times $t_\ell \in I_{m_\ell}$ such that
\EQ{ \label{eq:int-bd-part-1}
\lim_{\ell\to+\infty}\frac{\nu_{m_\ell}}{\nu(t_\ell)} = 0.
}
Let $\vec\psi_\ell(s)$ be the solution of \eqref{eq:wmk} with initial data
$
\vec\psi_\ell(0) = \vec\psi(t_\ell)_{1/\nu(t_\ell)}.
$
After extracting a subsequence, $\vec\psi_\ell(0)$ converges in $\HH_0$ to some $\vec\fy_0$.
Let $\vec\fy(s):[{-}s_0, s_0] \to \HH_0$ be the solution of \eqref{eq:wmk} with initial data
$\vec\fy(0) = \vec\fy_0$ (where $s_0 > 0$). By the standard Cauchy theory, for sufficiently large $\ell$
the solution $\vec\psi_\ell(s)$ is defined for $s \in [{-}s_0, s_0]$ and $\vec\psi_\ell(s) \to \vec\fy(s)$
in $\HH_0$, uniformly for $s \in [{-}s_0, s_0]$.

Let $t_\ell' \in I_{m_\ell}$ be any sequence. Let $s_\ell = \frac{t_\ell' - t_\ell}{\nu(t_\ell)}$,
Then \eqref{eq:int-bd-part-1} implies that $\lim_{\ell \to +\infty} s_\ell = 0$.
Thus $s_\ell \in [{-}s_0, s_0]$ for large $\ell$ and we deduce that
\EQ{
\lim_{\ell\to+\infty}\|\vec\psi_\ell(s_\ell) - \vec\fy(s_\ell)\|_{\HH_0} = 0.
}
But of course $\lim_{\ell\to+\infty}\|\vec\fy(s_\ell) - \vec\fy_0\|_{\HH_0} = 0$,
so the triangle inequality yields
\EQ{
\lim_{\ell\to+\infty}\|\vec\psi_\ell(s_\ell) - \vec\fy_0\|_{\HH_0} = 0.
}
In particular, $\lim_{\ell\to+\infty} \bfd(\vec\psi_\ell(s_\ell)) = \bfd(\vec\fy_0)$.
We have $\vec\psi_\ell(s_\ell) = \vec\psi(t_\ell')_{1/\nu(t_\ell)}$, thus $\bfd(\vec\psi(t_\ell')) = \bfd(\vec\psi_\ell(s_\ell))$
and we obtain
\EQ{
\lim_{\ell\to+\infty} \bfd(\vec\psi(t_\ell')) = \bfd(\vec\fy_0),
}
for any sequence $t_\ell' \in I_{m_\ell}$. This is impossible, because we know that $\bfd(\vec\psi(a_{m_\ell})) = \eps$
and on the other hand for each $\ell$ we have $\sup_{t\in I_{m_\ell}} \bfd(\vec\psi(t)) \geq 2\eps$.

Now suppose that there exist a sequence $m_\ell$ and times $t_\ell \in I_{m_\ell}$ such that
\EQ{
\lim_{\ell\to+\infty}\frac{\nu(t_\ell)}{\nu_{m_\ell}} = 0.
}
Without loss of generality we can assume that
\EQ{\label{eq:int-bd-part-2}
\lim_{\ell\to+\infty}\frac{\nu(t_\ell)}{a_{m_\ell} - t_\ell} = 0
}
(the case $\nu(t_\ell)/(t_\ell - b_{m_\ell-1}) \to 0$ can be treated similarly).

Again, let $\vec\psi_\ell(s)$ be the solution of \eqref{eq:wmk} with initial data
$\vec\psi_\ell(0) = \vec\psi(t_\ell)_{1/\nu(t_\ell)}$,
and let $\vec\psi_\ell(0) \to \vec\fy_0 \in \HH_0$.
Let $\vec\fy(s): (-T_-(\vec \fy_0), T_+(\vec \fy_0)) \to \HH_0$ be the solution of \eqref{eq:wmk} with initial data
$\vec\fy(0) = \vec\fy_0$.
By Lemma~\ref{l:1profile} we know that $\vec \fy(s)$ is non-scattering in both time directions and satisfies 
\EQ{
 \E( \vec \fy) = \E( \vec \psi) = 2 \E(\vec Q).
} 
Thus Proposition~\ref{p:cjk} implies that there exists $\sigma \in[0, T_+(\vec\fy_0))$ such that
$ \bfd(\vec\fy(\sigma)) \leq \frac 12 \eps'$.
By Cauchy theory, for $\ell$ large enough $\vec\psi_\ell(s)$ is defined for $s \in [0, \sigma]$
and $\vec\psi_\ell(\sigma) \to \vec\phi(\sigma)$ in $\HH_0$,
in particular $\bfd(\vec\psi_\ell(\sigma)) \to \bfd(\vec\phi(\sigma)) \leq \frac 12 \eps'$.

Let $t_\ell' := t_\ell + \nu(t_\ell)\sigma$. Then $\vec\psi(t_\ell') = \vec\psi_\ell(\sigma)_{\nu(t_\ell)}$, so we have
\EQ{
\lim_{\ell \to +\infty} \bfd(\vec\psi(t_\ell')) = \lim_{\ell\to+\infty} \bfd(\vec\psi_\ell(\sigma)) \leq \frac 12 \eps'.
}
However, \eqref{eq:int-bd-part-1} implies that for $\ell$ large enough there holds
$t_\ell \leq t_\ell' \leq a_{m_\ell}$, thus \eqref{eq:d-large-all} yields
$\bfd(\vec\psi(t_\ell')) \geq \eps'$. The contradiction finishes the proof.
\end{proof}

\begin{lem}\label{l:comp-vm}
There exists $\delta_1 > 0$ such that for all $m$ there holds
\EQ{
\int_{I_m} \|\partial_t \psi(t)\|_{L^2}^2\ud t \geq \delta_1^2 \nu_m.
}
\end{lem}
\begin{proof}
Let $t_m := \frac 12(b_{m-1} + a_m)$ and recall that $\nu_m := a_m - b_{m-1}$.  
Then for any $0<s_1 \le 1/2$, 
\EQ{
b_{m-1} \le  t_m - \nu_m s_1  \leq t_m + \nu_m s_1  \le   a_m. \label{eq:tm-dist}
}
We consider the following sequence of solutions of \eqref{eq:wmk}:
\EQ{
\vec \psi_m(s) := \vec\psi(t_m + \nu_m s)_{1/\nu_m}\qquad \text{for }s \in [-s_1, s_1].
}
Then \eqref{eq:tm-dist} implies that
\EQ{
\int_{I_m} \|\partial_t \psi(t)\|_{L^2}^2\ud t \geq \nu_m \int_{-s_1}^{s_1} \|\partial_s \psi_m(s)\|_{L^2}^2\ud s.
}
Suppose that the conclusion fails. Then there exists a sequence $m_1, m_2, \ldots$ such that
\EQ{
\lim_{l\to+\infty}\int_{-s_1}^{s_1}\|\partial_s \psi_{m_l}(s)\|_{L^2}^2\ud s = 0. \label{eq:dt-conv-0}
}
After extraction of a subsequence, $\vec \psi_{m_l}(0) \to \vec\fy_0 \in \HH_0$.
Let $\vec\fy(s)$ be the solution of \eqref{eq:wmk}
such that $\vec\fy(0) = \vec\fy_0$. Then by the standard Cauchy theory $\vec\psi_{m_l}(s) \to \vec\fy(s)$
in $\HH_0$, uniformly for $s \in [-s_1, s_1]$ for $s_1>0$ small enough.
In particular, \eqref{eq:dt-conv-0} yields
\EQ{
\int_{-s_1}^{s_1}\|\partial_s \fy(s)\|_{L^2}^2\ud s = 0,
}
so the limiting wave map $\vec \fy(s)$ must be time-independent.
Hence $\fy(s) \in H$ is a harmonic map. But then $\fy(s) \equiv 0$ since the constant map is the unique harmonic map with topological degree $0$. However, we also have $\E(\vec \fy) = 2 \E(\vec Q) >0$, which gives a contradiction. 
\end{proof}

\begin{lem}\label{l:err-on-I}
There exists $R_0 > 0$ such that if $R_1 \geq R_0$, then for all $m \in \{2, 3, \ldots\}$ there holds
\EQ{
\int_{I_m}\Omega_{\nu_m R_1}(\vec\psi(t))\ud t \leq \frac{\delta_1^2}{10}\nu_m.
}
\end{lem}
\begin{proof}
With a change of variables, it suffices to prove that for all $t \in I$ there holds
\EQ{
\Omega_{R_1}(\vec\psi(t)_{1/\nu_m}) \leq \frac{\delta_1^2}{10}.
}
This is a standard consequence of the pre-compactness of the set $\cK_1$ defined in \eqref{eq:K1-comp}.
\end{proof}

\subsection{Conclusions}
\begin{proof}[Proof of Proposition~\ref{p:psi_t}]
Choose $1\leq m_1 < m_2$ such that
\EQ{
\sqrt{\bfd(\vec\psi(c_{m_1}))} + \sqrt{\bfd(\vec\psi(c_{m_2}))} \leq \frac{\delta_1^2}{10C_0R_0}.
\label{eq:endpoints-small}
}
This is possible thanks to \eqref{eq:d-conv-0}.
Let
\EQ{
R := R_0 \max_{m_1 < m \leq m_2} \nu_m.
}

Inequalities \eqref{eq:err-mod-contr-1} and \eqref{eq:virial-err} yield
\EQ{
\begin{aligned}
\frac 15 \int_{c_{m_1}}^{c_{m_2}} \|\partial_t\psi(t)\|_{L^2}^2\ud t &\geq
\sum_{m = m_1+1}^{m_2}\frac 15\int_{b_{m-1}}^{a_m} \|\partial_t\psi(t)\|_{L^2}^2\ud t \\
&\geq \sum_{m=m_1+1}^{m_2}2C_0 \int_{a_m}^{c_m}\sqrt{\bfd(\vec\psi(t))}\ud t \\
&\geq 2\sum_{m=m_1+1}^{m_2}\int_{a_m}^{c_m}\Omega_R(\vec\psi(t))\ud t.
\end{aligned}
}
Similarly, \eqref{eq:err-mod-contr-2} and \eqref{eq:virial-err} yield
\EQ{
\begin{aligned}
\frac 15 \int_{c_{m_1}}^{c_{m_2}} \|\partial_t\psi(t)\|_{L^2}^2\ud t &\geq
\sum_{m = m_1}^{m_2-1}\frac 15\int_{b_{m}}^{a_{m+1}} \|\partial_t\psi(t)\|_{L^2}^2\ud t \\
&\geq \sum_{m=m_1}^{m_2-1}2C_0 \int_{c_m}^{b_m}\sqrt{\bfd(\vec\psi(t))}\ud t \\
&\geq 2\sum_{m=m_1}^{m_2-1}\int_{c_m}^{b_m}\Omega_R(\vec\psi(t))\ud t.
\end{aligned}
}
Next, from Lemma~\ref{l:comp-vm} we have
\EQ{ \label{eq:good-int}
\begin{aligned}
\frac 15 \int_{c_{m_1}}^{c_{m_2}} \|\partial_t\psi(t)\|_{L^2}^2\ud t &\geq
\sum_{m=m_1+1}^{m_2}\frac 15\int_{b_{m-1}}^{a_m}\|\partial_t\psi(t)\|_{L^2}^2\ud t
\geq \frac{\delta_1^2}{5}\sum_{m=m_1+1}^{m_2}\nu_m.
\end{aligned}
}
By the definition of $R$, for each $m \in \{m_1+1, m_1+2, \ldots, m_2\}$
we have $R = R_1\nu_m$ with $R_1 \geq R_0$. Thus Lemma~\ref{l:err-on-I} gives
\EQ{
\frac 15 \int_{c_{m_1}}^{c_{m_2}} \|\partial_t\psi(t)\|_{L^2}^2\ud t \geq 2\sum_{m=m_1+1}^{m_2}\int_{b_{m-1}}^{a_m}\Omega_R(\vec\psi(t))\ud t.
}
Finally, \eqref{eq:good-int} and \eqref{eq:endpoints-small} imply
\EQ{
\begin{aligned}
\frac 15 \int_{c_{m_1}}^{c_{m_2}} \|\partial_t\psi(t)\|_{L^2}^2\ud t &\geq \frac{\delta_1^2}{5}\max_{m_1 < m \leq m_2} \nu_m \\
&\geq 2C_0 R\Big(\sqrt{\bfd(\vec\psi(c_{m_1}))} + \sqrt{\bfd(\vec\psi(c_{m_2}))}\Big).
\end{aligned}
}
Summing the four inequalities above and using \eqref{eq:virR} for $(\tau_1, \tau_2) = (c_{m_1}, c_{m_2})$
we get
\EQ{
\frac 45 \int_{c_{m_1}}^{c_{m_2}} \|\partial_t\psi(t)\|_{L^2}^2\ud t \geq 2\int_{c_{m_1}}^{c_{m_2}} \|\partial_t\psi(t)\|_{L^2}^2\ud t.
}
This contradiction finishes the proof.
\end{proof}

\begin{proof}[Proof of the Main Theorem~\ref{t:main}]
~

\noindent
\textbf{Step 1.}
Let $\eps> 0$ be such that Proposition~\ref{prop:modulation} holds
with some $\eps_0 \geq 10\eps$. Define
\EQ{
T_1 := \sup\{t: \exists t' \geq t\text{ such that }\bfd(\vec\psi(t')) \geq \eps\}.
}
By Proposition \ref{p:psi_t} we have $T_1 < T_+$,
and we know that the modulation parameters $\lambda(t)$ and $\mu(t)$
are well-defined for $t \in [T_1, T_+)$. Define $\zeta(t)$ as in~\eqref{eq:zetadef}. 
Assume without loss of generality that $\mu(T_1) = 1$.

There exists a sequence $\tau_n \to T_+$ such that
\EQ{
\dd t\Big|_{t = \tau_n}\Big(\frac{\zeta(t)}{\mu(t)}\Big) \leq 0.
}
For any such $t_0 = \tau_n$ we are in the setting of Proposition~\ref{prop:modulation}
in the backward time direction, so we obtain times $t_1 \leq \tau_n$ and $t_2 \leq t_1$.
By the definition of $T_1$ and \eqref{eq:d-t1-t2} we have $t_1 \leq T_1$,
so the proof of Proposition~\ref{prop:modulation}
yields in particular $\frac 12 \leq \mu(t) \leq 2$ for $t \in [t_1, \tau_n]$,
thus $\frac 12 \leq \mu(t) \leq 2$ for $t \in [T_1, T_+)$.
Furthermore, \eqref{eq:uni-borne}
implies that $\int_{T_1}^{\tau_n}\sqrt{\bfd(\vec\psi(t))}\,\ud t$ is bounded as $\tau_n \to T_+$.
Thus \eqref{eq:mu'} implies that $\int_{T_1}^{T_+} |\mu'(t)|\ud t < +\infty$, hence $\mu(t)$ converges to some $\mu_0 \in [\frac 12, 2]$. Eventually rescaling again, we can assume that $\mu_0 = 1$.

As in the proof of Proposition~\ref{prop:modulation},
we consider $\xi(t) := b(t) + \kappa_1 \zeta(t)^\frac k2$ and we 
find that it is strictly decreasing on $[T_1, \tau_n]$ and satisfies
\EQ{
\xi'(t) \leq -\kappa_2 \xi(t)^\frac{2k-2}{k}. \label{eq:xi'-lbound}
}
Hence $\xi(t)$ is strictly decreasing on $[T_1, T_+)$ and satisfies \eqref{eq:xi'-lbound}
for $t \in [T_1, T_+)$. From the modulation equations we also obtain
\EQ{
\xi'(t) \geq -\kappa_3 \xi(t)^\frac{2k-2}{k}. \label{eq:xi'-ubound}
}
for some $\kappa_3$ depending only on $k$.
Indeed, \eqref{eq:la'} and \eqref{eq:b'} yield
\EQ{
|\xi'(t)| \lesssim |b'(t)| + |\zeta'(t)|\zeta(t)^\frac{k-2}{2} \lesssim \zeta(t)^{k-1} \lesssim \xi(t)^\frac{2k-2}{k},
}
Since $\lim_{t\to T_+} \xi(t) = 0$ and $\frac{2k-2}{k} \geq 1$,
\eqref{eq:xi'-ubound} implies that $T_+ = +\infty$.

Showing that the sign $\iota$ is constant is standard.
Lemma~\ref{l:dpm} implies that $\bfd_+(\vec\psi(t)) \leq \eps$
for $t \in [T_1, +\infty)$ or $\bfd_-(\vec\psi(t)) \leq \eps$ for $t \in [T_1, +\infty)$.
Indeed, suppose that $t_1, t_2 \geq T_1$, $t_1 \leq t_2$ are such that
$\bfd_+(\vec\psi(t_1)) \leq \eps$ and $\bfd_-(\vec\psi(t_2)) \leq \eps$.
Without loss of generality we can assume that $\eps < \frac 12 \alpha_0$,
where $\alpha_0$ is the constant from Lemma~\ref{l:dpm}.
Then Lemma~\ref{l:dpm} yields $\bfd_+(\vec\psi(t_2)) \geq \alpha_0$,
hence there exists $t_0 \in [t_1, t_2]$ such that $\bfd_+(\vec\psi(t_0)) = \frac 12 \alpha_0 > \eps$.
But Lemma~\ref{l:dpm} gives that also $\bfd_-(\vec(\psi(t_0)) \geq \alpha_0 > \eps$,
which contradicts the choice of $T_1$.

\noindent
\textbf{Step 2.}
We now deduce the rate of decay of $\lambda(t)$ as $t \to +\infty$.
Bounds \eqref{eq:psi-bound},  \eqref{eq:psi-lbound}, and~\eqref{eq:bound-on-l} imply that $\xi(t)$
is comparable to $\lambda(t)^\frac k2$.
Rewrite \eqref{eq:xi'-lbound} and \eqref{eq:xi'-ubound} as follows:
\EQ{
-\kappa_3 \leq \frac{\xi'(t)}{\xi(t)^\frac{2k-2}{k}} \leq -\kappa_2, \qquad \kappa_2, \kappa_3 > 0,
}
In the case $k = 2$, after integrating and possibly changing the values of the constants, we obtain
\EQ{
\mathrm e^{-\kappa_3 t} \leq \xi(t) \leq \mathrm e^{-\kappa_2 t},
}
which implies that there exists a constant $C$ such that
\EQ{\label{eq:lambda-asym-2}
\mathrm e^{-C t} \leq \lambda(t) \leq \mathrm e^{-\frac{1}{C} t}\quad \text{as }t \to +\infty.
}
(recall that we rescale the solution so that $\mu_0 = \lim_{t\to +\infty}\mu(t) = 1$).

Similarly, for $k > 2$ we obtain
\EQ{\label{eq:lambda-asym-k}
\frac 1C t^{-\frac{2}{k-2}} \leq \lambda(t) \leq C t^{-\frac{2}{k-2}}\quad \text{as }t \to +\infty,
}
with a constant depending on $k$.

\noindent
\textbf{Step 3.}
Suppose that $\vec\psi$ does not scatter in either time direction.
Take any $\delta > 0$.
Bounds \eqref{eq:lambda-asym-2} and \eqref{eq:lambda-asym-k} imply that
\EQ{\label{eq:d-integrable}
\int_{-\infty}^{+\infty}\sqrt{\bfd(\vec\psi(t))}\,\ud t < +\infty.
}
Indeed, it suffices to consider the behavior as $t \to \pm\infty$.
In this situation we have well-defined modulation parameters $\lambda(t)$, $\mu(t)$
and \eqref{eq:gH} together with the fact that $\mu(t)\to \mu_0 > 0$ imply that
$\bfd(\vec\psi(t)) \lesssim \lambda(t)^k$,
so \eqref{eq:d-integrable} follows from time integrability of $\lambda(t)^{\frac k2}$.

Thus \eqref{eq:virial-err} implies that there exist $T_1$, $T_2$ such that for all $R > 0$
\begin{align}
  \int_{-\infty}^{T_1} \Omega_R(\vec\psi(t))\,\ud t &\leq \frac 13\delta, \\
  \int_{T_2}^{+\infty} \Omega_R(\vec\psi(t))\,\ud t &\leq \frac 13\delta.
\end{align}
Since $[T_1, T_2]$ is a finite time interval, there exists $R > 0$ such that
\begin{equation}
  \int_{T_1}^{T_2}\Omega_R(\vec\psi(t))\,\ud t \leq \frac 13\delta.
\end{equation}
But $\bfd(\vec\psi(t)) \to 0$ as $t \to \pm\infty$, hence \eqref{eq:virR} yields
\begin{equation}
  \int_{-\infty}^{+\infty}\|\partial_t \psi(t)\|_{L^2}^2\,\ud t \leq \delta.
\end{equation}
This would imply that $\vec\psi(t)$ is a constant in time solution
(because $\delta$ was any strictly positive number), which is impossible.
\end{proof}
\begin{rem}
Suppose that $\vec\psi(t)$ does not scatter in the forward time direction.
From the modulation equations we know that
\EQ{
b'(t) \leq -\kappa_6 \lambda(t)^{k-1}, \qquad \kappa_6 > 0,
}
so integration yields
\EQ{
b(t) \geq \kappa_7 \int_t^\infty \lambda(s)^{k-1}\,\mathrm{d}s.
}
But, as noticed in Step 1. above,
we also have $|(\lambda(t)^\frac k2)'| \lesssim \lambda(t)^{k-1}$,
which yields $\lambda(t)^\frac k2 \lesssim \int_t^\infty \lambda(s)^{k-1}\,\mathrm{d}s$,
so we obtain
\EQ{
b(t) \geq \kappa_8 \lambda(t)^\frac k2 \quad \Rightarrow \quad b(t)^2 \geq \kappa_9 \bfd(\vec\psi(t)),
}
which in turn implies
\EQ{
\Big|\Big\langle\frac{1}{\lambda(t)}\Lambda Q_{\lambda(t)}, \partial_t \psi\Big\rangle\Big|^2 \geq c\,\bfd(\vec\psi(t)),\qquad \text{as }t \to +\infty,
}
where $c > 0$ is a constant depending only on $k$.
Thus the projection of the time derivative of the solution constitutes
at least a fixed fraction of the total distance from a two-bubble.
In fact, if we were more precise in our computations, we could probably obtain that
this projection is the leading term of the error.
\end{rem}

\bibliographystyle{plain}
\bibliography{researchbib}
\bigskip
\centerline{\scshape Jacek Jendrej}
\smallskip
{\footnotesize
  \centerline{CNRS and Universit\'e Paris 13, LAGA, UMR 7539}
  \centerline{99 av J.-B.~Cl\'ement, 93430 Villetaneuse, France}
  \centerline{\email{jendrej@math.univ-paris13.fr}}
}

\bigskip

\centerline{\scshape Andrew Lawrie}
\smallskip
{\footnotesize
 \centerline{Department of Mathematics, Massachusetts Institute of Technology}
\centerline{77 Massachusetts Ave, 2-267, Cambriedge, MA 02139, U.S.A.}
\centerline{\email{ alawrie@mit.edu}}
}

\end{document}